\documentclass[11pt, oneside]{report}

\usepackage{fullpage}
\usepackage[hang,flushmargin,symbol*]{footmisc}
\usepackage{amsmath}
\usepackage{amsthm}
\usepackage{amssymb}
\usepackage{mathtools}
\usepackage{enumitem}
\usepackage{graphicx}
\usepackage{color}
\definecolor{darkblue}{rgb}{0, 0, .6}
\definecolor{grey}{rgb}{.7, .7, .7}
\usepackage[breaklinks]{hyperref}
\hypersetup{
	colorlinks=true,
	linkcolor=darkblue,
	anchorcolor=darkblue,
	citecolor=darkblue,
	pagecolor=darkblue,
	urlcolor=darkblue,
	pdftitle={},
	pdfauthor={}
}

\newtheorem{theorem}{Theorem}[section]
\newtheorem{lemma}[theorem]{Lemma}

\newtheorem{corollary}[theorem]{Corollary}
\newtheorem{proposition}[theorem]{Proposition}

\theoremstyle{definition} 
\newtheorem{definition}[theorem]{Definition}
\newtheorem{example}[theorem]{Example}
\newtheorem{remark}[theorem]{Remark}

\newcommand{\Z}{\mathbb{Z}}
\newcommand{\N}{\mathbb{N}}
\newcommand{\A}{\mathcal{A}}
\newcommand{\C}{\widetilde{C}}
\renewcommand{\O}{\mathcal{O}}
\newcommand{\E}{\mathcal{E}}
\renewcommand{\S}{\mathcal{S}}
\newcommand{\z}{\mathsf{z}}
\newcommand{\x}{\mathsf{x}}
\newcommand{\y}{\mathsf{y}}
\renewcommand{\u}{\mathsf{u}}
\renewcommand{\v}{\mathsf{v}}
\newcommand{\wtri}{\vartriangle}
\newcommand{\btri}{\blacktriangle}
\renewcommand{\a}{\mathbf{a}}

\newcommand{\TL}{\mathrm{TL}}
\newcommand{\DTL}{\mathbb{D}\mathrm{TL}}
\renewcommand{\P}{\mathcal{P}}
\newcommand{\V}{\mathcal{V}}
\newcommand{\D}{\mathbb{D}}
\newcommand{\B}{\mathbf{B}}
\newcommand{\I}{\mathcal{I}}
\newcommand{\Diag}{\mathcal{D}}
\newcommand{\Irr}{\mathrm{Irr}}
\newcommand{\wcirc}{\circ}

\newcommand{\bcirc}{\bullet}

\newcommand{\supp}{\mathrm{supp}}

\renewcommand{\L}{\mathcal{L}}
\newcommand{\R}{\mathcal{R}}
\renewcommand{\(}{\left(}
\renewcommand{\)}{\right)}

\newcommand{\w}{\mathsf{w}}
\renewcommand{\H}{\mathcal{H}}
\renewcommand{\r}{\mathbf{r}}
\renewcommand{\emph}{\textbf}

\makeatletter 
\def\iddots{\mathinner{\mkern1mu\raise\p@ 
\vbox{\kern7\p@\hbox{.}}\mkern2mu 
\raise4\p@\hbox{.}\mkern2mu\raise7\p@\hbox{.}\mkern1mu}} 
\makeatother 
 
\begin{document}

\setlength{\parindent}{0pt}

\begin{titlepage}
\ 
\vspace{5cm}

{\huge \textbf{A diagrammatic representation of an affine $C$\\
Temperley--Lieb algebra\footnote{This is a revised version of the author's Ph.D. thesis, which was directed by Richard M. Green at the University of Colorado at Boulder.  See \hyperref[comments]{Comments} page for a complete list of revisions.}}}

\bigskip

Ph.D. Thesis, University of Colorado at Boulder, 2008

\bigskip

\bigskip

\href{http://oz.plymouth.edu/~dcernst}{\Large Dana C. Ernst}\\
Plymouth State University\\
Department of Mathematics\\
MSC 29, 17 High Street\\
Plymouth, NH 03264\\
\href{mailto:dcernst@plymouth.edu}{{dcernst@plymouth.edu}}

\end{titlepage}

\pagenumbering{roman}
\pagestyle{plain}

%%%%%%%%%Abstract%%%%%%%%%%%

\chapter*{Abstract}\normalsize
\addcontentsline{toc}{chapter}{Abstract}

In this thesis, I present an associative diagram algebra that is a faithful representation of a particular Temperley--Lieb algebra of type affine $C$, which has a basis indexed by the fully commutative elements of the Coxeter group of the same type.  The Coxeter group of type affine $C$ contains an infinite number of fully commutative elements, and so the corresponding Temperley--Lieb algebra is of infinite rank. With the exception of type affine $A$, all other generalized Temperley--Lieb algebras with known diagrammatic representations are of finite rank.  In the finite rank case, counting arguments are employed to prove faithfulness, but these techniques are not available in the affine $C$ case.  To prove faithfulness, I classify the fully commutative elements in Coxeter groups of types $B$ and affine $C$ that are irreducible under weak star reductions.  The classification of these irreducible elements provides the groundwork for inductive arguments that are used to prove faithfulness.  The classification of the weak star irreducible elements of type $B$ also verifies C.K. Fan's unproved claim about about the set of fully commutative elements in a Coxeter group of type $B$ having no generator appearing in the left or right descent set that can be left or right cancelled, respectively.  The results of this thesis will be used to construct a trace on the Hecke algebra of type affine $C$, which will then be used to compute leading coefficients of certain Kazhdan--Lusztig polynomials in a non-recursive way.

%%%%%%%%%Comments%%%%%%%%%%%

\chapter*{Comments}\label{comments}\normalsize
\addcontentsline{toc}{chapter}{Comments}

This is a revised version of the author's Ph.D. thesis, which was directed by Richard M. Green at the University of Colorado at Boulder.  The numbering of definitions, theorems, remarks, and examples is identical in both versions.  Here is a complete list of revisions:

\begin{enumerate}[label=\rm{(\arabic*)}]

\item This version was typeset using the \texttt{report} document class instead of the University of Colorado \texttt{thesis} document class.  As a result, there were some modifications in formatting, all of which were cosmetic.

\item This version was typeset using the \texttt{hyperref} package which has provided hyperlinks within the document.

\item Entries in the bibliography have been updated.

\item The last paragraph before Definition \ref{decorated.pseudo.diagram.def} has been modified to emphasize that all diagrams with $\a$-value 0 (including those with loops) are undecorated.

\item Fixed trivial typo in Example \ref{first.dec.diagram.ex}\ref{first.dec.diagram.ex.diagram2} (deleted extra ``is'').

\item The original statement of Remark \ref{LR-decorated}\ref{undammed.dec.loops} was wrong, but the claim had no bearing on subsequent statements.  I added the phrase ``by both types of decorations'' and now the statement is correct.

\end{enumerate}

%%%%%%%%%Dedication%%%%%%%%%%%

\chapter*{Dedication}\normalsize
\addcontentsline{toc}{chapter}{Dedication}

This thesis is dedicated to my older brother Brandt C. Ernst.  He is, and will always be, my inspiration. 

%%%%%%%%%Acknowledgements%%%%%%%%%%%

\chapter*{Acknowledgements}\normalsize
\addcontentsline{toc}{chapter}{Acknowledgements}

I am especially grateful to my advisor Richard M. Green for patiently guiding me on this journey and teaching me how to talk, write, and explore mathematics.  I am also grateful to my second reader  Nathaniel Thiem for taking the time to read my thesis and provide me with useful comments.  In addition, I would like to thank Chanyoung Lee Shader, Arlan Ramsay, and Marty Walter for serving on my committee.  For providing useful tips on using {\tt mfpic} (which I used to draw all of my figures), I thank Daniel Luecking.  I would also like to thank my office mate Jonas D'Andrea for being my sounding board and encouraging me along the way.  Lastly, I would like to thank my wife Jen and my two sons, Tristan and Kaden, for being so supportive and patient with me.  Without them this would not have been possible.

\tableofcontents

%%%%%%%  chapter 1 %%%%%%%%%%%%%

\begin{chapter}{Coxeter groups and fully commutative elements}

\pagenumbering{arabic}

In this first chapter, we will set up our notation and review some of the necessary background material.  For a reader unfamiliar with Coxeter groups, we recommend either the classic text by Humphreys \cite{Humphreys.J:A} or the recent text by Bj\"orner and Brenti \cite{Bjorner.A;Brenti.F:A} for a more combinatorial treatment.

\begin{section}{Coxeter groups}

A \emph{Coxeter group} is a group $W$ with together with a distinguished set of generating involutions $S$ subject to relations of the form $(st)^{m(s,t)}=1$, where $m(s, s) = 1$ and $m(s, t) = m(t, s)$.  In other words, $W$ is given by the presentation 
	$$W = \langle S :(st)^{m(s, t)} = 1 \text{ for } m(s, t) < \infty \rangle,$$ 
where $m(s, s) = 1$ and $m(s, t) = m(t, s)$.  It turns out that the elements of $S$ are distinct as group elements, and that $m(s, t)$ is the order of $st$.  We will only be interested in Coxeter groups in which the generating set $S$ is finite.  We call the pair $(W,S)$ a \emph{Coxeter system}.  Given a Coxeter system $(W,S)$, the associated \emph{Coxeter graph} is the graph $X$ on the generating set $S$ with edges connecting $s_{i}$ and $s_{j}$, labeled $m(s_{i},s_{j})$, for all pairs $i,j$ with $m(s_{i},s_{j})>2$.  If $m(s_{i},s_{j})=3$, it is customary to leave the corresponding edge unlabeled.  Given a Coxeter graph $X$, we can uniquely reconstruct the corresponding Coxeter system $(W,S)$.  When we have a particular Coxeter graph $X$ in mind, we will denote the underlying Coxeter group and distinguished generating set by $W(X)$ and $S(X)$, respectively.

\bigskip

\begin{example}
In this example, we introduce three specific Coxeter groups that will turn up frequently throughout this thesis.
\begin{itemize}
\item[($A$)] The Coxeter graph of type $A_{n}$ ($n\geq 1$) is as follows.
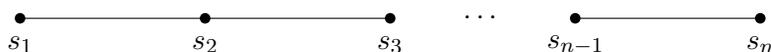
\begin{center}
%-- New mfpic environment, number 1 of 189. (size of end split: 2, should be 2)  ------------------->
\includegraphics{ThesisFigs1.001}
\end{center}

Then $W(A_{n})$ is generated by $S(A_{n})=\{s_{1}, s_{2}, \dots, s_{n}\}$ and is subject only to the relations
\begin{enumerate}[label=\rm{(\arabic*)}]
\item $s_{i}^{2}=1$ for all $i$;
\item $s_{i}s_{j}=s_{j}s_{i}$ if $|i-j|>1$;
\item $s_{i}s_{j}s_{i}=s_{j}s_{i}s_{j}$ if $|i-j|=1$.
\end{enumerate}
It is well-known (see \cite[Chapter 1]{Humphreys.J:A})  that $W(A_{n})$ is isomorphic to the symmetric group, $S_{n+1}$, under the correspondence 
	$$s_{i}\mapsto (i\ i+1),$$
where $(i\ i+1)$ is the adjacent transposition exchanging $i$ and $i+1$.

\bigskip

\item[($B$)] The Coxeter graph of type $B_{n}$ ($n\geq 2$) is as follows.
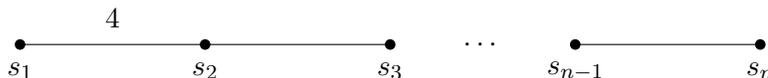
\begin{center}
%-- New mfpic environment, number 2 of 189. (size of end split: 2, should be 2)  ------------------->
\includegraphics{ThesisFigs1.002}
\end{center}

In this case, $W(B_{n})$ is generated by $S(B_{n})=\{s_{1}, s_{2}, \dots, s_{n}\}$ and is subject only to the relations
\begin{enumerate}[label=\rm{(\arabic*)}]
\item $s_{i}^{2}=1$ for all $i$;
\item $s_{i}s_{j}=s_{j}s_{i}$ if $|i-j|>1$;
\item $s_{i}s_{j}s_{i}=s_{j}s_{i}s_{j}$ if $|i-j|=1$ and $1<i,j\leq n$;
\item $s_{1}s_{2}s_{1}s_{2}=s_{2}s_{1}s_{2}s_{1}$.
\end{enumerate}

\item[($\C$)] The Coxeter graph of type $\C_{n}$ ($n \geq 2$), pronounced ``affine $C_{n}$'', is as follows.
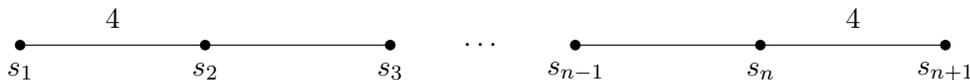
\begin{center}
%-- New mfpic environment, number 3 of 189. (size of end split: 2, should be 2)  ------------------->
\includegraphics{ThesisFigs1.003}
\end{center}

Here, we see that $W(\C_{n})$ is generated by $S(\C_{n})=\{s_{1}, s_{2}, \dots, s_{n}, s_{n+1}\}$ and is subject only to the relations
\begin{enumerate}[label=\rm{(\arabic*)}]
\item $s_{i}^{2}=1$ for all $i$;
\item $s_{i}s_{j}=s_{j}s_{i}$ if $|i-j|>1$;
\item $s_{i}s_{j}s_{i}=s_{j}s_{i}s_{j}$ if $|i-j|=1$ and $1< i,j < n+1$;
\item $s_{i}s_{j}s_{i}s_{j}=s_{j}s_{i}s_{j}s_{i}$ if $\{i,j\}=\{1,2\}$ or $\{n,n+1\}$.
\end{enumerate}
\end{itemize}
\end{example}

Let $X$ be an arbitrary Coxeter graph.  An \emph{expression} is any product of generators from $S(X)$.  The \emph{length} $l(w)$ of an element $w \in W(X)$ is the minimum number of generators appearing in any expression for the element $w$.  Such a minimum length expression is called a \emph{reduced expression}.  (Any two reduced expressions for $w \in W(X)$ have the same length.)  A product $w_{1}w_{2}\cdots w_{r}$ with $w_{i} \in W(X)$ is called \emph{reduced} if $l(w_{1}w_{2}\cdots w_{r})=\sum l(w_{i})$.  Each element $w \in W(X)$ can have several different reduced expressions that represent it.  Given $w \in W(X)$, if we wish to emphasize a fixed, possibly reduced, expression for $w$, we represent it in \textsf{sans serif} font, say $\w=\w_{1} \w_{2}\cdots \w_{r}$, where each $\w_{i} \in S(X)$.  Note that if we write $w_{i}$ without the sans serif font, then we are referring to a group element that is not necessarily a generator.  The context should also clarify any potential confusion.

\bigskip

\begin{example}
Let $w \in W(B_{3})$ with expression $\w=s_{1}s_{2}s_{1}s_{2}s_{3}s_{1}$.  Since $s_{1}s_{2}s_{1}s_{2}=s_{2}s_{1}s_{2}s_{1}$, $s_{1}s_{3}=s_{3}s_{1}$, and $s_{1}^{2}=1$ in $W(B_{3})$, we see that
	$$s_{1}s_{2}s_{1}s_{2}s_{3}s_{1}=s_{2}s_{1}s_{2}s_{1}s_{1}s_{3}=s_{2}s_{1}s_{2}s_{3}.$$
This shows that $\w$ is not reduced.  However, it is true (but not immediately obvious) that $s_{2}s_{1}s_{2}s_{3}$ is a reduced expression for $w$, so that $l(w)=4$.
\end{example}

The following theorem is a well-known result, attributed variously to Matsumoto \cite[Theorem 1.2.2]{Geck.M;Pfeiffer.G:A}, \cite{Matsumoto.H:A} and Tits \cite{Tits.J:A}.

\begin{theorem}[Matsumoto's Theorem]\label{Tits}
Let $X$ be an arbitrary Coxeter graph and let $w \in W(X)$.  Then every reduced expression for $w$ can be obtained from any other by applying a sequence of \emph{braid moves} of the form 
	$${\underbrace{s_i s_j s_i s_j \cdots }_{m(s_{i},s_{j})} } \mapsto {\underbrace{s_j s_i s_j s_i \cdots}_{m(s_{i},s_{j})}}$$
where $s_i, s_j \in S(X)$, and each factor in the move has $m(s_{i},s_{j})$ letters.  \hfill $\qed$
\end{theorem}

Let the \emph{support} of an element $w \in W(X)$, denoted $\supp(w)$, be the set of all generators appearing in any reduced expression for $w$, which is well-defined by Matsumoto's Theorem.

\bigskip

Given a reduced expression $\w=\w_{1}\w_{2}\cdots \w_{r}$ for $w \in W(X)$, we define a \emph{subexpression} of $\w$ to be any expression obtained by deleting some subsequence of generators in the expression for $\w$.  If $x \in W(X)$ has an expression that is equal to a subexpression of $\w$, then we write $x \leq w$.  This is a well-defined partial order \cite[Chapter 5]{Humphreys.J:A} on $W(X)$ and is called the \emph{(strong) Bruhat order}.  We will refer to a consecutive subexpression of $\w$ as a \emph{subword}.

\begin{example}
Let $w, u, v \in W(\C_{3})$ have reduced expressions $\w=s_{1}s_{3}s_{2}s_{4}s_{1}s_{3}$, $\u=s_{1}s_{2}s_{1}$, and $\v=s_{3}s_{2}s_{4}$, respectively.  Then $\u$ and $\v$ are both subexpressions for $\w$, while only $\v$ is a subword of $\w$.  Then we have $u, v \leq w$.
\end{example}

Let $w \in W(X)$.  We write
	$$\L(w)=\{s \in S(X): l(sw) < l(w)\}$$
and
	$$\R(w)=\{s \in S(X): l(ws) < l(w)\}.$$
The set $\L(w)$ (respectively, $\R(w)$) is called the \emph{left} (respectively, \emph{right}) \emph{descent set} of $w$.  It turns out that $s \in \L(w)$ (respectively, $\R(w)$) if and only if $w$ has a reduced expression beginning (respectively, ending) with $s$.

\begin{example}
Let $w \in W(B_{3})$ have reduced expression $\w=s_{1}s_{3}s_{2}s_{1}$.  Since $s_{1}$ and $s_{3}$ commute, but $s_{2}$ commutes with neither $s_{1}$ nor $s_{3}$, it follows from Matsumoto's Theorem (Theorem \ref{Tits}) that $\L(w)=\{s_{1}, s_{3}\}$ and $\R(w)=\{ s_{1}\}$.
\end{example}

It is known to be true (see \cite[Chapter 5]{Humphreys.J:A}) that we can obtain $W(B_{n})$ from $W(\C_{n})$ by removing the generator $s_{n+1}$ and the corresponding relations.  We also obtain a Coxeter group of type $B$ if we remove the generator $s_{1}$ and the corresponding relations.  To distinguish these two cases, we let $W(B_{n})$ denote the subgroup of $W(\C_{n})$ generated by $S(\C_{n})\setminus \{s_{n+1}\}=\{s_{1}, s_{2}, \dots, s_{n}\}$ and we let $W(B'_{n})$ denote the subgroup of $W(\C_{n})$ generated by $S(\C_{n})\setminus \{s_{1}\}= \{s_{2}, s_{3}, \dots, s_{n+1}\}$.  

\bigskip

It is well-known that $W(\C_{n})$ is an infinite Coxeter group while $W(B_{n})$ and $W(B'_{n})$ are both finite.  For a proof of this fact, see \cite[Chapters 2 and 6]{Humphreys.J:A}.

\end{section}

\begin{section}{Fully commutative elements of Coxeter groups}

Let $X$ be an arbitrary Coxeter graph and let $w \in W(X)$.  Following Stembridge \cite{Stembridge.J:B}, we define a relation $\sim$ on the set of reduced expressions for $w$.  Let $\w$ and $\w'$ be two reduced expressions for $w$.  We define $\w \sim \w'$ if we can obtain $\w'$ from $\w$ by applying a single \emph{commutation move} of the form $s_i s_j \mapsto s_j s_i$, where $m(s_{i},s_{j}) = 2$.  Now, define the equivalence relation $\approx$ by taking the reflexive transitive closure of $\sim$.  We refer to each equivalence class under $\approx$ as a \emph{commutation class}.  Note that $\w \approx \w'$ if and only if we can obtain $\w'$ from $\w$ by a sequence of commutation moves.  If $w$ has a single commutation class, then we say that $w$ is \emph{fully commutative}.    According to \cite[Proposition 2.1]{Stembridge.J:B}, an element $w$ is fully commutative if and only if no reduced expression for $w$ contains a subword of the form $s_i s_j s_i s_j \cdots$ of length $m(s_{i},s_{j}) \geq 3$.  This also follows from Matsumoto's Theorem (Theorem \ref{Tits}).  We will denote the set of all fully commutative elements of $W(X)$ by $W_{c}(X)$. 

\begin{remark}\label{illegal.convex.chains}
The fully commutative elements of $W(\C_{n})$ are precisely those such that all reduced expressions avoid subwords of the following types:
\begin{enumerate}[label=\rm{(\arabic*)}]
\item[(1)] $s_{i}s_{j}s_{i}$ for $|i-j|=1$ and $1< i,j < n+1$;
\item[(2)] $s_{i}s_{j}s_{i}s_{j}$ for $\{i,j\}=\{1,2\}$ or $\{n,n+1\}$.
\end{enumerate}
Note that the fully commutative elements of $W(B_{n})$ and $W(B'_{n})$ avoid the respective subwords above.
\end{remark}

\begin{example}
Let $w, w' \in W(\C_{3})$ have reduced expressions $\w=s_{1}s_{3}s_{2}s_{1}s_{2}$ and $\w'=s_{1}s_{2}s_{1}s_{3}s_{2}$, respectively.  Since $s_{1}$ and $s_{3}$ commute, we can write
	$$w=s_{1}s_{3}s_{2}s_{1}s_{2}=s_{3}s_{1}s_{2}s_{1}s_{2}.$$
This shows that $w$ has a reduced expression containing $s_{1}s_{2}s_{1}s_{2}$ as a subword, which implies that $w$ is not fully commutative.  On the other hand, we will never be able to rewrite $w'$ to produce an illegal subword, since the only relation we can apply is $s_{1}s_{3} \mapsto s_{3}s_{1}$ and this does not provide an opportunity to apply any additional relations.  So, $w'$ is fully commutative. 
\end{example}

In \cite{Stembridge.J:B}, Stembridge classified the Coxeter groups that contain a finite number of fully commutative elements.  According to \cite[Theorem 5.1]{Stembridge.J:B}, $W(\C_{n})$ contains an infinite number of fully commutative elements.  Since $W(B_{n})$ is a finite group, $W(B_{n})$, and hence $W(B'_{n})$, contains only finitely many fully commutative elements.  There are examples of infinite Coxeter groups that contain a finite number of fully commutative elements. 

\end{section}

\begin{section}{Star and weak star reductions}

The notion of star operation was originally defined by Kazhdan and Lusztig in \cite[\textsection 4.1]{Kazhdan.D;Lusztig.G:A} for simply laced Coxeter groups (i.e., $m(s,t)=3$ for all $s$ and $t$ that are adjacent in the Coxeter graph) and was later generalized to arbitrary Coxeter groups in \cite[\textsection 10.2]{Lusztig.G:A}.  If $I=\{s,t\}$ is a pair of noncommuting generators for $W(X)$, then $I$ induces four partially defined maps from $W(X)$ to itself, known as star operations. A star operation, when it is defined, respects the partition $W(X) = W_{c}(X)\ \dot{\cup}\  (W(X) \setminus W_{c}(X) )$ of the Coxeter group, and increases or decreases the length of the element to which it is applied by 1.  For our purposes, it is enough to define star operations that decrease length by 1, and so we will not develop the full generality.
 
\bigskip

Let $X$ be an arbitrary Coxeter graph and let $w \in W(X)$.  Suppose that $s \in \L(w)$.  Then we define  $w$ to be \emph{left star reducible} by $s$ to $sw$ if there exists $t \in \L(sw)$ with $m(s,t) \geq 3$.  In this case, we say that $w$ is left star reducible by $s$ with respect to $t$, and we define
	$$\bigstar_{s,t}^{L}(w)=sw,$$
and refer to $\bigstar_{s,t}^{L}$ as a \emph{left star reduction}.  We analogously define \emph{right star reducible} and \emph{right star reduction}.  If $w$ is right star reducible by $s$ with respect to $t$, then we define
	$$\bigstar_{s,t}^{R}(w)=ws.$$
If $w$ is not left (respectively, right) star reducible by $s$ with respect to $t$, then $\bigstar_{s,t}^{L}(w)$ (respectively, $\bigstar_{s,t}^{R}(w)$) is undefined.  Observe that if $m(s,t)\geq 3$, then $w$ is left (respectively, right) star reducible by $s$ with respect to $t$ if and only if $w=stv$ (respectively, $w=vts$), where the product is reduced.

\begin{example}
Let $w \in W(B_{n})$ (for any $n\geq 2$) have reduced expression $\w=s_{1}s_{2}$.  Then
	$$\bigstar_{s_{1},s_{2}}^{L}(w)=s_{1}s_{1}s_{2}=s_{2} \text{\quad and\quad} \bigstar_{s_{2},s_{1}}^{R}(w) = s_{1}s_{2}s_{2}=s_{1}.$$
If $u \in W(B_{n})$ has reduced expression $\u=s_{1}s_{3}$, then $\bigstar_{s,t}^{L}(u)$ and $\bigstar_{s,t}^{R}(u)$ are undefined for any noncommuting $s$ and $t$.  
\end{example}

If there is a (possibly trivial) sequence 
	$$u = w_0, w_1, \ldots, w_k = w,$$ 
where, for each $0 \leq i < k$, $w_{i+1}$ is left or right star reducible to $w_i$, we say that $w$ is \emph{star reducible to $u$}.  If $w$ is star reducible to some $u \neq w$, then we say that $w$ is \emph{star reducible}.  

\bigskip

We say that a Coxeter group $W(X)$, or its Coxeter graph $X$, is \emph{star reducible} if every element of $W_c(X)$ is star reducible to a product of commuting generators from $S(X)$.  In \cite{Green.R:P}, Green classified the star reducible Coxeter groups.  It turns out that $W(\C_{n})$ is star reducible if and only if $n$ is odd (recall that $W(\C_{n})$ has $n+1$ generators).  However, both $W(B_{n})$ and $W(B'_{n})$ are star reducible, regardless of the parity of $n$ (see \cite[Theorem 6.3]{Green.R:P}).

\begin{example}
Let $w \in W_{c}(\C_{4})$ have reduced expression $\w=s_{1}s_{3}s_{5}s_{2}s_{4}s_{1}s_{3}s_{5}$.  Then $\bigstar_{s,t}^{L}(w)$ and $\bigstar_{s,t}^{R}(w)$ are undefined for all pairs of noncommuting generators $s$ and $t$.  To see that this is true, first observe that $\L(w)=\{s_{1}, s_{3}, s_{5}\}=\R(w)$.  Then for $s \in \{s_{1}, s_{3}, s_{5}\}$, we have $l(sw)=l(w)-1=l(ws)$, where we obtain $sw$ (respectively, $ws$) by removing the leftmost (respectively, rightmost) occurrence of $s$ in $\w$.  Then we see that $\L(sw)$ (respectively, $\R(ws)$) does not contain a generator not commuting with $s$.  So, $w$ is neither left nor right star reducible to a product of commuting generators.  Therefore, $W(\C_{4})$ is not star reducible.  We can construct similar examples for $W(\C_{n})$ for any even $n$.
\end{example}

We now introduce the concept of weak star reducible, which is related to Fan's notion of cancellable in \cite{Fan.C:A}.  Similar to ordinary star reductions, if $I=\{s,t\}$ is a pair of noncommuting generators for $W(X)$, then $I$ induces four partially defined maps, called weak star reductions, from $W_{c}(X)$ to itself.  (Note our restriction to the set of fully commutative elements.)

\begin{definition}
Let $X$ be a Coxeter graph and let $w \in W_{c}(X)$.  Then $w$ is \emph{left weak star reducible} by $s$ with respect to $t$ to $sw$ if 
\begin{enumerate}[label=\rm{(\arabic*)}]
\item  $w$ is left star reducible by $s$ with respect to $t$; 
\end{enumerate}
and
\begin{enumerate}[label=\rm{(\arabic*)}, resume]
\item $tw \notin W_{c}(X)$.
\end{enumerate}
Observe that (1) implies that $m(s,t) \geq 3$ and that $s \in \L(w)$.  Furthermore, (2) implies that $l(tw)>l(w)$.  We analogously define \emph{right weak star reducible}.  If $w$ is left weak star reducible by $s$ with respect to $t$, then we define
	$$\star^{L}_{s,t}(w)=sw$$
and refer to $\star_{s,t}^{L}$ as a \emph{left weak star reduction}.  Similarly, if $w$ is right weak star reducible by $s$ with respect to $t$, we define
	$$\star^{R}_{s,t}(w)=ws$$
and refer to $\star_{s,t}^{R}$ as a \emph{right weak star reduction}.  If there is a (possibly trivial) sequence 
	$$u = w_0, w_1, \ldots, w_k = w,$$ 
where, for each $0 \leq i < k$, $w_{i+1}$ is left weak star reducible or right weak star reducible to $w_i$, we say that $w$ is \emph{weak star reducible to $u$}.  If $w$ is weak star reducible to some $u \neq w$, then we say that $w$ is \emph{weak star reducible}.  If $w$ is not weak star reducible, then we say that $w$ is \emph{weak star irreducible}, or simply \emph{irreducible}.
\end{definition}

\begin{example}
Let $w,w' \in W_{c}(\C_{n})$ have reduced expressions $\w=s_{1}s_{2}$ and $\w'=s_{1}s_{2}s_{1}$, respectively.  We see that $w'$ is left (and right) weak star reducible by $s_{2}$ with respect to $s_{1}$, and so $w'$ is not irreducible.  However, $w$ is irreducible.    
\end{example}

\begin{remark}\label{weak=ordinary}
We make several comments about weak star reducibility.
\begin{enumerate}[label=\rm{(\arabic*)}]
\item Note the similarity in notation for star reductions versus weak star reductions.  We use $\bigstar$ for ordinary star reductions and $\star$ for weak star reductions. There should be no confusion because we will only use weak star reductions in the remainder of this thesis.  

\item Observe that we restrict the definition of weak star reducible to fully commutative elements.  It follows from the subword characterization of full commutativity \cite[Proposition 2.1]{Stembridge.J:B} that if $w \in W_{c}(X)$ is left or right weak star reducible to $u$, then $u$ is also fully commutative.  

\item As the terminology suggests, if $w$ is weak star reducible to $u$, then $w$ is also star reducible to $u$.  However, there are examples of fully commutative elements that are star reducible, but not weak star reducible.  For example, consider $w=s_{1}s_{2} \in W_{c}(B_{2})$.  We see that $w$ is star reducible, but not weak star reducible since $tw$ and $wt$ are still fully commutative for any $t \in S(X)$.  Moreover, observe that if $m(s,t)=3$, then the definition of a weak star reduction agrees with the definition of a star reduction.  In particular, weak star reductions are equivalent to ordinary star reductions in a Coxeter group of type $A$.

\end{enumerate}
\end{remark}

If $w \in W_{c}(\C_{n})$, then $w$ is left weak star reducible by $s$ with respect to $t$ if and only if $w=stv$ (reduced) when $m(s,t)=3$, or $w=stsv$ (reduced) when $m(s,t)=4$.  In this case, we have
	$$\star^{L}_{s,t}(w)=sw=\begin{cases}
  	 tv,   & \text{if } m(s,t)=3, \\
  	 tsv,   & \text{if } m(s,t)=4.
	\end{cases}$$
Again, observe that the characterization above applies to $W_{c}(B_{n})$ and $W_{c}(B'_{n})$.  In Chapter \ref{chaptBwsrm}, we will classify the irreducible elements of $W(B_{n})$ and $W(B'_{n})$, which verifies Fan's unproved claim in \cite[\textsection 7.1]{Fan.C:A} about the set of $w \in W_{c}(B_{n})$ having no element of $\L(w)$ or $\R(w)$ that can be left or right cancelled, respectively.  We will then use the classification of the type $B$ and $B'$ irreducible elements to classify the irreducible elements in $W_{c}(\C_{n})$ (see Chapter \ref{chaptCwsrm}).

\end{section}

\end{chapter}

%%%%%%%%%% Chapter 2 %%%%%%%%%%%

\begin{chapter}{Heaps}

Every reduced expression can be associated with a partially ordered set called a heap that we define below.  This partially ordered set allows us to visualize a reduced expression as a set of lattice points while preserving the necessary information about the relations among the generators.  Cartier and Foata \cite{Cartier.P;Foata.D:A} were among the first to study heaps of dimers, and these were generalized to other settings by Viennot \cite{Viennot.G:A}.  Later, Stembridge studied enumerative aspects of heaps \cite{Stembridge.J:B,Stembridge.J:A} in the context of fully commutative elements, which is our motivation here.  In this chapter, we mimic the development found in \cite{Billey.S;Jones.B:A}, \cite{Jones.B:B}, and \cite{Stembridge.J:B}.  

\begin{section}{Heaps for elements of Coxeter groups}

Let $X$ be a Coxeter graph.  Suppose $\w = \w_1 \cdots \w_r$ is a fixed reduced expression for $w \in W(X)$.  As in \cite{Stembridge.J:B}, we define a partial ordering on the indices $\{1, \dots, r\}$ by the transitive closure of the relation $\lessdot$ defined via $j \lessdot i$ if $i < j$ and $\w_{i}$ and $\w_{j}$ do not commute.  In particular, $j \lessdot i$ if $i < j$ and $\w_i = \w_j$ (since we took the transitive closure).  It follows from \cite[Proposition 2.2]{Stembridge.J:B} that if $\w$ and $\w'$ are two reduced expressions for $w \in W(X)$ that are in the same commutativity class, then the labeled posets of $\w$ and $\w'$ are isomorphic, where $i$ is labeled by $\w_{i}$.  This isomorphism class of labeled posets is called the \emph{heap} of $\w$.  In particular, if $w$ is fully commutative then it has a single commutativity class, and so there is a unique heap associated to $w$.

\begin{example}\label{first.heap.ex}
Let $\w = s_3 s_2 s_1 s_2 s_5s_{4}s_{6}s_{5}$ be a reduced expression for $w \in W_{c}(\C_{5})$.  We see that $\w$ is indexed by $\{1, 2, 3, 4, 5, 6, 7, 8\}$.  As an example, we see that $3 \lessdot 2$ since $2 < 3$ and the second and third generators do not commute.  In general, the labeled Hasse diagram for the unique heap poset of $w$ is shown below.

\medskip

\begin{center}
%-- New mfpic environment, number 4 of 189. (size of end split: 2, should be 2)  ------------------->
\includegraphics{ThesisFigs1.004}
\end{center}
\end{example}

\begin{remark}
We remark that some authors define the heap of $\w$ to be the labeled poset that results from reversing the partial order.  That is, some authors draw their heaps upside-down relative to ours.  Our convention is appropriate because it more naturally aligns with the construction of our desired diagram algebra.
\end{remark}

\end{section}

\begin{section}{Representations of heaps for elements of $W(\C_{n})$}

Now, consider $X=\C_{n}$ and let $N(\C_{n})=\{1, 2, \dots, n, n+1\}$.  (Note that $N(\C_{n})$ is simply the index set for $S(\C_{n})$.)  Let $\w = \w_1 \cdots \w_r$ be a fixed reduced expression for $w \in W(X)$.  As in \cite{Billey.S;Jones.B:A} and \cite{Jones.B:B}, we will represent a heap as a set of lattice points embedded in $N(\C_{n}) \times \mathbb{N}$.  To do so, we assign coordinates (not unique) $(x,y) \in N(X) \times \mathbb{N}$ to each entry of the labeled Hasse diagram for the heap of $\w$ in such a way that:
\begin{enumerate}[label=\rm{(\arabic*)}]
\item if an entry represented by $(x,y)$ is labeled $s_i$ in the heap, then $x = i$; 
\end{enumerate}
and
\begin{enumerate}[label=\rm{(\arabic*)}, resume]
\item if an entry represented by $(x,y)$ is greater than an entry
represented by $(x',y')$ in the heap, then $y > y'$.
\end{enumerate}

Recall that a finite poset is determined by its covering relations.  In the case of $\C_{n}$, it follows from the definition that $(x,y)$ covers $(x',y')$ in the heap if and only if $x = x' \pm 1$, $y > y'$, and there are no entries $(x'', y'')$ such that $x'' \in \{x, x'\}$ and $y'< y'' < y$.  Hence, we can completely reconstruct the edges of the Hasse diagram and the corresponding heap poset from a lattice point representation. This representation enables us to make arguments ``by picture'' that are otherwise cumbersome.  

\begin{definition}
Let $\w$ be a reduced expression for $w \in W(\C_{n})$.  We let $H(\w)$ denote a lattice representation of the heap poset in $N(\C_{n}) \times \N$ constructed as described in the preceding paragraph.  \end{definition}

If $w$ is fully commutative, then the choice of reduced expression for $w$ is irrelevant. In this case, we may write $H(w)$ (note the absence of sans serif font) and we will refer to $H(w)$ as the heap of $w$.

\bigskip

Although there are many possible coordinate assignments for a given heap, the $x$-coordinates of each entry are fixed for all of them, and the coordinate assignments of any two entries only differ in the amount of vertical space between them.  We will say that all entries having the same $x$-coordinate lie in the same \emph{column}, where two columns are adjacent if they correspond to adjacent vertices in the Coxeter graph.  

\bigskip

Let $\w=\w_{1}\cdots \w_{r}$ be a reduced expression for $w \in W_{c}(\C_{n})$.  If $\w_{i}$ and $\w_{j}$ are adjacent generators in the Coxeter graph with $i<j$, then we must place the point labeled by $\w_i$ at a level that is \textit{above} the level of the point labeled by $\w_{j}$.  Because generators that are not adjacent in the Coxeter graph do commute, points that lie in non-adjacent columns can slide past each other or land at the same level.  To emphasize the covering relations of the lattice representation we will enclose each entry of the heap in a rectangle in such a way that if one entry covers another, the rectangles overlap halfway.

\begin{example}\label{second.heap.ex}
Let $w$ be as in Example \ref{first.heap.ex}.  Then one possible representation for $H(w)$ is as follows.
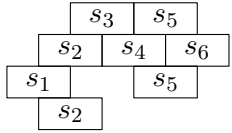
\begin{center}
%-- New mfpic environment, number 5 of 189. (size of end split: 2, should be 2)  ------------------->
\includegraphics{ThesisFigs1.005}
\end{center}
Here is another possible representation for $H(w)$.
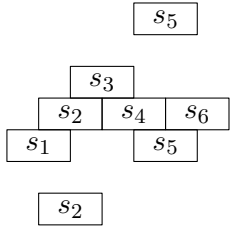
\begin{center}
%-- New mfpic environment, number 6 of 189. (size of end split: 2, should be 2)  ------------------->
\includegraphics{ThesisFigs1.006}
\end{center}
\end{example}

As in \cite{Billey.S;Jones.B:A}, when $w$ is fully commutative, we wish to make a canonical choice for the representation $H(w)$ by ``coalescing'' the entries in a particular way.  To do this, we give all entries corresponding to elements in $\L(w)$ the same vertical position; all other entries in the heap should have vertical position as high as possible.  One possible interpretation of this canonical choice is that adjacent columns represent overlapping stacks of cafeteria trays, where there are springs under each column of trays that maintain the height of the top row.  In Example \ref{second.heap.ex}, the first representation of $H(w)$ that we provided is the canonical representation.  When illustrating heaps, we will adhere to this canonical choice, but our arguments should always be viewed as referring to the underlying heap poset.  In particular, when we consider the heaps of arbitrary reduced expressions, we will only allude to the relative vertical positions of the entries, and never their absolute coordinates.  

\bigskip

Given a canonical representation of a heap, it makes sense to refer to the $k$th row of the heap, and we will occasionally do this when no confusion will arise.  (The first row of the heap corresponds to the left descent set.)  If $w \in W_{c}(\C_{n})$, let $\r_{k}$ denote the $k$th row of the canonical representation for $H(w)$.  We will write $s_{i} \in \r_{k}$ to denote that there is an entry occurring in the $k$th row labeled by $s_{i}$.  If $\r_{k}$ consists entirely of entries labeled by $s_{i_{1}}, s_{i_{2}}, \dots, s_{i_{r}}$, where $i_{j}<i_{j+1}$, then we will write $\r_{k}=s_{i_{1}} s_{i_{2}} \cdots s_{i_{r}}$.  For example, consider the canonical representation of $H(w)$ in Example \ref{second.heap.ex}.  Then $s_{5} \in \r_{3}$ and $\r_{2}=s_{2}s_{4}s_{6}$.

\begin{remark}
Our canonical representation of heaps of fully commutative elements corresponds exactly to the Cartier--Foata normal form for monomials \cite{Cartier.P;Foata.D:A, Green.R:P}.
\end{remark}

Let $w \in W_{c}(X)$ have reduced expression $\w=\w_{1}\cdots \w_{r}$ and suppose $\w_{i}$ and $\w_{j}$ are a pair of entries in the heap of $\w$ that correspond to the same generator $s_k$, so that they lie in the same column $k$ of the heap.  We say that $\w_{i}$ and $\w_{j}$ are \emph{consecutive} if there is no other occurrence of $s_{k}$ occurring between them in $\w$.  In this case, $\w_{i}$ and $\w_{j}$ are consecutive in $H(w)$, as well.  For example the two occurrences of the generator $s_{2}$ in the heaps given in Example \ref{second.heap.ex} are consecutive.  In general, for $\w$ to be reduced, there must exist at least one generator not commuting with $s_{k}$ that occurs between $\w_{i}$ and $\w_{j}$.

\begin{remark}
Since $W(B_{n})$ and $W(B'_{n})$ are both subgroups of $W(\C_{n})$, we naturally have heap representations of elements from these groups, where the appropriate entries never appear.  That is, the heap for a fully commutative element $w$ from $W(B_{n})$ (respectively, $W(B'_{n})$) is the same as the heap for $w$ when considered as an element of $W(\C_{n})$.
\end{remark}

\end{section}

\begin{section}{Saturated and convex subheaps}

Let $\w=\w_{1} \cdots \w_{r}$ be a reduced expression for $w \in W(\C_{n})$.  Then $H(w)$ is a representation of the heap poset on the set $\{1, \dots, r\}$, where $i$ is labeled by the generator $\w_{i}$.   We define a heap $H'$ to be a \emph{subheap} of $H(\w)$ if $H'=H(\w')$, where $\w'=\w_{i_{1}}\w_{i_{2}} \cdots \w_{i_{k}}$ is a subexpression of $\w$.  The subheap $H'$ is a representation of the heap poset of the set $\{i_{1}, \dots, i_{k}\} \subseteq \{1, \dots, r\}$.  We emphasize that the subexpression need not be a subword (i.e., a consecutive subexpression).  

\bigskip

A subheap $H'$ of $H(\w)$ is called a \emph{saturated subheap} if whenever $\w_{i}$ and $\w_{j}$ occur in $H'$ such that there exists a saturated chain from $i$ to $j$ in the underlying poset for $H(\w)$, there also exists a saturated chain $i=i_{k_{1}}<i_{k_{2}}< \cdots < i_{k_{l}}=j$ in the underlying poset for $H'$ such that $i=i_{k_{1}}<i_{k_{2}}< \cdots < i_{k_{l}}=j$ is also a saturated chain in the underlying poset for $H(\w)$.  (Note that, to the best of our knowledge, the notion of saturated subheap has not appeared in the literature before.)  

\begin{example}\label{third.heap.ex}
Let $\w= s_3 s_2 s_1 s_2 s_5s_{4}s_{6}s_{5}$ as in Example \ref{first.heap.ex}.  Also, let $\w'=s_3 s_1 s_{5}$ be the subexpression of $\w$ that results from deleting all but the first, third, and last generators of $\w$.  Then $H(\w')$ equals
\begin{center}
%-- New mfpic environment, number 7 of 189. (size of end split: 2, should be 2)  ------------------->
\includegraphics{ThesisFigs1.007}
\end{center}
and is a subheap of $H(\w)$.  However, $H(\w')$ is not a saturated subheap of $H(\w)$ since there is a saturated chain in $H(\w)$ from the lower occurrence of $s_{5}$ to the occurrence of $s_{3}$, but there is not a chain between the corresponding entries in $H(\w')$.  Note that we could obtain a heap with an identical representation by considering the subexpression that results from deleting all but the first, third, and fifth generators of $\w$.  Now, let $\w''=s_{5}s_{4}s_{5}$ be the subexpression of $\w$ that results from deleting all but fifth, sixth, and last generators of $\w$.  Then $H(\w'')$ equals 
\begin{center}
%-- New mfpic environment, number 8 of 189. (size of end split: 2, should be 2)  ------------------->
\includegraphics{ThesisFigs1.008}
\end{center}
and is a saturated subheap of $H(\w)$.  
\end{example}

Recall that a subposet $Q$ of $P$ is called \emph{convex} if $y \in Q$ whenever $x < y < z$ in $P$ and $x, z \in Q$.  We will refer to a subheap as a \emph{convex subheap} if the underlying subposet is convex.  (Note that all convex subheaps are saturated, but not all saturated subheaps are convex.)

\begin{example}
Let $\w$ and $\w''$ be as in Example \ref{third.heap.ex}.  Then $H(\w'')$ is not a convex subheap since there is an entry in $H(\w)$ labeled by $s_{6}$ occurring between the two consecutive occurrences of $s_{5}$ that does not occur in $H(\w'')$.  However, if we do include the entry labeled by $s_{6}$, then
\begin{center}
%-- New mfpic environment, number 9 of 189. (size of end split: 2, should be 2)  ------------------->
\includegraphics{ThesisFigs1.009}
\end{center}
is a convex subheap of $H(\w)$.  Let $\w'''=s_2 s_1 s_2$ be the subexpression that results from deleting all but the second, third, and fourth generators of $\w$.  Then $H(\w''')$ is equal to
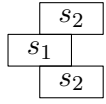
\begin{center}
%-- New mfpic environment, number 10 of 189. (size of end split: 2, should be 2)  ------------------->
\includegraphics{ThesisFigs1.010}
\end{center}
and is a convex subheap of $H(\w)$.  

\end{example}

From this point on, if there can be no confusion, we will not specify the exact subexpression that a subheap arises from.

\bigskip

The following fact is implicit in the literature (in particular, see the proof of Proposition 3.3 in \cite{Stembridge.J:B}) and follows easily from the definitions.

\begin{proposition}
Let $w \in W_{c}(X)$.  Then $H'$ is a convex subheap of $H(w)$ if and only if $H'$ is the heap for some subword of some reduced expression for $w$.   
\hfill $\qed$
\end{proposition}

\end{section}

\begin{section}{Heaps for fully commutative elements in $W(\C_{n})$}

It will be extremely useful for us to be able to recognize when a heap corresponds to a fully commutative element in $W(\C_{n})$.  The following lemma follows immediately from Remark \ref{illegal.convex.chains} and is also a special case of \cite[Proposition 3.3]{Stembridge.J:B}.

\begin{lemma}\label{impermissible.heap.configs}
Let $w \in W_{c}(\C_{n})$.  Then $H(w)$ cannot contain any of the following convex subheaps, where we use $\emptyset$ to emphasize that no element of the heap occupies that position.
\begin{enumerate}[label=\rm{(\roman*)}]
\item \begin{tabular}[c]{ccc}
\begin{tabular}[c]{c}
%-- New mfpic environment, number 11 of 189. (size of end split: 2, should be 2)  ------------------->
\includegraphics{ThesisFigs1.011}
\end{tabular}
& \begin{tabular}[c]{c}
and 
\end{tabular} &
\begin{tabular}[c]{c}
%-- New mfpic environment, number 12 of 189. (size of end split: 2, should be 2)  ------------------->
\includegraphics{ThesisFigs1.012}
\end{tabular}
\end{tabular};
\item \begin{tabular}[c]{ccc}
\begin{tabular}[c]{c}
%-- New mfpic environment, number 13 of 189. (size of end split: 2, should be 2)  ------------------->
\includegraphics{ThesisFigs1.013}
\end{tabular}
& \begin{tabular}[c]{c}
and 
\end{tabular} &
\begin{tabular}[c]{c}
%-- New mfpic environment, number 14 of 189. (size of end split: 2, should be 2)  ------------------->
\includegraphics{ThesisFigs1.014}
\end{tabular}
\end{tabular};

\item \begin{tabular}[c]{cccc}
\begin{tabular}[c]{c}
%-- New mfpic environment, number 15 of 189. (size of end split: 2, should be 2)  ------------------->
\includegraphics{ThesisFigs1.015}
\end{tabular}
& \begin{tabular}[c]{c}
and
\end{tabular} &
\begin{tabular}[c]{c}
%-- New mfpic environment, number 16 of 189. (size of end split: 2, should be 2)  ------------------->
\includegraphics{ThesisFigs1.016}
\end{tabular},
&
\begin{tabular}[c]{c}
 where $1<k<n+1$.
\end{tabular}
\end{tabular}
\end{enumerate} \hfill $\qed$
\end{lemma}

\begin{remark}
It is important to note that since $m(s_{1},s_{2})=m(s_{n},s_{n+1})=4$ in $W(\C_{n})$, the heap of a fully commutative element may contain the following convex chains:

\begin{center}
\begin{tabular}{ccccc}
\begin{tabular}[c]{c}
%-- New mfpic environment, number 17 of 189. (size of end split: 2, should be 2)  ------------------->
\includegraphics{ThesisFigs1.017}
\end{tabular}
&& and &&
\begin{tabular}[c]{c}
%-- New mfpic environment, number 18 of 189. (size of end split: 2, should be 2)  ------------------->
\includegraphics{ThesisFigs1.018}
\end{tabular}.
\end{tabular}
\end{center}
\end{remark}

\begin{definition}\label{n-value.def}
Let $w \in W_{c}(\C_{n})$.  We define $n(w)$ to be the maximum integer $k$ such that $w$ has a reduced expression of the form $w = u x v$ (reduced), where $u, x, v \in W_{c}(\C_{n})$, $l(x)=k$, and $x$ is a product of commuting generators.
\end{definition}

Note that $n(w)$ may be greater than the size of any row in the canonical representation of $H(w)$.  Also, it is known that $n(w)$ is equal to the size of a maximal antichain in the heap poset for $w$.

\begin{example}
Let $\w=s_{2}s_{1}s_{3}s_{5}$ be a reduced expression for $w \in W_{c}(\C_{4})$.  Then
	$$H(w)= \begin{tabular}[c]{c}
%-- New mfpic environment, number 19 of 189. (size of end split: 2, should be 2)  ------------------->
\includegraphics{ThesisFigs1.019}
\end{tabular}$$
and $n(w)=3$, where $x$ from Definition \ref{n-value.def} equals $s_{1}s_{3}s_{5}$. 
\end{example}

\end{section}

\begin{section}{Weak star operations on heaps for fully commutative elements}

We conclude this chapter with a few observations regarding heaps and weak star reductions.  Let $\w=\w_{1} \cdots \w_{r}$ be a reduced expression for $w \in W_{c}(\C_{n})$.  Then $w$ is left weak star reducible by $s$ with respect to $t$ if and only if
\begin{enumerate}[label=\rm{(\arabic*)}]
\item there is an entry in $H(\w)$ labeled by $s$ that is not covered by any other entry; 
\end{enumerate}
and
\begin{enumerate}[label=\rm{(\arabic*)}, resume]
\item the heap $H(t\w)$ contains one of the convex subheaps of Lemma \ref{impermissible.heap.configs}.
\end{enumerate}

Of course, we have an analogous statement for right weak star reducibility.

\begin{example}
Let $\w=s_{3}s_{5}s_{2}s_{4}s_{6}s_{1}s_{2}$ be a reduced expression for $w \in W_{c}(\C_{5})$.  Then
	$$H(w)= \begin{tabular}[c]{c}
%-- New mfpic environment, number 20 of 189. (size of end split: 2, should be 2)  ------------------->
\includegraphics{ThesisFigs1.020}
\end{tabular}$$
and we see that $w$ is only left weak star reducible by $s_{3}$ with respect to $s_{2}$ and only right weak star reducible by $s_{2}$ with respect to $s_{1}$.  That is,
	$$H\(\star^{L}_{3,2}(w)\)=\begin{tabular}[c]{c}
%-- New mfpic environment, number 21 of 189. (size of end split: 2, should be 2)  ------------------->
\includegraphics{ThesisFigs1.021}
\end{tabular}$$
and
	$$H\(\star^{R}_{2,1}(w)\)=\begin{tabular}[c]{c}
%-- New mfpic environment, number 22 of 189. (size of end split: 2, should be 2)  ------------------->
\includegraphics{ThesisFigs1.022}
\end{tabular}.$$
All other weak star reductions on $w$ are undefined.
\end{example}

\end{section}

\end{chapter}

%%%%%%%%%% Chapter 3 %%%%%%%%%%%%%

\begin{chapter}{The type I and type II elements of $W(\C_{n})$}

In this chapter, we study the combinatorics of Coxeter groups of types $B$ and $\C$.  Our immediate goal is to define two classes of fully commutative elements of $W(\C_{n})$ that play a central role in the remainder of this thesis.  Some of these elements will turn out to be on our list of irreducible elements which appears in Chapter \ref{chaptCwsrm}.

\begin{section}{The type I elements}

\begin{definition}\label{def.zigzags}
Define the following elements of $W(\C_{n})$.
\begin{enumerate}[label=\rm{(\arabic*)}]
\item If $i<j$, let
	$$\z_{i,j}=s_{i}s_{i+1}\cdots s_{j-1}s_{j}$$
and
	$$\z_{j,i}=s_{j}s_{j-1}\cdots s_{i-1}s_{i}.$$
We also let $\z_{i,i}=s_{i}$.

\item If $1< i \leq n+1$ and $1 < j \leq n+1$, let
	$$\z^{L,2k}_{i,j}=\z_{i,2}(\z_{1,n}\z_{n+1,2})^{k-1}\z_{1,n}\z_{n+1,j}.$$
	
\item If $1< i \leq n+1$ and $1 \leq j < n+1$, let
	$$\z^{L,2k+1}_{i,j}=\z_{i,2}(\z_{1,n}\z_{n+1,2})^{k}\z_{1,j}.$$

\item If $1\leq i < n+1$ and $1 \leq j <  n+1$, let
	$$\z^{R,2k}_{i,j}=\z_{i,n}(\z_{n+1,2}\z_{1,n})^{k-1}\z_{n+1,2}\z_{1,j}.$$
	
\item If $1\leq i < n+1$ and $1 < j \leq  n+1$, let 
	$$\z^{R,2k+1}_{i,j}=\z_{i,n}(\z_{n+1,2}\z_{1,n})^{k}\z_{n+1,j}.$$

\end{enumerate}
We will refer to the elements in (1)--(5) as the \emph{type I} elements.
\end{definition}

\begin{remark}\label{typeI.n(w)=1}
The heaps (which motivated the notation above) corresponding to the type I elements are as follows.
\begin{enumerate}[label=\rm{(\arabic*)}]

\item If $i<j$, then
	$$H\(\z_{i,j}\)=\begin{tabular}[c]{c}
%-- New mfpic environment, number 23 of 189. (size of end split: 2, should be 2)  ------------------->
\includegraphics{ThesisFigs1.023}
\end{tabular} \quad \text{ and } \quad H\(\z_{j,i}\)=\begin{tabular}[c]{c}
%-- New mfpic environment, number 24 of 189. (size of end split: 2, should be 2)  ------------------->
\includegraphics{ThesisFigs1.024}
\end{tabular}.$$

\item If $1< i \leq n+1$ and $1 < j \leq n+1$, then
	$$H\(\z^{L,2k}_{i,j}\)=\begin{tabular}[c]{c}
%-- New mfpic environment, number 25 of 189. (size of end split: 2, should be 2)  ------------------->
\includegraphics{ThesisFigs1.025}
\end{tabular},$$
where we encounter an entry labeled by either $s_{1}$ or $s_{n+1}$ a combined total of $2k$ times if $i\neq n+1$ and $2k+1$ times if $i=n+1$.

\item If $1< i \leq n+1$ and $1 \leq j < n+1$, then
	$$H\(\z^{L,2k+1}_{i,j}\)=\begin{tabular}[c]{c}
%-- New mfpic environment, number 26 of 189. (size of end split: 2, should be 2)  ------------------->
\includegraphics{ThesisFigs1.026}
\end{tabular},$$
where we encounter an entry labeled by either $s_{1}$ or $s_{n+1}$ a combined total of $2k+1$ times if $i\neq n+1$ and $2k+2$ times if $i=n+1$.

\item If $1\leq i < n+1$ and $1 \leq j <  n+1$, then
	$$H\(\z_{i,j}^{R,2k}\)=\begin{tabular}[c]{c}
%-- New mfpic environment, number 27 of 189. (size of end split: 2, should be 2)  ------------------->
\includegraphics{ThesisFigs1.027}
\end{tabular},$$
where we encounter an entry labeled by either $s_{1}$ or $s_{n+1}$ a combined total of $2k$ times if $i\neq 1$ and $2k+1$ times if $i=1$.

\item If $1\leq i < n+1$ and $1 < j \leq  n+1$, then
	$$H\(\z^{R,2k+1}_{i,j}\)=\begin{tabular}[c]{c}
%-- New mfpic environment, number 28 of 189. (size of end split: 2, should be 2)  ------------------->
\includegraphics{ThesisFigs1.028}
\end{tabular},$$
where we encounter an entry labeled by either $s_{1}$ or $s_{n+1}$ a combined total of $2k+1$ times if $i\neq 1$ and $2k+2$ times if $i=1$.

\end{enumerate}

\end{remark}

\begin{remark}
We gather a few remarks about the type I elements.
\begin{enumerate}[label=\rm{(\arabic*)}]
\item Every type I element is rigid, in the sense that each has a unique reduced expression.  This implies that every type I element is fully commutative (there are no relations of any kind to apply).  

\item The index $i$ (respectively, $j$) tells us which generator is in the left (respectively, right) descent set.  These are the unique elements occurring in the left and right descent sets since each element is rigid.  The $L$ (respectively, $R$) in the notation means that we begin multiplying $s_{i}$ by $s_{i-1}$ (respectively, $s_{i+1}$).  Also, $2k+1$ (respectively, $2k$) indicates the number of times we should encounter an end generator (i.e., $s_{1}$ or $s_{n+1}$) after the first occurrence of $s_{i}$ as we zigzag through the generators.  If $s_{i}$ is an end generator, it is not included in this count.  However, if $s_{j}$ is an end generator, it is included. 

\item It is clear from looking at the heaps for the type I elements that if $w$ is of type I, then $n(w)=1$ (see Definition \ref{n-value.def}).  Conversely, it follows by induction on $l(w)$ that if $n(w)=1$ for some $w \in W_{c}(\C_{n})$, then $w$ must be of type I. 

\item Lastly, note that there is an infinite number of type I elements (there is no limit to the zigzagging that the heaps of the type I elements can do).
\end{enumerate}
\end{remark}

\begin{example}
Consider $W(\C_{4})$.  Then
	$$\z^{L,2}_{2,3}=s_{2}s_{1}s_{2}s_{3}s_{4}s_{5}s_{4}s_{3}$$
with
	$$H\(\z^{L,2}_{2,3}\)=\begin{tabular}[c]{c}
%-- New mfpic environment, number 29 of 189. (size of end split: 2, should be 2)  ------------------->
\includegraphics{ThesisFigs1.029}
\end{tabular}.$$
Also, we have
	$$\z^{R,2}_{1,1}=s_{1}s_{2}s_{3}s_{4}s_{5}s_{4}s_{3}s_{2}s_{1}$$
with
	$$H\(\z^{R,2}_{1,1}\)=\begin{tabular}[c]{c}
%-- New mfpic environment, number 30 of 189. (size of end split: 2, should be 2)  ------------------->
\includegraphics{ThesisFigs1.030}
\end{tabular}.$$
\end{example}

The next proposition follows immediately from the discussion above.

\begin{proposition}\label{zigzags}
If $w \in W(\C_{n})$ is of type I, then $w$ is fully commutative with $n(w)=1$.  Conversely, if $n(w)=1$, then $w$ is one of the elements on the list in Definition \ref{def.zigzags}.  \hfill $\qed$
\end{proposition}

\begin{remark}\label{irred.zigzags}
It is easily seen that if $w$ is of type I beginning (respectively, ending) with $s_{i}$, then $w$ is left (respectively, right) weak star reducible by $s_{i}$ if and only if $i \notin \{1,n+1\}$.  That is, $\z_{1,1}^{R,2k}$, $\z_{1,n+1}^{R,2k+1}$, $\z_{n+1,n+1}^{L,2k}$, and $\z_{n+1,1}^{L,2k+1}$ are the only type I elements that are irreducible.  We will see these elements appearing on our list of irreducible elements in Chapter \ref{chaptCwsrm}.
\end{remark}

\end{section}

\begin{section}{The type II elements}

It will be helpful for us to define $l=\lceil \frac{n-1}{2}\rceil$.  Then regardless of whether $n$ is odd or even, $2l$ will always be the largest even number in $\{1,2,\dots, n, n+1\}$.  Similarly, $2l+1$ will always be the largest odd number in $\{1,2,\dots, n, n+1\}$.

\begin{definition}
Define $\O=\{1,3, \dots, 2l-1, 2l+1\}$ and $\E=\{2, 4, \dots, 2l-2, 2l\}$.  Note that $\O$ consists of all of the odd indices and $\E$ consists of all of the even indices amongst $\{1, 2, \dots, n, n+1\}$.  
\begin{enumerate}[label=\rm{(\arabic*)}]
\item Let $i$ and $j$ be of the same parity with $i<j$.  Then we define 
	$$\x_{i,j}=s_{i}s_{i+2}\cdots s_{j-2}s_{j}.$$
Also, define
	$$\x_{\O}=\x_{1,2l+1}=s_{1}s_{3}\cdots s_{2l-1}s_{2l+1},$$
and
	$$\x_{\E}=\x_{2,2l}=s_{2}s_{4}\cdots s_{2l-2}s_{2l}.$$
\item Let $1\leq k < \infty$.  Define
	$$\y_{k}=\underbrace{(\x_{\O}\x_{\E})\cdots (\x_{\O}\x_{\E})}_{k \text{ copies}}.$$
\end{enumerate}
We will refer to finite alternating products of $\x_{\O}$ and $\x_{\E}$ as \emph{type II} elements.  That is, the type II elements are
\begin{enumerate}[label=\rm{(\arabic*)}]
\item $\x_{\O}$;
\item $\x_{\E}$;
\item $\y_{k}$;
\item $\x_{\E}\y_{k}$;
\item $\x_{\E}\y_{k}\x_{\O}$;
\item $\y_{k}\x_{\O}$.
\end{enumerate}
\end{definition}

Note that $\x_{\O}$ and $\x_{\E}$ are products of commuting generators, and so the order of the multiplication is immaterial.  The heaps corresponding to the type II elements are as follows.
\begin{enumerate}[label=\rm{(\arabic*)}]
\item $H\(\x_{\O}\)=\begin{tabular}[c]{c}
%-- New mfpic environment, number 31 of 189. (size of end split: 2, should be 2)  ------------------->
\includegraphics{ThesisFigs1.031}
\end{tabular}.$

\item $H\(\x_{\E}\)=\begin{tabular}[c]{c}
%-- New mfpic environment, number 32 of 189. (size of end split: 2, should be 2)  ------------------->
\includegraphics{ThesisFigs1.032}
\end{tabular}.$
	
\item If $n$ is even, then
	$$H\(\y_{k}\)=\begin{tabular}[c]{c}
%-- New mfpic environment, number 33 of 189. (size of end split: 2, should be 2)  ------------------->
\includegraphics{ThesisFigs1.033}
\end{tabular}.$$
If $n$ is odd, then
	$$H\(\y_{k}\)=\begin{tabular}[c]{c}
%-- New mfpic environment, number 34 of 189. (size of end split: 2, should be 2)  ------------------->
\includegraphics{ThesisFigs1.034}
	\end{tabular}.$$
In each case, the canonical representation of $H(\y_{k})$ has $2k$ rows.

\item  If $n$ is even, then
	$$H\(\x_{\E}\y_{k}\)=\begin{tabular}[c]{c}
%-- New mfpic environment, number 35 of 189. (size of end split: 2, should be 2)  ------------------->
\includegraphics{ThesisFigs1.035}
	\end{tabular}.$$
If $n$ is odd, then
	$$H\(\x_{\E}\y_{k}\)=\begin{tabular}[c]{c}
%-- New mfpic environment, number 36 of 189. (size of end split: 2, should be 2)  ------------------->
\includegraphics{ThesisFigs1.036}
	\end{tabular}.$$
In each case, the canonical representation of $H\(\x_{\E}\y_{k}\)$ has $2k+1$ rows.

\item If $n$ is even, then
	$$H\(\x_{\E}\y_{k}\x_{\O}\)=\begin{tabular}[c]{c}
%-- New mfpic environment, number 37 of 189. (size of end split: 2, should be 2)  ------------------->
\includegraphics{ThesisFigs1.037}
	\end{tabular}.$$
If $n$ is odd, then
	$$H\(\x_{\E}\y_{k}\x_{\O}\)=\begin{tabular}[c]{c}
%-- New mfpic environment, number 38 of 189. (size of end split: 2, should be 2)  ------------------->
\includegraphics{ThesisFigs1.038}
	\end{tabular}.$$
In each case, the canonical representation of $H\(\x_{\E}\y_{k}\x_{\O}\)$ has $2k+2$ rows.

\item If $n$ is even, then
	$$H\(\y_{k}\x_{\O}\)=\begin{tabular}[c]{c}
%-- New mfpic environment, number 39 of 189. (size of end split: 2, should be 2)  ------------------->
\includegraphics{ThesisFigs1.039}
	\end{tabular}.$$
If $n$ is odd, then
	$$H\(\y_{k}\x_{\O}\)=\begin{tabular}[c]{c}
%-- New mfpic environment, number 40 of 189. (size of end split: 2, should be 2)  ------------------->
\includegraphics{ThesisFigs1.040}	
\end{tabular}.$$
In each case, the canonical representation of $H\(\y_{k}\x_{\O}\)$ has $2k+1$ rows.
\end{enumerate}

The next proposition follows immediately from looking at the heaps above.

\begin{proposition}\label{sandwich_stacks}
If $w \in W(\C_{n})$ is of type II, then $w$ is fully commutative with $n(w)=l=\lceil \frac{n-1}{2}\rceil$.  

\hfill $\qed$
\end{proposition}

\begin{remark}
It is easily seen by inspecting the heaps for the type II elements that if $w$ is of type II, then $w$ is irreducible.  We will see these elements appearing on our list of irreducible elements in the next chapter.
\end{remark}

Note that if $w \in W_{c}(\C_{n})$, then $l$ is the maximum value that $n(w)$ can take.  There are fully commutative elements with $n(w)=l$ that are not of type II.  For example, if $w=s_{2}s_{1}s_{3}s_{5} \in W_{c}(\C_{4})$, then $n(w)=3$, but $w$ is not of type II.  Note, however, that there is an infinite number of type II elements.  

\bigskip

In \cite{Green.R:P}, Green proved that when $n$ is odd, $W(\C_{n})$ is a star reducible Coxeter group.
Note that if $n$ is even, then every $\y_{k}\x_{\O}$ from (5) above is not star reducible.  This fact is implicit in \cite{Green.R:P} and is easily verified.  We conjecture that these elements are the only non-star reducible elements in $W_{c}(\C_{n})$ (with $n$ even) other than products of commuting generators.  This conjecture will be a topic of future research.  

\end{section}

\end{chapter}

%%%%%%%%%% Chapter 4 %%%%%%%%%%%

\begin{chapter}{Classification of the weak star irreducible elements in $W(B_{n})$}\label{chaptBwsrm}

The goal of this chapter is to classify the irreducible elements of $W(B_{n})$.  We will use this classification to prove the classification of the irreducible elements of $W(\C_{n})$, which plays a pivotal role in the proof of our main result in the final chapter.

\begin{section}{Statement of theorem}

As mentioned in Chapter 1, our definition of weak star reducible is related to Fan's notion of cancellable in \cite{Fan.C:A}.  In addition, the next theorem verifies Fan's unproved claim in \cite[\textsection 7.1]{Fan.C:A} about the set of $w \in W_{c}(B_{n})$ having no element of $\L(w)$ or $\R(w)$ that can be left or right cancelled, respectively.  

\begin{theorem}\label{Bwsrm}
Let $w \in W_{c}(B_{n})$.  Then $w$ is irreducible if and only if $w$ is equal to one of the elements on the following list.
\begin{enumerate}[label=\rm{(\roman*)}]
\item $w_{p}$;
\item $s_{1}s_{2}w_{p}$, where $s_{1}, s_{2}, s_{3} \notin \supp(w_{p})$;
\item $s_{2}s_{1}w_{p}$, where $s_{1}, s_{2}, s_{3} \notin \supp(w_{p})$;
\end{enumerate}
where in each case $w_{p}$ is equal to a product of commuting generators in $W_{c}(B_{n})$.

\bigskip
We have an analogous statement for $W_{c}(B'_{n})$, where $s_{1}$ and $s_{2}$ are replaced with $s_{n+1}$ and $s_{n}$, respectively.
\end{theorem}

To prove this theorem, we require several lemmas.

\end{section}

\begin{section}{Preparatory lemmas}

In this section, we state and prove several technical lemmas that will be used to prove Theorem \ref{Bwsrm}.  A few of these lemmas will also be useful in later chapters.  Note that all statements about $W(\C_{n})$ also apply to $W(B_{n})$ and $W(B'_{n})$.

\bigskip

All of the results in this section are new. 

\bigskip

Before proceeding, we make a comment on notation.  Let $w \in W_{c}(\C_{n})$.  When representing saturated and convex subheaps of $H(w)$, we will use the symbol $\emptyset$ to emphasize the absence of an entry in this location in $H(w)$.  It is important to note that the occurrence of the symbol $\emptyset$  implies that an entry from $H(w)$ cannot be shifted vertically from above or below to occupy the location of the symbol $\emptyset$.  

\begin{example}
In the following heap, the region enclosed by the dotted line and labeled with $\emptyset$ indicates that no entry of the heap may occupy this region.

%-- New mfpic environment, number 41 of 189. (size of end split: 2, should be 2)  ------------------->
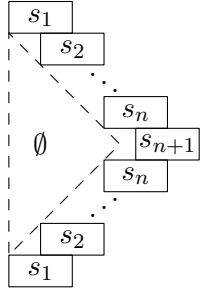
\begin{center}
\includegraphics{ThesisFigs1.041}
\end{center}
In this example, if the heap is a saturated subheap of some larger heap, then the subheap is convex.
\end{example}

\begin{lemma}\label{is_zigzag}
Let $w \in W_{c}(\C_{n})$.  Suppose that $w$ has a reduced expression having one of the following fully commutative elements as a subword:  
\begin{enumerate}[label=\rm{(\roman*)}]
\item $\z^{L,2}_{2,n}=s_{2}s_{1}s_{2}s_{3} \cdots s_{n-1}s_{n}s_{n+1}s_{n}$,
\item $\z^{R,2}_{n,2}=s_{n}s_{n+1}s_{n}s_{n-1}\cdots s_{3}s_{2}s_{1}s_{2}$,
\item $\z^{R,2}_{1,1}=s_{1}s_{2}\cdots s_{n}s_{n+1}s_{n}\cdots s_{2}s_{1}$,
\item $\z^{L,2}_{n+1,n+1}=s_{n+1}s_{n}\cdots s_{2}s_{1}s_{2}\cdots s_{n}s_{n+1}$ .
\end{enumerate}
Then $w$ is of type I.
\end{lemma}

\begin{proof}
First, we prove (i); (ii) follows by a symmetric argument.  Assume that $\z^{L,2}_{2,n}$ is a subword of $w$.  Then

%-- New mfpic environment, number 42 of 189. (size of end split: 2, should be 2)  ------------------->
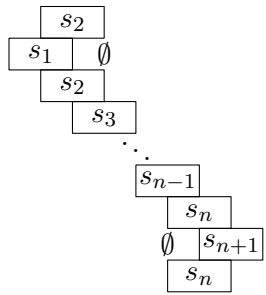
\begin{center}
\includegraphics{ThesisFigs1.042}
\end{center}

is a convex subheap of $H(w)$.  Since $\z^{L,2}_{2,n}$ is a subword of $w$, we can write $w=u \z^{L,2}_{2,n} v$ (reduced).  If $u$ and $v$ are empty, we are done.  Without loss of generality, assume that $u$ is nonempty.  If the higher occurrence of $s_{n}$ in the heap for $\z^{L,2}_{2,n}$ is covered by an entry in the heap for $w$ labeled by $s_{n+1}$, we would obtain one of the impermissible convex chains of Lemma \ref{impermissible.heap.configs}, which would contradict $w \in W_{c}(\C_{n})$.  This implies that the entry in the heap for $\z^{L,2}_{2,n}$ labeled by $s_{n}$ may only be covered by the entry labeled by $s_{n-1}$.  Iterating, we see that each entry in the heap for $\z^{L,2}_{2,n}$ labeled by $s_{i}$, for $3\leq i \leq n+1$, may only be covered by an entry labeled by $s_{i-1}$.  Then 

%-- New mfpic environment, number 43 of 189. (size of end split: 2, should be 2)  ------------------->
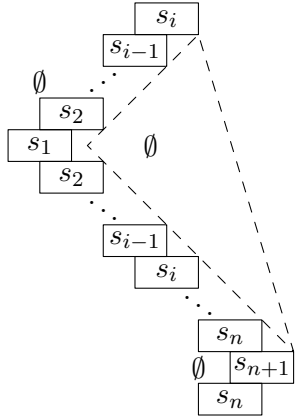
\begin{center}
\includegraphics{ThesisFigs1.043}
\end{center}

must be a convex subheap of $w$ for some $3 \leq i \leq n+1$, where each of the entries labeled by the same generator are consecutive occurrences.   Choose the largest such $i$.  By repeating similar arguments, we quickly see that $u \z^{L,2}_{2,n}$ must be of type I; otherwise, at some point, we contradict $w \in W_{c}(\C_{n})$.  A similar argument shows that $\z^{L,2}_{2,n} v$ must also be of type I.  Therefore, $w$ is of type I.

\bigskip

Now, we prove (iii); (iv) follows by a symmetric argument.  Assume that $\z^{R,2}_{1,1}$ is a subword of $w$.  Then

%-- New mfpic environment, number 44 of 189. (size of end split: 2, should be 2)  ------------------->
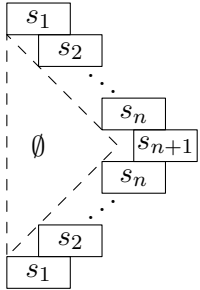
\begin{center}
\includegraphics{ThesisFigs1.044}
\end{center}

is a convex subheap of $H(w)$, where each of the entries labeled by the same generator are consecutive occurrences.  We can write $w=u \z^{R,2}_{1,1} v$ (reduced).  If $u$ and $v$ are empty, there is nothing to prove.  Without loss of generality, assume that $u$ is nonempty.  An argument identical to the above shows that the higher occurrence of the entry in the heap for $\z^{R,2}_{1,1}$ labeled by $s_{i}$, for $2 \leq i \leq n+1$, can only be covered by $s_{i-1}$.  Since $u$ is nonempty, the higher occurrence of the entry in the heap for $\z^{R,2}_{1,1}$ labeled by $s_{1}$ must be covered by an entry in the heap of $w$ labeled by $s_{2}$.  But this implies that 

%-- New mfpic environment, number 45 of 189. (size of end split: 2, should be 2)  ------------------->
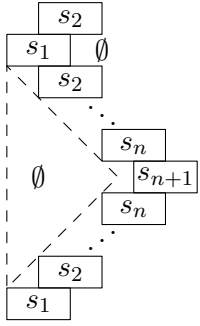
\begin{center}
\includegraphics{ThesisFigs1.045}
\end{center}

is a convex subheap of $H(w)$.  So, $\z^{L,2}_{2,1}$ is a subword of $w$.  In particular, $\z^{L,2}_{2,n}$ is then a subword of $w$.  By case (i), $w$ must be of type I.
\end{proof}

The purpose of the next three lemmas (Lemmas \ref{all.filled.in}, \ref{zigzag_subword}, and \ref{typeI.or.nothing.above}) is to prove Lemma \ref{main_zigzag_lemma}.

\begin{lemma}\label{all.filled.in}
Let $w \in W_{c}(\C_{n})$.  If

%-- New mfpic environment, number 46 of 189. (size of end split: 2, should be 2)  ------------------->
\begin{center}
\includegraphics{ThesisFigs1.046}
\end{center}

is a saturated subheap of $H(w)$, where $i \neq n+1$, then

%-- New mfpic environment, number 47 of 189. (size of end split: 2, should be 2)  ------------------->
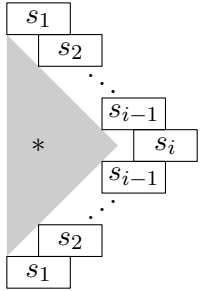
\begin{center}
\includegraphics{ThesisFigs1.047}
\end{center}

is a convex subheap of $H(w)$, where the shaded triangle labeled by $*$ means that every possible entry occurs in this region (i.e., convex closure) of the subheap.
\end{lemma}

\begin{proof}
This follows quickly from Lemma \ref{impermissible.heap.configs}; all other configurations will violate $w$ being fully commutative.
\end{proof}

\begin{example}
Let $w \in W_{c}(\C_{n})$ for $n\geq 5$, so that $m(s_{4}, s_{5})=3$.  If

%-- New mfpic environment, number 48 of 189. (size of end split: 2, should be 2)  ------------------->
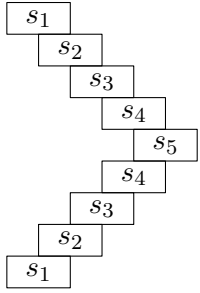
\begin{center}
\includegraphics{ThesisFigs1.048}
\end{center}

is a saturated subheap of $H(w)$, then

%-- New mfpic environment, number 49 of 189. (size of end split: 2, should be 2)  ------------------->
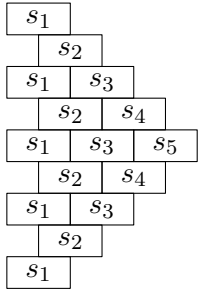
\begin{center}
\includegraphics{ThesisFigs1.049}
\end{center}
is also a saturated subheap of $H(w)$ and
	$$s_{1}s_{2}s_{1}s_{3}s_{2}s_{4}s_{1}s_{3}s_{5}s_{2}s_{4}s_{1}s_{3}s_{2}s_{1}$$
is a subword of some reduced expression for $w$.
\end{example}

\begin{lemma}\label{zigzag_subword}
Let $w \in W_{c}(\C_{n})$.  Suppose that there exists $i$ with $1<i < n+1$ such that $s_{i+1}$ does not occur between two consecutive occurrences of $s_{i}$ in $w$.  Then $\z^{L,1}_{i,i}$ is a subword of some reduced expression for $w$.  In terms of heaps, if  $H(w)$ has two consecutive occurrences of entries labeled by $s_{i}$ such that there is no entry labeled by $s_{i+1}$ occurring between them, then 

%-- New mfpic environment, number 50 of 189. (size of end split: 2, should be 2)  ------------------->
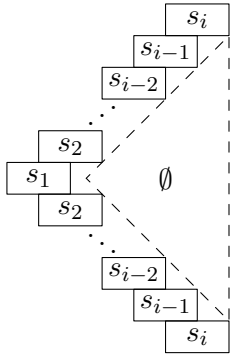
\begin{center}
\includegraphics{ThesisFigs1.050}
\end{center}

is a convex subheap of $H(w)$. 
\end{lemma}

\begin{proof}
We proceed by induction on $i$.  For the base case, let $i=2$ and suppose that there exist two consecutive occurrences of $s_{2}$ such that $s_{3}$ does not occur between them.  Then

%-- New mfpic environment, number 51 of 189. (size of end split: 2, should be 2)  ------------------->
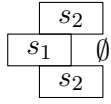
\begin{center}
\includegraphics{ThesisFigs1.051}
\end{center}

must be a convex subheap of $H(w)$, which implies that $s_{2}s_{1}s_{2}$ is a subword of $w$, as desired.  For the inductive step, assume that for $2\leq j \leq i-1$, whenever the hypotheses are met for $j$, $\z_{j,j}^{L,1}$ is a subword of $w$.  Now, assume that hypotheses are true for $i$.  Consider the entries in $H(w)$ corresponding to the two consecutive occurrences of $s_{i}$.  Since there is no entry labeled by $s_{i+1}$ occurring between these entries and $w$ is fully commutative, there must be at least two entries labeled by $s_{i-1}$ occurring between the consecutive occurrences of $s_{i}$.  For sake of a contradiction, assume that there are three or more entries in $H(w)$ labeled by $s_{i-1}$ occurring between the minimal pair of entries labeled by $s_{i}$.  By induction (on $i-1$), 

%-- New mfpic environment, number 52 of 189. (size of end split: 2, should be 2)  ------------------->
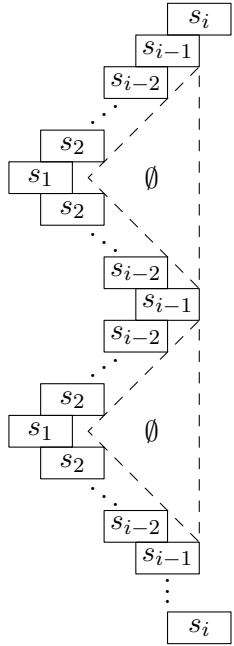
\begin{center}
\includegraphics{ThesisFigs1.052}
\end{center}

is a saturated subheap of $H(w)$.  But by Lemma \ref{all.filled.in}, the convex closure of the saturated subheap occurring between the top two occurrences of $s_{1}$ must be completely filled in.  This produces a convex chain that corresponds to the subword $s_{2}s_{1}s_{2}s_{1}$, which contradicts $w \in W_{c}(\C_{n})$.  Therefore, between the consecutive occurrences of entries labeled by $s_{i}$, there must be exactly two occurrences of an entry labeled by $s_{i}$.  That is, by induction,

%-- New mfpic environment, number 53 of 189. (size of end split: 2, should be 2)  ------------------->
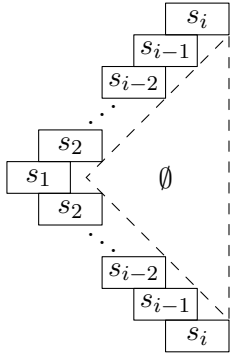
\begin{center}
\includegraphics{ThesisFigs1.053}
\end{center}
is a convex subheap of $H(w)$ and hence $\z_{i,i}^{L,1}$ is a subword of some reduced expression for $w$.
\end{proof}

\begin{lemma}\label{typeI.or.nothing.above}
Let $w \in W_{c}(\C_{n})$ such that $s_{2}s_{1}s_{2}$ is a subword of some reduced expression for $w$ and let $i$ be the largest index such that

%-- New mfpic environment, number 54 of 189. (size of end split: 2, should be 2)  ------------------->
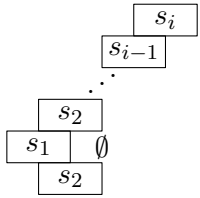
\begin{center}
\includegraphics{ThesisFigs1.054}
\end{center}

is a saturated subheap of $H(w)$.  Then one or both of the following must be true about $w$: 
\begin{enumerate}[label=\rm{(\roman*)}]
\item $w$ is of type I; 
\item the following subheap is the northwest corner of $H(w)$.

%-- New mfpic environment, number 55 of 189. (size of end split: 2, should be 2)  ------------------->
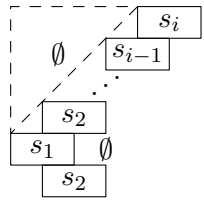
\begin{center}
\includegraphics{ThesisFigs1.055}
\end{center}

In particular, the entry labeled by $s_{i}$ in the heap above is not covered.
\end{enumerate}
\end{lemma}

\begin{proof}
The higher entry labeled by $s_{2}$ cannot be covered by an entry labeled by $s_{1}$; otherwise, we produce one of the impermissible configurations of Lemma \ref{impermissible.heap.configs} and violate $w$ being fully commutative.  Then the entry labeled by $s_{3}$ cannot be covered by an entry labeled by $s_{2}$; again, we would contradict Lemma \ref{impermissible.heap.configs}.  Iterating, we see that each entry on the diagonal of the subheap labeled by $s_{j}$, for $2 \leq j \leq i-1$, can only be covered by an entry labeled by $s_{j+1}$.  If $i<n+1$, we are done since the entry labeled by $s_{i}$ cannot be covered by an entry labeled by $s_{i-1}$.  Assume that $i=n+1$.  If the entry labeled by $s_{n+1}$ at the top of the diagonal in the subheap is covered by an entry labeled by $s_{n}$, then by Lemma \ref{is_zigzag}, $w$ is of type I.  If the entry labeled by $s_{n+1}$ is not covered, then we are done, as well.
\end{proof}

\begin{remark}\label{remark.typeI.or.nothing.above}
The previous lemma has versions corresponding to the southwest, northeast, and southeast corners of $H(w)$.  
\end{remark}

As stated above, the purpose of the previous three lemmas was to aid in the proof of the next important lemma.  

\begin{lemma}\label{main_zigzag_lemma}
Let $w \in W_{c}(\C_{n})$.  Suppose that there exists $i$ with $1<i < n+1$ such that $s_{i+1}$ does not occur between two consecutive occurrences of $s_{i}$ in $w$.    Then one or both of the following must be true about $w$:
\begin{enumerate}[label=\rm{(\roman*)}]
\item $w$ is of type I; 
\item $w$ is of the form
$$u\z^{L,1}_{i,i}v=u s_{i}s_{i-1}\cdots s_{2}s_{1}s_{2}\cdots s_{i-1}s_{i} v,$$
where $\supp(u), \supp(v) \subseteq \{s_{i+1}, s_{i+2}, \dots, s_{n}, s_{n+1}\}$.  
\end{enumerate}
In terms of heaps, if $H(w)$ has two consecutive occurrences of entries labeled by $s_{i}$ such that there is no entry labeled by $s_{i+1}$ occurring between them, then either $w$ is of type I or

%-- New mfpic environment, number 56 of 189. (size of end split: 2, should be 2)  ------------------->
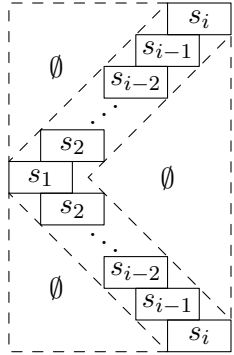
\begin{center}
\includegraphics{ThesisFigs1.056}
\end{center}

is a convex subheap of $H(w)$ and there are no other occurrences of entries labeled by $s_{1}, s_{2}, \dots, s_{i}$ in $H(w)$.
\end{lemma}

\begin{proof}
Choose the largest index $i$ with $1<i<n+1$ such that $s_{i+1}$ does not occur between two consecutive occurrences of $s_{i}$ in $w$.  By Lemma \ref{zigzag_subword}, $\z_{i,i}^{L,1}$ is a subword of some reduced expression for $w$ and 

%-- New mfpic environment, number 57 of 189. (size of end split: 2, should be 2)  ------------------->
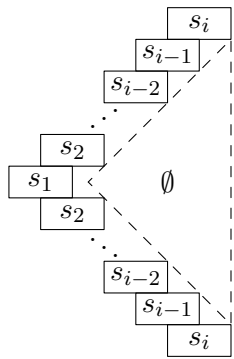
\begin{center}
\includegraphics{ThesisFigs1.057}
\end{center}

is a convex subheap of $H(w)$.  Let $k$ be largest index with $i\leq k \leq n+1$ such that each entry labeled by $s_{j}$ on the upper diagonal in the above subheap covers an entry labeled by $s_{j-1}$ for $j\leq k$.  Similarly, let $l$ be the largest index with $i\leq l \leq n+1$ such that each entry on the lower diagonal in the above subheap labeled by $s_{j}$ is covered by an entry labeled by $s_{j-1}$ for $j\leq l$.  Then 

%-- New mfpic environment, number 58 of 189. (size of end split: 2, should be 2)  ------------------->
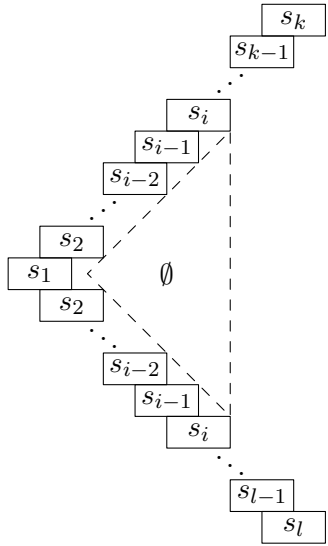
\begin{center}
\includegraphics{ThesisFigs1.058}
\end{center}

is a convex subheap of $H(w)$.  By the northwest and southwest versions of Lemma \ref{typeI.or.nothing.above}, 

%-- New mfpic environment, number 59 of 189. (size of end split: 2, should be 2)  ------------------->
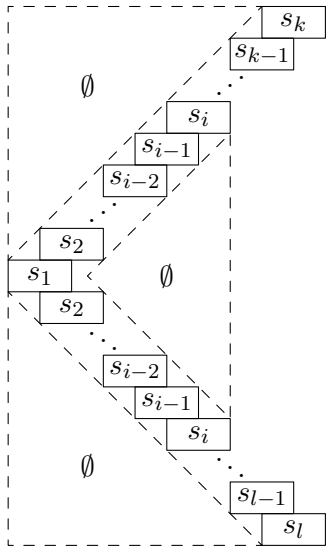
\begin{center}
\includegraphics{ThesisFigs1.059}
\end{center}

must be a convex subheap of $H(w)$.  If neither of $k$ or $l$ are equal to $n+1$, then we are done since the entry labeled by $s_{k}$ (respectively, $s_{l}$) cannot be covered by (respectively, cover) an entry labeled by $s_{k-1}$ (respectively, $s_{l-1}$); otherwise, we contradict $w \in W_{c}(\C_{n})$.  Assume that $k=n+1$; the case $l=n+1$ follows by a symmetric argument.  If the entry labeled by $s_{n+1}$ is covered, it must be covered by an entry labeled by $s_{n}$.  But by Lemma \ref{is_zigzag}, $w$ would then be of type I.
\end{proof}

\begin{remark}
Note that all of the previous lemmas of this section have analogous statements where $s_{1}$ and $s_{2}$ are replaced with $s_{n+1}$ and $s_{n}$, respectively.
\end{remark}

The next three lemmas are generalizations of Lemma 5.3 in \cite{Green.R:P}.   In both of these lemmas, we assume that $w \in W_{c}(\C_{n})$ is irreducible.

\bigskip

\begin{lemma}\label{weak.star.lemma.middle}
Let $w \in W_{c}(\C_{n})$ for $n \geq 4$ and suppose that $w$ is irreducible and not of type I.   Consider the canonical representation of $H(w)$.  If $s_{i} \in \r_{k+1}$ with $i \notin \{1,2,n,n+1\}$, then the entry labeled by $s_{i} \in \r_{k+1}$ is covered by entries labeled by $s_{i-1}$ and $s_{i+1}$, where at least one of these occurs in $\r_{k}$.
\end{lemma}

\begin{proof}
We proceed by induction on $k$.  For the base case, assume that $k=1$, so that $s_{i} \in \r_{2}$.  Then this entry is covered by at least one of $s_{i-1}$ or $s_{i+1}$.  Note that our restrictions on $i$ and $n$ force $m(s_{i}, s_{i-1})=m(s_{i}, s_{i+1})=3$.  This implies that both $s_{i-1}$ and $s_{i+1}$ occur in $\r_{1}$;  otherwise, $w$ is left weak star reducible.  For the inductive step, assume that the theorem is true for $2 \leq k' \leq k-1$ for some $k$.  Now, suppose that $s_{i} \in \r_{k+1}$ with $i \notin \{1,2, n, n+1\}$.  Then at least one of $s_{i-1}$ or $s_{i+1}$ occur in $\r_{k}$.  We consider two cases: (1) $i \notin \{1,2,3,n-1,n,n+1\}$ and (2) $i \in \{3,n-1\}$.

\bigskip

Case (1): Assume that $i \notin \{1,2,3,n-1,n,n+1\}$.  Observe that this forces $n \geq 6$.  Without loss of generality, assume that $s_{i-1} \in \r_{k}$.  The case with $s_{i+1} \in \r_{k}$ is symmetric since the restrictions on $i$ imply that we may apply the induction hypothesis to either $i-1$ or $i+1$.  By induction, the entry labeled by $s_{i-1}$ occurring in $\r_{k}$ is covered by an entry labeled by $s_{i-2}$ and an entry labeled by $s_{i}$.  This implies that the entry labeled by $s_{i}$ occurring in $\r_{k+1}$ must be covered by an entry labeled by $s_{i+1}$; otherwise, we produce a convex chain corresponding to the subword $s_{i}s_{i-1}s_{i}$.  This yields our desired result.

\bigskip

Case (2):  For the second case, assume that $i \in \{3, n-2\}$.  Without loss of generality, assume that $i=3$, the other case being similar.  Then $s_{3} \in \r_{k+1}$.  This entry is covered by either (a) an entry labeled by $s_{2}$, (b) an entry labeled by $s_{4}$, or (c) both.  If we are in situation (c), then we are done.  For sake of a contradiction, assume that exactly one of (a) or (b) occurs.

\bigskip

First, assume that (a) occurs, but (b) does not.  That is, an entry labeled by $s_{2}$ covers the entry labeled by $s_{3}$ that occurs in $\r_{k+1}$.  Meanwhile, the entry labeled by $s_{3}$ that occurs in $\r_{k+1}$ is not covered by an entry labeled by $s_{4}$.  This implies that $s_{2} \in \r_{k}$.  Since $k \geq 2$, the entry labeled by $s_{2}$ occurring in $\r_{k}$ is covered by at least one of $s_{1}$ or $s_{3}$.  Since $w$ is fully commutative, the entry labeled by $s_{2}$ occurring in $\r_{2}$ cannot be covered by an entry labeled by $s_{3}$.  Thus, $s_{1} \in \r_{k-1}$.  

\bigskip

For sake of a contradiction, assume that $k \neq 2$, so that $\r_{k-1}$ is not the top row of the canonical representation for $H(w)$.  Then we must have $s_{2} \in \r_{k-2}$.  Since $w$ is not left weak star reducible, we cannot have $k=3$; otherwise, $w$ is left weak star reducible by $s_{2}$ with respect to $s_{1}$.  So, $k>3$, which implies that the entry labeled by $s_{2}$ occurring in $\r_{k-2}$ is covered.  This entry cannot be covered by $s_{1}$ since $w$ is fully commutative.  Therefore, we have $s_{3} \in \r_{k-3}$.  But by induction, this entry is covered by an entry labeled by $s_{2}$ and an entry labeled by $s_{4}$.  But this produces a convex chain that corresponds to the subword $s_{2}s_{3}s_{2}$, which contradicts $w$ being fully commutative.  We have just shown that $k=2$.  

\bigskip

Thus far, we have $s_{1} \in \r_{1}$, $s_{2} \in \r_{2}$, and $s_{3} \in \r_{3}$, while $s_{4}$ does not cover the entry labeled by $s_{3}$ occurring in $\r_{3}$ and $s_{3}$ does not cover the entry labeled by $s_{2}$ occurring in $\r_{2}$.  Since $w$ is not left weak star reducible, $s_{1} \notin \r_{3}$.  Since $w$ is not right weak star reducible, the entry labeled by $s_{3}$ occurring in $\r_{3}$ must cover an entry.  However, since $w$ is fully commutative, we cannot have $s_{2} \in \r_{4}$, and so, we must have $s_{4} \in \r_{4}$.  Continuing with similar reasoning, we see that

%-- New mfpic environment, number 60 of 189. (size of end split: 2, should be 2)  ------------------->
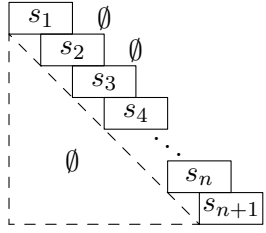
\begin{center}
\includegraphics{ThesisFigs1.060}
\end{center}

is a saturated subheap of $H(w)$, where $s_{j}$ occurs in $\r_{j}$.  

\bigskip

For sake of a contradiction, assume that the occurrence of $s_{n}$ in $\r_{n}$ is covered by an entry labeled by $s_{n+1}$.  In this case, the occurrence of $s_{n+1} \in \r_{n+1}$ cannot cover any other entry, namely one labeled by $s_{n}$, since $w$ is fully commutative; otherwise, we produce a convex chain that corresponds to the subword $s_{n+1}s_{n}s_{n+1}s_{n}$.  But then $w$ is right weak star reducible by $s_{n+1}$ with respect to $s_{n}$, which contradicts $w$ being irreducible.  Therefore, the entry labeled by $s_{n}$ occurring in $\r_{n}$ is only covered by the entry labeled by $s_{n-1}$ occurring in $\r_{n-1}$.    

\bigskip

Since $w$ is fully commutative and we must avoid the convex chains of Lemma \ref{impermissible.heap.configs}, we can quickly conclude that

%-- New mfpic environment, number 61 of 189. (size of end split: 2, should be 2)  ------------------->
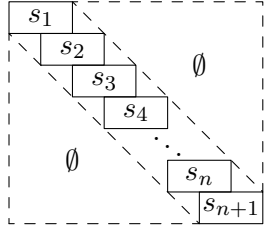
\begin{center}
\includegraphics{ThesisFigs1.061}
\end{center}

is the top $n+1$ rows of the canonical representation of $H(w)$.  Since $w$ is not of type I, this cannot be all of $w$.  The only possibility is that $s_{n} \in \r_{n+2}$.  Since $w$ is not right weak star reducible, this cannot be all of $H(w)$ either.  So, at least one of $s_{n-1}$ or $s_{n+1}$ occur in $\r_{n+3}$.  We cannot have $s_{n+1} \in \r_{n+3}$ because $w$ is fully commutative.  Thus, $s_{n-1} \in \r_{n+3}$ while $s_{n+1}$ is not.  Again, since $w$ is not right weak star reducible, we must have $s_{n-2} \in \r_{n+4}$ while $s_{n} \notin \r_{n+4}$.  Continuing with similar reasoning, we quickly see that

%-- New mfpic environment, number 62 of 189. (size of end split: 2, should be 2)  ------------------->
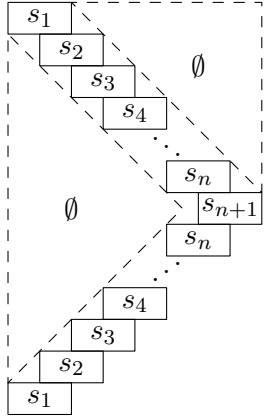
\begin{center}
\includegraphics{ThesisFigs1.062}
\end{center}

is a convex subheap of $H(w)$.  But then by Lemma \ref{is_zigzag}, $w$ is of type I, which is a contradiction.  Therefore, we cannot have possibility (a) occurring while (b) does not.

\bigskip

The only remaining possibility is that (b) occurs, but (a) does not.  That is, $s_{3} \in \r_{k+1}$ and $s_{4} \in \r_{k}$, while the entry labeled by $s_{3}$ occurring in $\r_{k+1}$ is not covered by an entry labeled by $s_{2}$.  Observe that the case $n=4$ is covered by an argument that is symmetric to the argument made above when we assumed that (a) occurs, but (b) does not, where we take $i=n-1$ instead of $i=3$.  So, assume that $n>4$.  Then by induction, we must have the entry labeled by $s_{4}$ occurring in $\r_{k}$ covered by an entry labeled by $s_{3}$ (and by an entry labeled by $s_{5}$).  But then we produce a convex chain corresponding to the subword $s_{3}s_{4}s_{3}$, which contradicts $w$ being fully commutative.  Therefore, if $i=3$, then the entry labeled by $s_{i}$ occurring in $\r_{k+1}$ is covered by an entry labeled by $s_{i-1}=s_{2}$ and by an entry labeled by $s_{i+1}=s_{4}$, as desired.

\bigskip

We have exhausted all possibilities, and hence we have our desired result.
\end{proof}

\begin{lemma}\label{weak.star.lemma.end.n=3}
Let $w \in W_{c}(\C_{3})$ and suppose that $w$ is irreducible and not of type I.   Consider the canonical representation of $H(w)$.  Then
\begin{enumerate}[label=\rm{(\roman*)}]
\item if $s_{2} \in \r_{k+1}$ is covered by an entry labeled by $s_{3}$, then an entry labeled by $s_{1}$ covers the entry labeled by $s_{2} \in \r_{k+1}$;
\item if $s_{3} \in \r_{k+1}$ is covered by an entry labeled by $s_{2}$, then an entry labeled by $s_{4}$ covers the entry labeled by $s_{3} \in \r_{k+1}$.
\end{enumerate}
\end{lemma}

\begin{proof}
We prove (i); (ii) follows by a symmetric argument.  Note that since $n=3$, $m(s_{2},s_{3})=3$ and $m(s_{3},s_{4})=4$.  If $k=1$, then $s_{2} \in \r_{2}$ and $s_{3} \in \r_{1}$.    This implies that $s_{1}$ must occur in $\r_{1}$, otherwise, $w$ is left weak star reducible by $s_{3}$ with respect to $s_{2}$.  So, assume that $k \geq 2$.  For sake of a contradiction, assume that the entry labeled by $s_{2}$ occurring in $\r_{k+1}$ is not covered by an entry labeled by $s_{1}$.  This forces $s_{3} \in \r_{k}$.  Then at least one of $s_{2}$ or $s_{4}$ must cover the entry labeled by $s_{3}$ occurring in $\r_{k-1}$.  Since $s_{1}$ does not cover the occurrence of $s_{2} \in \r_{k+1}$ and $w$ is fully commutative, it must be the case that $s_{4} \in \r_{k-1}$, while the entries labeled by $s_{3}$ and $s_{2}$ occurring in $\r_{k}$ and $\r_{k+1}$, respectively, are not covered by entries labeled by $s_{2}$ and $s_{1}$, respectively.  Then

%-- New mfpic environment, number 63 of 189. (size of end split: 2, should be 2)  ------------------->
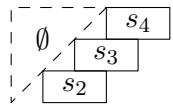
\begin{center}
\includegraphics{ThesisFigs1.063}
\end{center}

is a convex subheap of $H(w)$.  We consider two cases: (1) $k=2$ and (2) $k>2$.

\bigskip

Case (1):  First, assume that $k=2$, so that $s_{2} \in \r_{3}$, $s_{3} \in \r_{2}$, $s_{4} \in \r_{1}$, and neither $s_{1}$ nor $s_{2}$ occur in $\r_{1}$ or $\r_{2}$.  Then the subheap immediately above is the northwest corner of $H(w)$, where the entry labeled by $s_{4}$ occurs in the top row.  We are assuming that $w$ is not of type I, so this cannot be all of $H(w)$.  We cannot have $s_{4} \in \r_{3}$; otherwise, $w$ would be left weak star reducible by $s_{4}$ with respect to $s_{3}$.  The entry labeled by $s_{2}$ occurring in $\r_{3}$ must cover at least one entry labeled by $s_{1}$ or $s_{3}$ in $\r_{4}$.  But it cannot be an entry labeled by $s_{3}$ since we would produce a convex chain corresponding to the subword $s_{3}s_{2}s_{3}$.  This implies that 

%-- New mfpic environment, number 64 of 189. (size of end split: 2, should be 2)  ------------------->
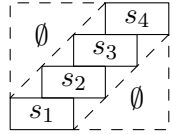
\begin{center}
\includegraphics{ThesisFigs1.064}
\end{center}

is the top of $H(w)$.  This cannot be all of $H(w)$ since we are assuming that $w$ is not of type I.  So, we must have $s_{2} \in \r_{5}$.  Note that this must be all of $\r_{5}$.  Again, this cannot be all of $H(w)$.  So, we must have $s_{1}$ or $s_{3}$ occurring in $\r_{6}$.  Since $w$ is fully commutative, we cannot have $s_{1}$ occurring in $\r_{6}$, and so $s_{3} \in \r_{6}$.  Yet again, this cannot be all of $H(w)$.  At least one of $s_{2}$ or $s_{4}$ must occur in $\r_{7}$, but it cannot be $s_{2}$; otherwise, we produce a convex chain corresponding to the subword $s_{2}s_{3}s_{2}$.  Therefore, the convex subheap

%-- New mfpic environment, number 65 of 189. (size of end split: 2, should be 2)  ------------------->
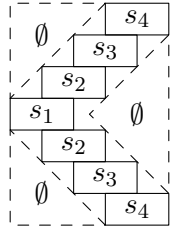
\begin{center}
\includegraphics{ThesisFigs1.065}
\end{center}

is the top of $H(w)$.  But then by Lemma \ref{is_zigzag}, $w$ must be of type I, which is a contradiction.   This completes the case $k=2$. 

\bigskip

Case (2):  Now, assume that $k>2$.  Recall that we are assuming that $s_{4} \in \r_{k-1}$, but that the entries labeled by $s_{3}$ and $s_{2}$ occurring in $\r_{k}$ and $\r_{k+1}$, respectively, are not covered by entries labeled by $s_{2}$ and $s_{1}$, respectively.  Then the entry in $\r_{k-1}$ labeled by $s_{4}$ must be covered by an entry labeled by $s_{3}$.  This implies that $s_{4}$ cannot occur in $\r_{k+1}$ since $w$ is fully commutative.  Then

%-- New mfpic environment, number 66 of 189. (size of end split: 2, should be 2)  ------------------->
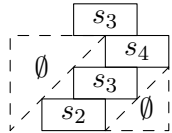
\begin{center}
\includegraphics{ThesisFigs1.066}
\end{center}

is a convex subheap of $H(w)$.  This cannot be all of $H(w)$ since $w$ would then be left weak star reducible by $s_{3}$ with respect to $s_{4}$.  So, the entry labeled by $s_{3}$ in $\r_{k-2}$ must be covered by an entry labeled by $s_{2}$ (this entry cannot be covered by an entry labeled by $s_{4}$ since $w$ is fully commutative).  Again, since $w$ is fully commutative, but not left weak star reducible, the entry labeled by $s_{2}$ occurring in $\r_{k-3}$ must be covered by an entry labeled by $s_{1}$.  This implies that

%-- New mfpic environment, number 67 of 189. (size of end split: 2, should be 2)  ------------------->
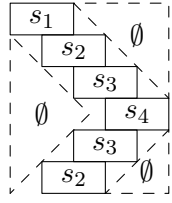
\begin{center}
\includegraphics{ThesisFigs1.067}
\end{center}

is a convex subheap of $H(w)$.  Since $w$ is not of type I, this cannot be all of $w$.  Since $w$ is fully commutative, the only two possibilities are: (a) the entry labeled by $s_{1}$ in $\r_{k-4}$ is covered by an entry labeled by $s_{2}$ and (b) the entry labeled by $s_{2}$ in $\r_{k+1}$ covers an entry labeled by $s_{1}$.  In either case, $w$ is of type I by Lemma \ref{is_zigzag}, which is a contradiction.

\bigskip

We have exhausted all possibilities.  Therefore, it must be the case that the entry labeled by $s_{2} \in \r_{k+1}$ is covered by an entry labeled by $s_{1}$, as desired.
\end{proof}

The next lemma is similar to the previous, except that we assume that $n\geq 4$.  We have separated these two lemmas because their proofs are different.

\begin{lemma}\label{weak.star.lemma.end.n>=4}
Let $w \in W_{c}(\C_{n})$ for $n \geq 4$ and suppose that $w$ is irreducible and not of type I.   Consider the canonical representation of $H(w)$.  Then
\begin{enumerate}[label=\rm{(\roman*)}]
\item if $s_{2} \in \r_{k+1}$ is covered by an entry labeled by $s_{3}$, then an entry labeled by $s_{1}$ covers the entry labeled by $s_{2} \in \r_{k+1}$;
\item if $s_{n} \in \r_{k+1}$ is covered by an entry labeled by $s_{n-1}$, then an entry labeled by $s_{n+1}$ covers the entry labeled by $s_{n} \in \r_{k+1}$.
\end{enumerate}
\end{lemma}

\begin{proof}
We prove (i); (ii) follows by a symmetric argument.  Note that since $n\geq 4$, $m(s_{2},s_{3})=3$.  As in the proof of the previous lemma, if $k=1$, then $s_{2} \in \r_{2}$ and $s_{3} \in \r_{1}$.    This implies that $s_{1}$ must occur in $\r_{1}$, otherwise, $w$ is left weak star reducible by $s_{3}$ with respect to $s_{2}$.  Now, assume that $k \geq 2$, so that $s_{2} \in \r_{k+1}$ and this entry is covered by an entry labeled by $s_{3}$.  Then by Lemma \ref{weak.star.lemma.middle}, entries labeled by $s_{2}$ and $s_{4}$ cover the entry labeled by $s_{3}$ occurring in $\r_{k}$.  Since $w$ is fully commutative, we must have the entry labeled by $s_{2}$ occurring in $\r_{k+1}$ covered by an entry labeled by $s_{1}$, as desired; otherwise, we produce one of the impermissible configurations of Lemma \ref{impermissible.heap.configs} and violate $w$ being fully commutative.
\end{proof}

\begin{remark}
Lemmas \ref{weak.star.lemma.middle}, \ref{weak.star.lemma.end.n=3}, and \ref{weak.star.lemma.end.n>=4} all have ``upside-down'' versions, where we replace $k+1$ with $k-1$ and we swap the phrases ``is covered by'' and ``covers.'' 
\end{remark}

\end{section}

\begin{section}{Proof of the classification of the irreducible elements in $W(B_{n})$}

Finally, we are ready to prove the classification of the irreducible elements in $W(B_{n})$.

\bigskip

\textit{Proof of Theorem \ref{Bwsrm}.}  First, observe that if $w$ is irreducible in $W(B_{n'})$ for $n' \leq n$, then $w$ is also irreducible in $W(B_{n})$ when considered as an element of the larger group.  Also, we see that every element on our list is, in fact, irreducible.  It remains to show that our list is complete. We induct on $n$.  

\bigskip

For the base case, consider $n=2$.  An exhaustive check verifies that the only irreducible elements in $W(B_{2})$ are $s_{1}$, $s_{2}$, $s_{1}s_{2}$, and $s_{2}s_{1}$, which agrees with the statement of the theorem.  

\bigskip

For the inductive step, assume that for all $n' \leq n-1$, our list is complete.  Let $w \in W_{c}(B_{n})$ and assume that $w$ is irreducible.  For sake of a contradiction, assume that $w$ is not on our list.  First, we argue that $\supp(w)$ must be all of $S(B_{n})$.  Suppose that there exists $s \notin \supp(w)$.  Say $s=s_{i}$.  Then we can factorize $w$ as $w=uv$ (reduced), where $\supp(u) \subseteq \{s_{1}, \dots, s_{i-1}\}$ and $\supp(v) \subseteq \{s_{i+1},\dots, s_{n}\}$.  Note that since $w$ is irreducible, both $u$ and $v$ are irreducible, as well.  In particular, $u$ is irreducible in $W(B_{i-1})$.  By induction, $u$ is on our list.  Since no generator occurring in $v$ involves a bond of strength 4, all weak star reductions are equivalent to ordinary star reductions.  Then by \cite[Theorem 6.3]{Green.R:P}, $v$ is equal to a product of commuting generators (where we consider $v$ as an element of a type $A$ Coxeter group).  Since $\supp(u) \cap \supp(v) = \emptyset$, $w$ must already be on our list, which contradicts our assumption that it is not.  So, we must have $\supp(w)=S(B_{n})$.  In particular, $s_{n-1}$ and $s_{n}$ occur in $\supp(w)$.  

\bigskip

Next, we argue that there can only be a single occurrence of $s_{n}$.  For sake of a contradiction, assume that there are at least two occurrences of $s_{n}$ in any reduced expression for $w$.  Consider any two consecutive occurrences of $s_{n}$ in the canonical representation of $H(w)$.  Since there is no generator $s_{n+1}$, we can apply Lemma \ref{main_zigzag_lemma} and conclude that 

%-- New mfpic environment, number 68 of 189. (size of end split: 2, should be 2)  ------------------->
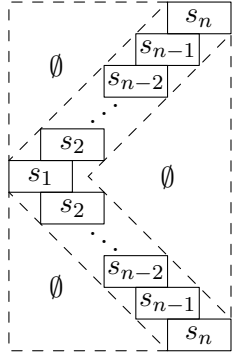
\begin{center}
\includegraphics{ThesisFigs1.068}
\end{center}

is a convex subheap of $H(w)$ and there are no other occurrences of entries labeled by $s_{1}, s_{2}, \dots, s_{n}$ in $H(w)$.  But then the heap above must be all of $H(w)$.  We see that $w$ is left weak star reducible by $s_{n}$ with respect to $s_{n-1}$, which contradicts $w$ being irreducible.  Thus, there must be a single occurrence of $s_{n}$ in any reduced expression for $w$.  

\bigskip

Assume that the canonical representation for $H(w)$ has $m$ rows and suppose that the unique occurrence of $s_{n}$ is in $\r_{k}$.  At this point, we will consider three cases: (1) $k=1$, (2) $k=2$, and (3) $k >2$.

\bigskip

Case (1):  First, assume that $k=1$, so that $s_{n} \in \r_{1}$.  Since $\supp(w)=S(B_{n})$, the entry labeled by $s_{n}$ must cover an entry labeled by $s_{n-1}$.  But since there is no generator $s_{n+1}$ available, we contradict the upside-down version of Lemma \ref{weak.star.lemma.end.n=3} if $n=3$ and Lemma \ref{weak.star.lemma.end.n>=4} if $n \geq 4$.

\bigskip

Case (2):  For the second case, assume that $k=2$, so that $s_{n} \in \r_{2}$.  Then we must have $s_{n-1} \in \r_{1}$.  But this implies that $w$ is left weak star reducible by $s_{n-1}$ with respect to $s_{n}$, which contradicts $w$ being irreducible.

\bigskip

Case (3):  For the final case, assume that $k>2$.  For this case, we consider two separate subcases: (a) $n=3$ and (b) $n \geq 4$.

\bigskip

(a): Suppose that $n=3$.  In this case, there is a unique entry of $H(w)$ labeled by $s_{3}$, which occurs in $\r_{k}$ for $k>2$.  Then we must have $s_{2} \in \r_{k-1}$.  Then $s_{1} \in \r_{k-2}$ since this is the only generator that is available to cover an entry labeled by $s_{2}$.  By the upside-down version of Lemma \ref{weak.star.lemma.end.n=3}, the entry labeled by $s_{2}$ occurring in $\r_{k-1}$ must cover an entry labeled by $s_{1}$.  Then 

%-- New mfpic environment, number 69 of 189. (size of end split: 2, should be 2)  ------------------->
\begin{center}
\includegraphics{ThesisFigs1.069}
\end{center}

is a convex subheap of $H(w)$, where we remind the reader that there is a unique occurrence of an entry in $H(w)$ labeled by $s_{3}$.  Since $w$ is fully commutative, we cannot have the higher occurrence of $s_{1}$ above covered an entry labeled by $s_{2}$; otherwise, we violate $w$ being fully commutative.  But then $w$ is left weak star reducible by $s_{1}$ with respect to $s_{2}$, which contradicts $w$ being irreducible.  That is, if $n=3$, our list is complete.

\bigskip

(b): Now, suppose that $n\geq 4$.  We have $s_{n} \in \r_{k}$ with $k>2$ and $n\geq 4$.  Then $s_{n-1} \in \r_{k-1}$.  By Lemma \ref{weak.star.lemma.middle}, we must have the entry labeled by $s_{n-1}$ occurring in $\r_{k-1}$ covered by both an entry labeled by $s_{n-1}$ and an entry labeled by $s_{n}$.  But this contradicts $w$ having a unique occurrence of $s_{n}$.  Thus, if $n \geq 4$, our list is complete.

\bigskip

Therefore, there are no irreducible elements in $W(B_{n})$ with support equal to all of $S(B_{n})$, and so our list must be complete.  

\hfill $\qed$

\end{section}

\end{chapter}

%%%%%%%%%%% Chapter 5 %%%%%%%%%%%%

\begin{chapter}{Classification of the weak star irreducible elements in $W(\C_{n})$}\label{chaptCwsrm}

The goal of this chapter is to classify the irreducible elements of $W(\C_{n})$. 

\begin{section}{Statement of theorem}

The following theorem is a new result.

\begin{theorem}\label{affineCwsrm}
Let $w \in W_{c}(\C_{n})$.  Then $w$ is irreducible if and only if $w$ is equal to one of the elements on the following list.
\begin{enumerate}[label=\rm{(\roman*)}]
\item $uv$, where $u$ is a type $B$ irreducible element and $v$ is a type $B'$ irreducible element with $\supp(u)\cap \supp(v)=\emptyset$;
\item $\z^{R,2k}_{1,1}$, $\z^{L,2k}_{n+1,n+1}$, $\z^{L,2k+1}_{n+1,1}$, and $\z^{R,2k+1}_{1,n+1}$; 
\item Any type II element.
\end{enumerate}
\end{theorem}

\begin{remark}
The elements listed in (i) include all possible products of commuting generators.  This includes $\x_{\O}$ and $\x_{\E}$, which are also included in (iii).  The elements listed in (ii) are the type I elements with left and right descent sets equal to either $s_{1}$ or $s_{n+1}$.
\end{remark}

As with the proof of the classification of the irreducible elements of $W(B_{n})$, we require a few lemmas.

\end{section}

\begin{section}{More preparatory lemmas}

In this section, we state and prove a few more technical lemmas that will be used to prove Theorem \ref{affineCwsrm}.  Note that all statements about $W(\C_{n})$ also apply to $W(B_{n})$ and $W(B'_{n})$.  Also, recall that if $w \in W_{c}(\C_{n})$, then $\r_{k}$ is the $k$th row of the canonical representation of $H(w)$.

\begin{lemma}\label{sandwich.stack.lemma.n=2}
Let $w \in W_{c}(\C_{2})$ and suppose that $w$ is irreducible.  Consider the canonical representation of $H(w)$.  If $w$ is not of type I, then $\r_{k+1}=\x_{\O}$ (respectively $\x_{\E}$) implies $\r_{k}=\x_{\E}$ (respectively $\x_{\O}$).
\end{lemma}

\begin{proof}
Note that when $n=2$, we have $\x_{\O}=s_{1}s_{3}$ and $\x_{\E}=s_{2}$.  If $\r_{k+1}=s_{1}s_{3}$, then it is clear that $\r_{k}=s_{2}$.  Now, assume that $\r_{k+1}=s_{2}$.  Then at least one of $s_{1}$ or $s_{3}$ occurs in $\r_{k}$.  For sake of a contradiction, assume that only one of these occurs in $\r_{k}$, and without loss of generality, assume that $s_{1} \in \r_{k}$ while $s_{3} \notin \r_{k}$ (this makes sense since $n=2$).  We consider two cases: (1) $k=1$ and (2) $k \geq 2$.

\bigskip

Case (1):  First, assume that $k=1$, so that $s_{1} \in \r_{1}$ and $s_{2} \in \r_{2}$, while $s_{3} \notin \r_{1}$.  This cannot be all of $H(w)$ since $w$ is not of type I.  So, the occurrence of $s_{2}$ in $\r_{2}$ must cover an entry labeled by either $s_{1}$ or $s_{3}$.  Since $w$ is not left weak star reducible, we cannot have the occurrence of $s_{2}$ in $\r_{2}$ covering an entry labeled by $s_{1}$.  So, we must have $s_{3} = \r_{3}$.  Thus far, the top three rows of the canonical representation for $H(w)$ are as follows.

%-- New mfpic environment, number 70 of 189. (size of end split: 2, should be 2)  ------------------->
\begin{center}
\includegraphics{ThesisFigs1.070}
\end{center}

Since $w$ is not of type I, this cannot be all of $H(w)$.  Then $s_{2} \in \r_{4}$.  Again, this cannot be all of $H(w)$ since $w$ is not of type I.  So, we must have at least one of $s_{1}$ or $s_{3}$ occurring in $\r_{5}$.  However, since $w$ is fully commutative, $s_{3} \notin \r_{5}$, and so $s_{1} \in \r_{5}$.  But then the top five rows of the canonical representation for $H(w)$ are as follows.

%-- New mfpic environment, number 71 of 189. (size of end split: 2, should be 2)  ------------------->
\begin{center}
\includegraphics{ThesisFigs1.071}
\end{center}

According to Lemma \ref{is_zigzag}, $w$ is of type I, which is a contradiction.

\bigskip

Case (2):  For the second case, assume that $k\geq 2$.  Then we must have $s_{2} \in \r_{k-1}$.  Since $w$ is not left weak star reducible, we cannot have $k=2$; otherwise, $w$ is left weak star reducible by $s_{2}$ with respect to $s_{1}$.  Thus, $k>2$, and hence at least one of $s_{1}$ or $s_{3}$ occurs in $\r_{k-2}$.  Since $w$ is fully commutative, we cannot have $s_{1} \in \r_{k-2}$, and so, $s_{3} \in \r_{k-2}$.  Thus far, 

%-- New mfpic environment, number 72 of 189. (size of end split: 2, should be 2)  ------------------->
\begin{center}
\includegraphics{ThesisFigs1.072}
\end{center}

is a convex subheap of $H(w)$.  This cannot be all of $H(w)$ since $w$ is not of type I.  The only possibilities are that $s_{2} \in \r_{k-3}$ or $s_{3} \in \r_{k+2}$ (both possibilities could occur simultaneously).  In either case, $w$ must be of type I by Lemma \ref{is_zigzag}, which is a contradiction.

\bigskip

Therefore, $\r_{k+1}=\x_{\O}$ (respectively $\x_{\E}$) implies $\r_{k}=\x_{\E}$ (respectively $\x_{\O}$).
\end{proof}

The next lemma is of a similar flavor as the previous, except here we take $n=3$.  Also, observe that it is unnecessary to require $w$ to not be of type I.

\begin{lemma}\label{sandwich.stack.lemma.n=3}
Let $w \in W_{c}(\C_{3})$ and suppose that $w$ is irreducible.  Consider the canonical representation of $H(w)$.  Then $\r_{k+1}=\x_{\O}$ (respectively $\x_{\E}$) implies $\r_{k}=\x_{\E}$ (respectively $\x_{\O}$).
\end{lemma}

\begin{proof}
Note that when $n=3$, we have $\x_{\O}=\x_{1,3}=s_{1}s_{3}$ and $\x_{\E}=\x_{2,4}=s_{2}s_{4}$.  Assume that $\r_{k+1}=\x_{1,3}$; the proof of the other case is similar.  Then we must have $s_{2} \in \r_{k}$ since this is the only generator available to cover $s_{1} \in \r_{k+1}$.  By Lemma \ref{weak.star.lemma.end.n=3}, the entry labeled by $s_{3} \in \r_{k+1}$ must be covered by an entry labeled by $s_{4}$.  If $k=1$, then $s_{4} \in \r_{1}=\r_{k}$, as desired.  Assume that $k>1$.  Then at least one of $s_{1}$ or $s_{3}$ occurs in $\r_{k-1}$.  For sake of a contradiction assume that $s_{1} \in \r_{k-1}$, but $s_{3} \notin \r_{k-1}$.  If $k=2$, then $w$ would be left weak star reducible by $s_{1}$ with respect to $s_{2}$.  So, we must have $k>2$, in which case, $s_{2} \in \r_{k-2}$.  But then we have a convex chain corresponding to the subword $s_{2}s_{1}s_{2}s_{1}$, which contradicts $w$ being fully commutative.  Thus,  $s_{3} \in \r_{k-1}$.  If $s_{3} \in \r_{k-1}$, then the entry labeled by $s_{4}$ that covers $s_{3} \in \r_{k+1}$ must occur in $\r_{k}$.  So, $\r_{k}=\x_{\E}$, as desired.
\end{proof}

\begin{lemma}\label{sandwich.stack.lemma.n=4}
Let $w \in W_{c}(\C_{4})$ and suppose that $w$ is irreducible with $\supp(w)=S(\C_{4})$.  Consider the canonical representation of $H(w)$.  Then $\r_{k+1}=\x_{\O}$ (respectively $\x_{\E}$) implies $\r_{k}=\x_{\E}$ (respectively $\x_{\O}$).
\end{lemma}

\begin{proof}
Note that when $n=4$, we have $\x_{\O}=\x_{1,5}=s_{1}s_{3}s_{5}$ and $\x_{\E}=\x_{2,4}=s_{2}s_{4}$.   We consider two cases: (1) $\r_{k+1}=\x_{\O}$ and (2) $\r_{k+1}=\x_{\E}$.

\bigskip

Case (1):  First, assume that $\r_{k+1}=\x_{\O}$.  By Lemma \ref{weak.star.lemma.middle}, the entry labeled by $s_{3} \in \r_{k+1}$ is covered by labeled by $s_{2}$ and $s_{4}$, where at least one of these occurs in $\r_{k}$.  Since $s_{2}$ (respectively, $s_{4}$) is the only generator that may cover an entry labeled by $s_{1}$ (respectively, $s_{5}$), we must have both $s_{2}$ and $s_{4}$ occurring in $\r_{k}$, as desired.

\bigskip

Case (2):  Now, assume that $\r_{k+1}=\x_{\E}$.  First, we argue that an entry labeled by $s_{3}$ covers the occurrences of $s_{2}$ and $s_{4}$ in $\r_{k+1}$.  For sake of a contradiction, assume otherwise.  Then we must have $s_{1}$ and $s_{5}$ both occurring in $\r_{k}$.  

\bigskip

Assume that $k=1$, so that $\r_{k}=\r_{1}=s_{1}s_{5}$ and $\r_{k+1}=\r_{2}=s_{2}s_{4}$.  We cannot have $s_{1}$ (respectively, $s_{5}$) occurring in $\r_{k+2}$; otherwise $w$ would be left weak star reducible by $s_{1}$ (respectively $s_{5}$) with respect to $s_{2}$ (respectively, $s_{4}$).  Since $\supp(w)=S(\C_{4})$, we must have $s_{3} \in \r_{k+2}=\r_{3}$.  By the upside-down version of  Lemma \ref{weak.star.lemma.middle}, the entry labeled by $s_{3} \in \r_{3}$ must cover entries labeled by $s_{2}$ and $s_{4}$.  But this produces convex chains corresponding to the subwords $s_{2}s_{3}s_{2}$ and $s_{4}s_{3}s_{4}$, which contradicts $w$ being fully commutative.  

\bigskip

Now, assume that $k\geq 2$.  Then we must have $s_{2}$ and $s_{4}$ occurring in $\r_{k-1}$.  Since $w$ is not left weak star reducible, we must have $k > 2$; otherwise, $w$ is left weak star reducible by $s_{2}$ (respectively, $s_{4}$) with respect to $s_{1}$ (respectively, $s_{5}$).  We cannot have the entry labeled by $s_{2}$ (respectively, $s_{4}$) occurring in $\r_{k-1}$ covered by $s_{1}$ (respectively, $s_{5}$); otherwise, we produce one of the impermissible configurations of Lemma \ref{impermissible.heap.configs} and violate $w$ being fully commutative.  So, we must have $s_{3} \in \r_{k-2}$.  If $k=3$, then $w$ would be left weak star reducible by $s_{3}$ with respect to either $s_{2}$ or $s_{4}$.  Thus, $k \geq 4$.  By Lemma \ref{weak.star.lemma.middle}, the entry labeled by $s_{3} \in \r_{k-2}$ is covered by entries labeled by $s_{2}$ and $s_{4}$.  But then we produce convex chains corresponding to $s_{2}s_{3}s_{2}$ and $s_{4}s_{3}s_{4}$, which contradicts $w$ being fully commutative.

\bigskip

We have shown that if $\r_{k+1}=\x_{\E}$, then the entries labeled by $s_{2}$ and $s_{4}$ occurring in $\r_{k+1}$ must be covered by an entry labeled by $s_{3}$.  By Lemma \ref{weak.star.lemma.end.n>=4}, an entry labeled by $s_{1}$ (respectively, $s_{5}$) covers the entry labeled by $s_{2}$ (respectively, $s_{4}$) occurring in $\r_{k+1}$.   If $\r_{k} \neq s_{1}s_{3}s_{5}$, we quickly contradict Lemma \ref{impermissible.heap.configs} or Lemma \ref{weak.star.lemma.middle}.  Therefore, we must have $\r_{k}=\x_{\O}$, as desired.
\end{proof}

\begin{lemma}\label{sandwich.stack.lemma.n>4}
Let $w \in W_{c}(\C_{n})$ for $n > 4$ and suppose that $w$ is irreducible.  Consider the canonical representation of $H(w)$.  Then $\r_{k+1}=\x_{\O}$ (respectively $\x_{\E}$) implies $\r_{k}=\x_{\E}$ (respectively $\x_{\O}$).
\end{lemma}

\begin{proof}
If $k=1$, then the result follows by Lemmas \ref{weak.star.lemma.middle} and \ref{weak.star.lemma.end.n>=4}.  If $k>1$, the result follows by making repeated applications of Lemmas \ref{weak.star.lemma.middle} and \ref{weak.star.lemma.end.n>=4} while avoiding the convex chains of Lemma \ref{impermissible.heap.configs}.
\end{proof}

\begin{remark}
Lemmas \ref{sandwich.stack.lemma.n=2}, \ref{sandwich.stack.lemma.n=3}, \ref{sandwich.stack.lemma.n=4} and \ref{sandwich.stack.lemma.n>4} all have ``upside-down'' versions since all of the arguments reverse nicely.  That is, if $w$ is irreducible (and not of type I when $n=2$), then $\r_{k}=\x_{\O}$ (respectively, $\x_{\E}$) implies $\r_{k+1}=\x_{\E}$ (respectively, $\x_{\O}$).
\end{remark}

\end{section}

\begin{section}{Proof of the classification of the irreducible elements in $W(\C_{n})$}

We are now ready to prove the classification of the irreducible elements in $W(\C_{n})$.

\bigskip

\textit{Proof of Theorem \ref{affineCwsrm}.}  It is easily seen that every element on our list is irreducible.  For sake of a contradiction, assume that there exists $w \in W_{c}(\C_{n})$ such that $w$ is irreducible, but not on our list.  If there exists $s \notin \supp(w)$, then $w$ is equal to $uv$ (reduced), where $u$ is of type $B$ and $v$ is of type $B'$ and $\supp(u) \cap \supp(v) = \emptyset$.  Since $w$ is irreducible, both $u$ and $v$ are irreducible.  In this case, $w$ must be one of the elements from (i).  So, if $w$ is not on our list, we must have $\supp(w)=S(\C_{n})$.  This implies that $w$ is not a product of commuting generators.  According to Proposition \ref{zigzags}, the only irreducible elements with $n(w)=1$ are already listed in (ii).  Hence $n(w)>1$ (i.e., $w$ is not of type I).  

\bigskip

Now, consider the canonical representation of $H(w)$ and suppose that it has $m$ rows.  Since $w$ is not a product of commuting generators, $m\geq 2$.  We consider three main cases: (1) $n=2$, (2) $n=3$, and (3) $n \geq 4$.

\bigskip

Case (1):  Assume that $n=2$.  In this case, $\x_{\O}=s_{1}s_{3}$ and $\x_{\E}=s_{2}$.  First, we argue that $\r_{m}=\x_{\O}$ or $\x_{\E}$.  If $s_{2} \in \r_{m}$, then $\r_{m}=s_{2}=\x_{\E}$.  Assume that $s_{2} \notin \r_{m}$.  Then at least one of $s_{1}$ or $s_{3}$ occurs in $\r_{m}$.  Then we must have $s_{2} \in \r_{m-1}$.  In fact, $\r_{m-1}=s_{2}=\x_{\E}$.  By the upside-down version of Lemma \ref{sandwich.stack.lemma.n=2}, $\r_{m}=\x_{\O}$.  We have shown that $\r_{m}=\x_{\O}$ or $\x_{\E}$.  Now, by repeated applications of Lemma \ref{sandwich.stack.lemma.n=2}, $w$ must be equal to an alternating product of $\x_{\O}$ and $\x_{\E}$.  This implies that $w$ is of type II, which contradicts that $w$ is not on our list.

\bigskip

Case (2):  For the second case, assume that $n=3$.  In this case, $\x_{\O}=s_{1}s_{3}$ and $\x_{\E}=s_{2}s_{4}$.  As in Case (1), we wish to show that $\r_{m}=\x_{\O}$ or $\x_{\E}$.  For sake of a contradiction, assume that $s_{k} \in \r_{m}$ but $s_{k'} \notin \r_{m}$, where $|k-k'|=2$.  We consider the subcases: (a) $k=1$ and (b) $k=3$.  The cases $k=2$ and $k=4$ are similar.

\bigskip

(a):  Suppose that $k=1$, so that $s_{1} \in \r_{m}$ while $s_{3} \notin \r_{m}$.  Then we must have $s_{2} \in \r_{m-1}$.  Since $\supp(w) = S(\C_{n})$ and $s_{3}$ does not occur in $\r_{m-1}$ or $\r_{m}$, we must have $m \geq 3$.  Then the entry labeled by $s_{2}$ occurring in $\r_{m-1}$ cannot be covered by an entry labeled by $s_{1}$; otherwise, $w$ is right weak star reducible by $s_{1}$ with respect to $s_{2}$.  Thus, $s_{3} \in \r_{m-2}$.  But according to Lemma \ref{weak.star.lemma.end.n=3}, the entry labeled by $s_{2}$ occurring in $\r_{m-1}$ must be covered by an entry labeled by $s_{1}$, which is a contradiction.  

\bigskip

(b):  Next, suppose that $k=3$, so that $s_{3} \in \r_{m}$ while $s_{1} \notin \r_{m}$.  Then at least one of $s_{2}$ or $s_{4}$ occurs in $\r_{m-1}$.  If $s_{2} \in \r_{m-1}$, then $w$ would be right weak star reducible by $s_{3}$ with respect to $s_{2}$.  In fact, if the entry labeled by $s_{3} \in \r_{m}$ is covered by an entry labeled by $s_{2}$, we have the same contradiction.  So, it must be the case that $s_{4} \in \r_{m-1}$, while the entry labeled by $s_{3}$ is not covered by an entry labeled by $s_{2}$.  Since $\supp(w)=S(\C_{3})$, we must have $m\geq 3$; otherwise, $s_{1}, s_{2} \notin \supp(w)$.  This implies that $s_{3} \in \r_{m-2}$.  But then $w$ is right weak star reducible by $s_{3}$ with respect to $s_{4}$, which is a contradiction.

\bigskip

We have shown that $\r_{m}=\x_{\O}$ or $\x_{\E}$.  Now, by repeated applications of Lemma  \ref{sandwich.stack.lemma.n=3}, $w$ must be equal to an alternating product of $\x_{\O}$ and $\x_{\E}$.  This implies that $w$ is of type II, which contradicts that $w$ is not on our list.

\bigskip

Case (3):  Lastly, assume that $n\geq 4$.  As in the previous cases, we first show that $\r_{m}=\x_{\O}$ or $\x_{\E}$.  For sake of a contradiction, assume that $s_{k} \in \r_{m}$ but $s_{k'} \notin \r_{m}$, where $|k-k'|=2$.  Without loss of generality, assume that $1\leq k \leq n-1$ and $k'=k+2$, so that $s_{k+2} \notin \r_{m}$; the remaining cases are similar.  We consider three possibilities: (a) $k=1$, (b) $k=2$, and (c) $3\leq k \leq n-1$.

\bigskip

(a): If $k=1$, then this case is identical to (a) in Case (2), except that we contradict Lemma \ref{weak.star.lemma.end.n>=4}.

\bigskip

(b): Next, assume that $k=2$, so that $s_{2} \in \r_{m}$ while $s_{4} \notin \r_{m}$.  Then the entry labeled by $s_{2} \in \r_{m}$ cannot be covered by an entry labeled by $s_{3}$; otherwise, $w$ would be right weak star reducible by $s_{2}$ with respect to $s_{3}$.  This implies that we must have $s_{1} \in \r_{m-1}$.  Since $\supp(w) = S(\C_{n})$ and $s_{3}$ does not occur in $\r_{m-1}$ or $\r_{m}$, we must have $m \geq 3$.  Then $s_{2} \in \r_{m-2}$.  But then $w$ is right weak star reducible by $s_{2}$ with respect to $s_{1}$, which is a contradiction.

\bigskip

(c): Lastly, assume that $3\leq k \leq n-1$, so that $s_{k} \in \r_{m}$ while $s_{k+2} \notin \r_{m}$.  By Lemma \ref{weak.star.lemma.middle}, the entry labeled by $s_{k} \in \r_{m}$ must be covered by entries labeled by $s_{k-1}$ and $s_{k+1}$.  However, this implies that $w$ is right weak star reducible by $s_{k}$ with respect to $s_{k+1}$, which is a contradiction.

\bigskip

We have shown that $\r_{m}=\x_{\O}$ or $\x_{\E}$.  By making repeated applications of Lemma \ref{sandwich.stack.lemma.n=4} if $n=4$ or Lemma \ref{sandwich.stack.lemma.n>4} if $n>4$, $w$ must be equal to an alternating product of $\x_{\O}$ and $\x_{\E}$.  This implies that $w$ is of type II, which contradicts that $w$ is not on our list.
\bigskip

We have exhausted all possibilities, and hence our list is complete.  \hfill $\qed$

\end{section}

\end{chapter}

%%%%%%%%%%% Chapter 6 %%%%%%%%%%%

\begin{chapter}{Hecke algebras and Temperley--Lieb algebras}

In this chapter, we recall the necessary terminology and facts about Hecke algebras and their associated Temperley--Lieb algebras.

\begin{section}{Hecke algebras}

Let $X$ be an arbitrary Coxeter graph.  We define the \emph{Hecke algebra} of type $X$, denoted by $\H_{q}(X)$, to be the $\Z[q,q^{-1}]$-algebra with basis consisting of elements $T_{w}$, for all $w \in W(X)$, satisfying
	$$T_{s}T_{w}=\begin{cases}
T_{sw},  & \text{if } l(sw)>l(w), \\
qT_{sw}+(q-1)T_{w},  & \text{if } l(sw)<l(w)
\end{cases}$$
where $s \in S(X)$ and $w \in W(X)$.  This is enough to compute $T_{x}T_{w}$ for arbitrary $x, w \in W(X)$.  Also, it follows from the definition that each $T_{w}$ is invertible.

\bigskip

It is convenient to extend the scalars of $\H_{q}(X)$ to produce an $\A$-algebra, $\H(X)=\A \otimes_{\Z[q,q^{-1}]} \H_{q}(X)$, where $\A=\Z[v,v^{-1}]$ and $v^{2}=q$.  For more on Hecke algebras, we refer the reader to \cite[Chapter 7]{Humphreys.J:A}.

\bigskip

The Laurent polynomial $v+v^{-1} \in \A$ will occur frequently and will usually be denoted by $\delta$.

\begin{remark}
Since $W(\C_{n})$ is an infinite group, $\H(\C_{n})$ is an $\A$-algebra of infinite rank.  On the other hand, since $W(B_{n})$ and $W(B'_{n})$ are finite, both $\H(B_{n})$ and $\H(B'_{n})$ are $\A$-algebras of finite rank.
\end{remark}

\end{section}

\begin{section}{Temperley--Lieb algebras}

Let $J(X)$ be the two-sided ideal of $\H(X)$ generated by the set $\{J_{s,t}: 3\leq m(s,t)<\infty\}$, where
	$$J_{s,t}=\sum_{w \in \langle s, t \rangle}T_{w}$$
and $\langle s, t \rangle$ is the subgroup generated by $s$ and $t$.

\bigskip

Following Graham \cite[Definition 6.1]{Graham.J:A}, we define the \emph{(generalized) Temperley--Lieb algebra}, $\TL(X)$, to be the quotient $\A$-algebra $\H(X)/J(X)$.  We denote the corresponding canonical epimorphism by $\phi: \H(X) \to \TL(X)$.  

\begin{remark}
Except for in the case of type $A$, there are many Temperley--Lieb type quotients that appear in the literature.  That is,  some authors define a Temperley--Lieb algebra as a different quotient of $\H(X)$.  In particular, the blob algebra of \cite{Martin.P;Saleur.H:A} is a smaller Temperley--Lieb type quotient of $\H(B_{n})$ than $\TL(B_{n})$.  Also, the symplectic blob algebra of \cite{Green.R;Martin.P;Parker.A:A} and  \cite{Martin.P;Green.R;Parker.A:A} is a finite rank quotient of $\H(\C_{n})$, whereas, as we shall see, $\TL(\C_{n})$ is of infinite rank.  Typically, authors that study these smaller Temperley--Lieb type quotients are interested in representation theory, where our motivation is Kazhdan--Lusztig theory.
\end{remark}

Let $t_{w}=\phi(T_{w})$.  The following fact is due to Graham \cite[Theorem 6.2]{Graham.J:A}.

\begin{theorem}\label{t-basis}
The set $\{t_{w}: w \in W_{c}(X)\}$ is an $\A$-basis for $\TL(X)$. \hfill $\qed$
\end{theorem}

We will refer to the basis of Theorem \ref{t-basis} as the ``$t$-basis.''  For our purposes, it will be more useful to work a different basis, which we define in terms of the $t$-basis.

\begin{definition}
For each $s \in S(X)$, define $b_{s}=v^{-1}t_{s}+v^{-1}t_{e}$, where $e$ is the identity in $W(X)$.  If $s=s_{i}$, we will write $b_{i}$ in place of $b_{s_{i}}$.  If $w \in W_{c}(X)$ has reduced expression $\w=\w_{1}\cdots \w_{r}$, then we define 
	$$b_{\w}=b_{\w_{1}}\cdots b_{\w_{r}}.$$
\end{definition}

Note that if $\w$ and $\w'$ are two different reduced expressions for $w \in W_{c}(X)$, then $b_{\w}=b_{\w'}$ since $\w$ and $\w'$ are commutation equivalent and $b_{i}b_{j}=b_{j}b_{i}$ when $m(s_{i}, s_{j})=2$.  So, we will write $b_{w}$ if we do not have a particular reduced expression in mind.  It is well-known (and follows from \cite[Proposition 2.4]{Green.R:P}) that the set $\{b_{w}: w \in W_{c}(X)\}$ forms an $\A$-basis for $\TL(X)$.  This basis is referred to as the \emph{monomial basis} or ``$b$-basis.''  We let $b_{e}$ denote the identity of $\TL(X)$.

\begin{remark}
Recall from Chapter 1 that $W(\C_{n})$ contains an infinite number of fully commutative elements, while $W(B_{n})$ and $W(B'_{n})$ contain finitely many.  Hence $\TL(\C_{n})$ is an $\A$-algebra of infinite rank while $\TL(B_{n})$ and $\TL(B'_{n})$ are of finite rank.  (We can have $W(X)$ infinite while $W_{c}(X)$ is finite, in which case, $\H(X)$ is of infinite rank and $\TL(X)$ is of finite rank.)
\end{remark}

\end{section}

\begin{section}{A presentation for $\TL(X)$}

It will be convenient for us to have a presentation for $\TL(X)$ in terms of generators and relations.  We follow \cite{Green.R:P}.  

\begin{definition}\label{chebyshev.polys}
Let $\{U_{k}(x)\}_{k \in \N}$ be the sequence of polynomials defined by the conditions $U_{0}(x)=1$, $U_{1}(x)=x$, and the recurrence relation $U_{k+1}(x)=xU_{k}(x)-U_{k-1}(x)$ for $k>1$.
\end{definition}

The polynomials $U_{n}(2x)$ are sometimes called ``type II Chebyshev polynomials.''  If $f(x) \in \Z[x]$, we define $f^{s,t}_{b}$ to be the element of $\TL(X)$ given by the linear extension of the map sending $x^{k}$ to the product
	$$\underbrace{b_{s}b_{t}\cdots}_{k\text{ factors}}$$
of alternating factors starting with $b_{s}$.  For example, if $f(x)=x^{2}-1$, then $f^{s,t}_{b}=b_{s}b_{t}-1$.

\bigskip

The following theorem is \cite[Proposition 2.6]{Green.R:P}.  

\begin{theorem}
As a unital algebra, $\TL(X)$ is generated by $\{ b_{s}: s \in S(X)\}$ and relations
\begin{enumerate}[label=\rm{(\roman*)}]
\item $b_{s}^{2}=\delta b_{s}$ for all $s \in S(X)$ (where $\delta = v+v^{-1}$);
\item $b_{s}b_{t}=b_{t}b_{s}$ if $m(s,t)=2$;
\item $(xU_{m-1})^{s,t}_{b}(x)=0$ if $2<m=m(s,t)<\infty$.
\end{enumerate} \hfill $\qed$
\end{theorem}

\begin{remark}\label{affineC.relations}
If $m(s,t)=3$, then the last relation above becomes $b_{s}b_{t}b_{s}-b_{s}=0$.  If $m(s,t)=4$, we have $b_{s}b_{t}b_{s}b_{2}-2b_{s}b_{t}=0$.  This implies that $\TL(\C_{n})$ is generated as a unital algebra by $b_{1}, b_{2}, \dots, b_{n+1}$ with defining relations
\begin{enumerate}[label=\rm{(\arabic*)}]
\item $b_{i}^{2}=\delta b_{i}$ for all $i$ (where $\delta = v+v^{-1}$);
\item $b_{i}b_{j}=b_{i}b_{j}$ if $|i-j|>1$;
\item $b_{i}b_{j}b_{i}=b_{i}$ if $|i-j|=1$ and $1< i,j < n+1$;
\item $b_{i}b_{j}b_{i}b_{j}=2b_{i}b_{j}$ if $\{i,j\}=\{1,2\}$ or $\{n,n+1\}$.
\end{enumerate}
In addition, $\TL(B_{n})$ (respectively, $\TL(B'_{n})$) is generated as a unital algebra by $b_{1}, b_{2}, \dots, b_{n}$ (respectively, $b_{2}, b_{3}, \dots, b_{n+1}$) with the corresponding relations above.  It is known that we can consider $\TL(B_{n})$ and $\TL(B'_{n})$ as subalgebras of $\TL(\C_{n})$ in the obvious way.
\end{remark}

\end{section}

\begin{section}{Weak star reducibility and the monomial basis}

\begin{remark}\label{monomial.weak.star.reductions}
Suppose that $w \in W_{c}(\C_{n})$ is left weak star reducible by $s$ with respect to $t$.  Recall from Chapter 1 that this implies that $w=stv$ (reduced) when $m(s,t)=3$ or $w=stsv$ (reduced) when $m(s,t)=4$.  In this case, we have
	$$b_{t}b_{w}=\begin{cases}
  	 b_{tv},   & \text{if } m(s,t)=3, \\
  	 2b_{tsv},   & \text{if } m(s,t)=4.
	\end{cases}$$
It is important to note that $l(tv)=l(w)-1$ when $m(s,t)=3$ and $l(tsv)=l(w)-1$ when $m(s,t)=4$.  We have a similar characterization for right weak star reducibility.  
\end{remark}

\begin{example}
Let $w \in W_{c}(\C_{4})$ have reduced expression $\w=s_{1}s_{2}s_{1}s_{3}$.  Then $w$ is left weak star reducible by $s_{1}$ with respect to $s_{2}$, and so we have
	$$b_{2}b_{w}=2b_{s_{2}s_{1}s_{3}}.$$
Observe that $s_{1}\w=s_{2}s_{1}s_{3}$ is left weak star reducible by $s_{2}$ with respect to $s_{3}$ to the irreducible element $s_{1}s_{3}$.  Then
	$$b_{3}b_{2}b_{w}=2b_{3}b_{s_{2}s_{1}s_{3}}=2b_{s_{1}s_{3}}.$$
\end{example}

\begin{remark}
It is tempting to think that if $b_{w}$ is a monomial basis element such that $b_{t}b_{w}=2^{c}b_{y}$, where $c\in \{0,1\}$ and $l(y)<l(w)$, then $w$ is weak star reducible by some $s$ with respect to $t$, where $m(s,t)\geq 3$.  But this is not true.  For example, let $w=s_{1}s_{2}s_{3}s_{4} \in W_{c}(\C_{n})$ with $n \geq 3$, so that $m(s_{2},s_{3})=3$, and let $t=s_{3}$.  Then
\begin{align*}
b_{t}b_{w} &= b_{3}b_{1}b_{2}b_{3}b_{4} \\
&= b_{1}b_{3}b_{2}b_{3}b_{4}\\
&= b_{1}b_{3}b_{4} \\
&= b_{s_{1}s_{3}s_{4}}.
\end{align*}
We see that $l(s_{1}s_{3}s_{4}) < l(w)$, but $w$ is not left weak star reducible by $s_{3}$ (or any generator).
\end{remark}

The next lemma is useful for reversing the multiplication of monomials corresponding to weak star reductions.

\begin{lemma}\label{weak.star.reverse}
Let $w \in W_{c}(\C_{n})$ and suppose that $w$ is left weak star reducible by $s$ with respect to $t$.  Then
$$b_{s}b_{t}b_{w}=\begin{cases}
  	 b_{w},   & \text{if } m(s,t)=3, \\
  	 2b_{w},   & \text{if } m(s,t)=4.
	\end{cases}$$
We have an analogous statement if $w$ is right weak star reducible by $s$ with respect to $t$.
\end{lemma}

\begin{proof}
Suppose that $w$ is left weak star reducible by $s$ with respect to $t$.  Then we can write $w=stv$ (reduced) when $m(s,t)=3$ or $w=stsv$ (reduced) when $m(s,t)=4$, which implies that
	$$b_{t}b_{w}=\begin{cases}
  	 b_{tv},   & \text{if } m(s,t)=3, \\
  	 2b_{tsv},   & \text{if } m(s,t)=4.
	\end{cases}$$
Therefore, we have
\begin{align*}
b_{s}b_{t}b_{w}=&\begin{cases}
  	 b_{s}b_{tv},   & \text{if } m(s,t)=3, \\
  	 2b_{s}b_{tsv},   & \text{if } m(s,t)=4,
	\end{cases}\\
	=& \begin{cases}
  	 b_{stv},   & \text{if } m(s,t)=3, \\
  	 2b_{stsv},   & \text{if } m(s,t)=4,
	\end{cases}\\
	=& \begin{cases}
  	 b_{w},   & \text{if } m(s,t)=3, \\
  	 2b_{w},   & \text{if } m(s,t)=4,
	\end{cases}
\end{align*}
as desired.
\end{proof}

\end{section}

\begin{section}{A characterization for an arbitrary product of monomials}

It will be useful for us to know what form an arbitrary product of monomial generators takes in $\TL(\C_{n})$.  The next lemma is similar to \cite[Lemma 2.1.3]{Green.R;Losonczy.J:E}, which is a statement involving $W(B_{n})$.

\begin{lemma}\label{powers_of_2_and_delta_monomials}
Let $w \in W_{c}(\C_{n})$ and let $s \in S(\C_{n})$.  Then
	$$b_{s}b_{w}=2^{k}\delta^{m}b_{w'}$$
for some $k, m \in \Z^{+}\cup \{0\}$ and $w' \in W_{c}(\C_{n})$.
\end{lemma}

\begin{proof}
We induct on the length of $w$.  For the base case, assume that $l(w)=0$.  That is, $w=e$.  Then for any $s \in S(\C_{n})$, we have $b_{s}b_{e}=b_{s}$, which gives us our desired result.  Now, assume that $l(w)=p>1$.  There are three possibilities to consider.  

\bigskip

Case (1):  First, if $sw$ is reduced and fully commutative, then
	$$b_{s}b_{w}=b_{sw},$$
which agrees with the statement of the lemma.

\bigskip

Case (2):  Second, if $sw$ is not reduced, then $s \in \L(w)$, as so we must be able to write $w=sv$ (reduced).  In this case, we see that
	$$b_{s}b_{w}=b_{s}b_{s}b_{v}=\delta b_{s}b_{v}=\delta b_{w}.$$
Again, this agrees with the statement of the lemma.

\bigskip

Case (3):  For the final case, assume that $sw$ is reduced, but not fully commutative.  Then $sw$ must have a reduced expression containing the subword $sts$ if $m(s,t)=3$ or $stst$ if $m(s,t)=4$.  So, we must be able to write
	$$w=\begin{cases}
		utsv, & \text{if } m(s,t)=3,\\
		utstv, & \text{if } m(s,t)=4,
	\end{cases}$$
where each product is reduced, $u, v \in W_{c}(\C_{n})$, and $s$ commutes with every element of $\supp(u)$, so that
	$$sw=\begin{cases}
		ustsv, & \text{if } m(s,t)=3,\\
		uststv, & \text{if } m(s,t)=4.
	\end{cases}$$
This implies that
	$$b_{s}b_{w}=\begin{cases}
		b_{u}b_{s}b_{t}b_{s}b_{v}=b_{u}b_{s}b_{v}, & \text{if } m(s,t)=3,\\
		b_{u}b_{s}b_{t}b_{s}b_{t}b_{v}=2b_{u}b_{s}b_{t}b_{v}, & \text{if } m(s,t)=4.
	\end{cases}$$
Note that $l(u)+1+l(v) < p$ in the $m(s,t)=3$ case and that $l(u)+2+l(v) < p$ when $m(s,t)=4$.  So, we can apply the inductive hypothesis $l(u)+1$ (respectively, $l(u)+2$) times if $m(s,t)=3$ (respectively, $m(s,t)=4$) starting with $b_{s}b_{v}$ (respectively, $b_{t}b_{v}$).  Therefore, we obtain
	$$b_{s}b_{w}=2^{k}\delta^{m}b_{w'}$$
for some $k, m \in \Z^{+}\cup \{0\}$ and $w' \in W_{c}(\C_{n})$, as desired.
\end{proof}

\begin{remark}
If $b_{i_{1}}, b_{i_{2}}, \dots, b_{i_{p}}$ is any collection of $p$ monomial generators, then it follows immediately from Lemma \ref{powers_of_2_and_delta_monomials}  that
	$$b_{i_{1}}b_{i_{2}}\cdots b_{i_{p}}=2^{k}\delta^{m} b_{w}$$
for some $k, m \in \Z^{+}\cup \{0\}$ and $w \in W_{c}(\C_{n})$.
\end{remark}

\end{section}

\end{chapter}

%%%%%%%%%%%%% Chapter 7 %%%%%%%%%%%%

\begin{chapter}{Free products of associative algebras}

In the next chapter, we begin describing the construction of our desired diagram algebra.  In order to define the decoration set for our diagram algebra, we will need to describe a method for combining two associative algebras together to create a new associative algebra.  To accomplish this, we will make use of Bergman's diamond lemma \cite{Bergman.G:A}.  We follow the development in \cite{Green.R:B}.

\begin{section}{Bergman's diamond lemma}

Let $R$ be a commutative ring and let $X$ be a nonempty set.  Define $X^*$ to be the free monoid generated by $X$.  Let $\leq_X$ be a semigroup partial order on $X^*$: that is, if $\lambda, \mu, \nu$ are (possibly empty) words in $X^*$ and $\mu \leq \nu$, then we have $\lambda \mu \leq \lambda \nu$ and $\mu \lambda \leq \nu \lambda$.  We say $\leq_X$ satisfies the \emph{descending chain condition} if any sequence $\mu_1 >_X \mu_{2} >_X \cdots$ terminates.  A \emph{reduction system} $\S$ for the module $RX^*$ is a set of rules $s: X^{*} \to RX^{*}$ of the form  
	$$s(\mu) = \sum_{i=1}^{k}\alpha_{i}\lambda_{i},$$
where each $\alpha_{i} \in R$ and each $\lambda_{i} \in X^{*}$.  If a rule $s$ acts nontrivially on a monomial $\mu$, then we will write $\mu=\mu_{s}$ and denote the image of $\mu$ after applying $s$ by $a_{s}$.  That is, $s(\mu_{s})=a_{s}$.  (Note that $a_{s}$ is a linear combination of monomials with coefficients from $R$.)  The $R$-module maps $RX^*\to RX^*$ used to apply rules are known as \emph{reductions}; these may consist of several rules performed sequentially.  The two-sided ideal $\I(\S)$ of $RX^*$ is that generated by all elements $\mu_s - a_s$ for all rules $s \in S$.  We say $s$ is \emph{compatible} with $\leq_X$ if $a_s$ can be written as a linear combination of monomials strictly less than $\mu_s$ in $\leq_X$, and we say $\S$ is compatible with $\leq_X$ if each of its rules is.

\bigskip

An \emph{overlap ambiguity} occurs when there are two rules $s_1$ and $s_2$ such that there exist monomials $\lambda$ and $\nu$ with $\mu_{s_1}\lambda = \nu \mu_{s_2}$; it is said to be \emph{resolvable} if there are reductions $t_1$ and $t_2$ such that $t_1\( a_{s_1}\lambda\) =t_2\( \nu a_{s_2}\)$.  An \emph{inclusion ambiguity} occurs when there are two rules $s_1$ and $s_2$ such that there exist monomials $\lambda$ and $\nu$ with $\lambda \mu_{s_1} \nu=\mu_{s_2}$; it is said to be \emph{resolvable} if there are reductions $t_1$ and $t_2$ such that $t_1\( \lambda a_{s_1} \nu\) =t_2\(a_{s_2}\)$.

\bigskip

A reduction $t$ is said to act trivially on $a \in RX^*$ if $t(a)=a$, and if all reductions act trivally on $a$, we say $a$ is \emph{irreducible}.  The set of irreducible elements arising from $\S$ is denoted $\Irr(\S)$ and has an obvious $R$-module structure.  A \emph{normal form} of $a \in RX^*$ is an element $b \in \Irr(\S)$ to which $a$ can be reduced; it is not immediate that normal forms always exist or that they are unique.

\bigskip

The following theorem is part of Bergman's diamond lemma, which is proved in \cite{Bergman.G:A}.

\begin{theorem}[Bergman's diamond lemma]\label{bergman}
Let $R$ be a commutative ring with 1.  Let $X$ be a nonempty set, let $\leq_X$ be a semigroup partial order on $X^*$, and let $\S$ be reduction system for $RX^*$.  If $\S$ is compatible with $\leq_X$, and $\leq_X$ satisfies the descending chain condition, then the following are equivalent:
\begin{enumerate}[label=\rm{(\roman*)}]
\item All ambiguities in $\S$ are resolvable.
\item Every element $a \in RX^*$ has a unique normal form equaling $t(a)$ for some reduction $t$.
\item $RX^*$ and $\Irr(\S) \oplus \I(\S)$ are isomorphic as $R$-modules. 
\end{enumerate} \hfill $\qed$
\end{theorem}

We will make use of Theorem \ref{bergman} in the next section to prove a result about the free product of two associative algebras.

\end{section}

\begin{section}{Free product of associative algebras}

Let $R$ be a commutative ring with $1$.  Let  $A_1$ and $A_2$ be two associative algebras over $R$ with bases $\B_{1}$ and $\B_{2}$, respectively, each containing $1$.  Let $X=(\B_{1} \cup \B_{2})\backslash  \{1\}$ and consider $X^{*}$.  We wish to describe a set of rules on $X^{*}$ for reducing monomials.  Let $c_{1}, c_{2} \in X$.  That is, each $c_{i}$ is a non-identity element of either $\B_{1}$ or $\B_{2}$.  Recall that if both $c_{1}$ and $c_{2}$ are elements of the same $\B_{j}$, then there exist $\alpha_1, \dots, \alpha_k \in R$ and $b_1, \dots, b_k \in \B_j$ such that 
	$$c_1 c_2=\sum_{i=1}^k \alpha_i b_i.$$
For each such relation, define a rule $s$ on $X^*$ that sends the product of two basis elements from $\B_{j}$ to the corresponding linear combination.  Let $\S$ be the reduction system for $RX^*$ consisting of all such rules.  If $c_{1} \in \B_{1}$ and $c_{2} \in \B_{2}$, then there is no rule for reducing $c_1 c_2$ or $c_2 c_1$.  Define 
	$$A_1*A_2=RX^*/\I(\S).$$
The algebra $A_{1}*A_{2}$ is generated as a unital algebra by the disjoint union of $\B_{1}$ and $\B_{2}$, subject to the corresponding relations of $(A_1, \B_1)$ and $(A_2, \B_2)$.  The point is that there is no interaction between the elements of $\B_1$ and $\B_2$ in $A_1*A_2$.  It turns out that $A_{1}*A_{2}$ is the \emph{free product} of $A_{1}$ and $A_{2}$. 

\bigskip

Next, define $\B_1*\B_2$ to be the union of $\{1\}$ with the set of all finite alternating products of non-identity elements from $\B_1$ and $\B_2$.  We will write an arbitrary non-identity element $\beta \in \B_1*\B_2$ as $\beta=c_1 c_2 \cdots c_n \in \B_1*\B_2$, where each $c_i$ is a non-identity element of either $\B_1$ or $\B_2$ and if $c_{i} \in \B_{1}$ (respectively, $\B_{2}$), then $c_{i+1} \in \B_{2}$ (respectively, $\B_{1}$).  Note that $\Irr(\S)$ is the set of all finite linear combinations of elements from $\B_1*\B_2$.

\bigskip

The following proposition is probably well-known, but we have been unable to find a reference.  So, we provide a proof here.

\begin{proposition}\label{free_product_algebras}
Let $R$ be a commutative ring with $1$ and let $A_1$ and $A_2$ be two associative $R$-algebras with bases $\B_{1}$ and $\B_{2}$, respectively, each containing $1$.  Then $A_1*A_2$ is an associative algebra over $R$ with basis $\B_1*\B_2$ containing $1$.
\end{proposition}

\begin{proof}
Certainly, $\B_1*\B_2$ is a spanning set for $A_1*A_2$.  We will use Bergman's diamond lemma to show that $\B_1*\B_2$ is in fact a basis.  Define $\leq_X$ on $X^{*}$ to be the partial order such that $c_1c_2\cdots c_n \leq_X c'_1c'_2 \cdots c'_m$ if and only if either $n<m$ or $n=m$ and $c_i=c'_i$ for all $i$.  This is a semigroup partial order on $X^{*}$ and satisfies the descending chain condition.  Since all reductions acting nontrivially send pairs of elements from $X$ to linear combinations of single elements from $X$, $\leq_X$ is compatible with $\S$.  We now argue that all the ambiguities are resolvable.  Since all reductions apply (locally) to pairs of elements from $X$, there are no inclusion ambiguities.  All overlap ambiguities are of the form $\gamma c_1 c_2 c_3 \gamma'$, where $c_1, c_2, c_3 \in \B_j$ for some $j$ and $\gamma, \gamma' \in X^*$.  Since each $A_j$ is associative, every overlap ambiguity must be resolvable.  Hence all ambiguities are resolvable.  By Theorem \ref{bergman}, every element $\beta \in RX^*$ has a unique normal form which equals $t(\beta)$ for some reduction $t$ and 
	$$RX^*=\Irr(\S)\oplus \I(\S)$$
as $R$-modules.  This implies that 
	$$A_1*A_2\cong \Irr(\S)$$
as $R$-modules, where $\Irr(\S)$ is the set of all finite linear combinations of elements from $\B_1*\B_2$.  Thus, $\B_1*\B_2$ is a basis for $A_1*A_2$.  We have also shown that $A_1*A_2$ is an associative algebra.
\end{proof}

\end{section}

\begin{section}{Free product of two Verlinde algebras}

We are interested in a very specific free product of algebras.  Recall the definition of the type II Chebyshev polynomials $\{U_{k}(x)\}$ given in Definition \ref{chebyshev.polys}.  We use these polynomials to define the Verlinde algebra, which first appeared in \cite{Verlinde.E:A}. 

\begin{definition}
Let $r\geq 1$.  The \emph{Verlinde algebra}, $V_{r}$, is defined to be the quotient of $\Z[x]$ by the ideal generated by $U_{r}(x)$.
\end{definition}

\begin{lemma}\label{cheby.basis.verlinde}
The algebra $V_{r}$ has rank $r$ and is equipped with a $\Z$-basis consisting of the images $u_{k}$ of the elements $U_{k}(x)$ for $0\leq k <r$.
\end{lemma}

\begin{proof}
See \cite[Proposition 1.2.3]{Green.R:M}.
\end{proof}

For example, $V_{3}$ is equal to the quotient of $\Z[x]$ by the ideal generated by the polynomial $U_{3}(x)=x^{3}-2x$.  A basis for $V_{3}$ is given by the images of $U_{0}(x)=1$, $U_{1}(x)=x$, and $U_{2}(x)=x^{2}-1$.  We are interested in the algebra $V_{3}$, but for our current purposes, we need a different basis.  Define $V_{0}(x)=U_{0}(x)$, $V_{1}(x)=U_{1}(x)$, and $V_{2}(x)=x^{2}$ and let $v_{0}$, $v_{1}$, and $v_{2}$ be their respective images in $V_{3}$.  Then it is readily seen that $\{v_{0}, v_{1}, v_{2}\}$ forms a basis for $V_{3}$.  Note that $v_{0}$ is the identity in $V_{3}$ and the other basis elements satisfy the relations
\begin{enumerate}[label=\rm{(\arabic*)}]
\item $v_{1}^{2}=v_{2}$;
\item $v_{1}v_{1}v_{1}=2v_{1}$.
\end{enumerate}

Now, consider two copies of $V_{3}$.  As above, $\{v_{0}, v_{1}, v_{2}\}$ forms a basis for $V_{3}$.  To avoid confusion, let $V'_{3}$ denote the second copy of $V_{3}$ and denote its basis by $\{v'_{0}, v'_{1}, v'_{2}\}$.  It will be very useful for us to establish the following correspondence, which will be used later to add decorations to our diagrams:  
\begin{enumerate}[label=\rm{(\arabic*)}]
\item $v_1 := \bcirc$;
\item $v_2 := \btri$;
\item $v'_{1} := \wcirc$;
\item $v'_{2} := \wtri$;
\end{enumerate}
so that $\{1, \bcirc, \btri \}$ and $\{1, \wcirc, \wtri \}$ are each a basis for $V_{3}$ and $V'_{3}$, respectively.  Then, the relations in $V_{3}$ and $V'_{3}$, respectively, become
\begin{enumerate}[label=\rm{(\arabic*)}]
\item $\bcirc \bcirc = \btri$;
\item $\bcirc \bcirc \bcirc = 2\ \bcirc$;
\item $\wcirc \wcirc =\ \wtri$;
\item $\wcirc \wcirc \wcirc = 2 \wcirc.$
\end{enumerate}

Note that these relations imply that
\begin{enumerate}[label=\rm{(\arabic*)}]
\item $\btri \btri = 2\btri$;
\item $\wtri \wtri =  2\wtri$;
\item $\bcirc \btri=\btri \bcirc=2 \bcirc$;
\item $\wcirc \wtri=\wtri \wcirc=2 \wcirc$.
\end{enumerate}

We will refer to $\bcirc$ and $\btri$ as \emph{closed decorations} and $\wcirc$ and $\wtri$ as \emph{open decorations}.  By Proposition \ref{free_product_algebras}, $\V:=V_{3}*V'_{3}$ is an associative algebra with a basis consisting of the identity and all finite alternating products of open and closed decorations.  For example, 
	$$\bcirc \bcirc \wcirc \bcirc \wcirc \wcirc \bcirc=\btri \wcirc \bcirc \wtri \bcirc,$$
where the expression on the right is a basis element of $\V$, while the expression on the left is not.

\begin{remark}\label{beads}
We can also think of $\V$ as being constructed in the following way.  Let $R=\Z$ and $X=\{\bcirc, \btri, \wcirc, \wtri\}$.  Consider the free $R$-algebra $RX^{*}$.  Then $\V$ is equal to the quotient of $RX^{*}$ by the following relations:
\begin{enumerate}[label=\rm{(\arabic*)}]
\item $\bcirc \bcirc = \btri$;
\item $\bcirc \bcirc \bcirc = 2\ \bcirc$;
\item $\wcirc \wcirc =\ \wtri$;
\item $\wcirc \wcirc \wcirc = 2\ \wcirc$.
\end{enumerate}
\end{remark}
\end{section}
\end{chapter}

%%%%%%%%%%%% Chapter 8 %%%%%%%%%%%

\begin{chapter}{Diagram algebras}

The goal of this chapter is to familiarize the reader with the necessary background on diagram algebras and to define a generating set for the diagram algebra that we are interested.  We will also define a set of diagrams that will turn out to be a basis for this diagram algebra (see Chapter 9).  It is important to note that there is currently no rigorous definition of the term ``diagram algebra.''  However, our diagram algebras possess many of the same features as those already appearing in the literature. 

\begin{section}{Summary of notation}

For the reader's reference, we summarize here the notations used throughout the remainder of this thesis (and indicate where each is defined): 

\begin{center}
\begin{tabular}[c]{|p{.65in}|p{4.1in}|p{1.1in}|}
\hline
$T_{k}(\emptyset)$ & set of (undecorated) pseudo $k$-diagrams & Definition \ref{def.T_k(emptyset)} \\
\hline
$\DTL(A_{n})$ & $\Z[\delta]$-algebra having $T_{k}(\emptyset)$ as a basis & Definition \ref{def.DTL(A_n)} \\
\hline
$T_{k}(A)$ & set of $A$-decorated pseudo $k$-diagrams & Definition \ref{def.T_k(A)} \\
\hline
$\P_{k}(A)$ & $R$-algebra having $T_{k}(A)$ as a basis & Definition \ref{A-decorated_diagram_algebra} \\
\hline
$T_{n+2}^{LR}(\V)$ & set of LR-decorated diagrams from $T_{n+2}(\V)$ & Definition \ref{def.T_n+2^LR(V)} \\
\hline
$\P_{n+2}^{LR}(\V)$ & $\Z[\delta]$-module spanned by $T_{n+2}^{LR}(\V)$ & Definition \ref{def.P_n+2^LR(V)} \\
\hline
$\widehat{\P}_{n+2}^{LR}(\V)$ & quotient of $\P_{n+2}^{LR}(\V)$ by the relations induced by $\V$ & Definition \ref{big.diagram.alg.defn} \\
\hline
$\D_{n}$ & $\Z[\delta]$-algebra generated (as a unital algebra) by $d_{1}, d_{2}, \dots, d_{n+1}$ & Definition \ref{def.D_n}\\
\hline
$\Diag^{b}_{n}(\V)$ & set of all $\C$-admissible $(n+2)$-diagrams & Definition \ref{admissible.def} \\
\hline
$\mathcal{M}[\Diag_{n}^{b}(\V)]$ & $\Z[\delta]$-module spanned by the admissible diagrams & Definition \ref{def.mathcalM[Diag_n^b(V)} \\
\hline
$\DTL(B_{n})$ & $\Z[\delta]$-algebra generated (as a unital algebra) by $d_{1}, d_{2}, \dots, d_{n}$  & Definition \ref{def.DTL(B_n)} \\
\hline
$\DTL(B'_{n})$ & $\Z[\delta]$-algebra generated (as a unital algebra) by $d_{2}, d_{3}, \dots, d_{n+1}$ & Definition \ref{def.DTL(B_n)} \\
\hline
\end{tabular}
\end{center}

\end{section}

\begin{section}{Ordinary Temperley--Lieb pseudo diagrams}

It is worth noting that our development in this chapter is more general than many of the standard developments.  The usual developments are too restrictive to accomplish the task of finding a diagrammatic representation of the infinite dimensional algebra $\TL(\C_{n})$.  Yet, our development is modeled after \cite{Green.R:M}, \cite{Green.R;Martin.P;Parker.A:A}, \cite{Jones.V:A}, and \cite{Martin.P;Green.R;Parker.A:A}.

\begin{definition}\label{k_box}
Let $k$ be a nonnegative integer.  The \emph{standard $k$-box} is a rectangle with $2k$ marks points, called \emph{nodes} (or \emph{vertices}) labeled as follows.

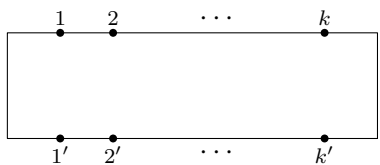
\begin{center}
%-- New mfpic environment, number 73 of 189. (size of end split: 2, should be 2)  ------------------->
\includegraphics{ThesisFigs1.073}
\end{center}

We will refer to the top of the rectangle as the \emph{north face} and the bottom as the \emph{south face}.  Often, it will be useful for us to think of the standard $k$-box as being embedded in the plane.  In this case, we put the lower left corner of the rectangle at the origin such that each node $i$ (respectively, $i'$) is located at the point $(i,1)$ (respectively, $(i,0)$).
\end{definition}

Next, we summarize the construction of the ordinary Temperley--Lieb pseudo diagrams.  

\begin{definition}\label{def.T_k(emptyset)}
A \emph{concrete pseudo $k$-diagram} consists of a finite number of disjoint curves (planar), called \emph{edges}, embedded in the standard $k$-box.  Edges may be closed (isotopic to circles), but not if their endpoints coincide with the nodes of the box.  The nodes of the box are the endpoints of curves, which meet the box transversely.  Otherwise, the curves are disjoint from the box.  We define an equivalence relation on the set of concrete pseudo $k$-diagrams.  Two concrete pseudo $k$-diagrams are \emph{(isotopically) equivalent} if one concrete diagram can be obtained from the other by isotopically deforming the edges such that any intermediate diagram is also a concrete pseudo $k$-diagram.  A \emph{pseudo $k$-diagram} (or an \emph{ordinary Temperley-Lieb pseudo-diagram}) is defined to be an equivalence class of equivalent concrete pseudo $k$-diagrams.  We denote the set of pseudo $k$-diagrams by $T_{k}(\emptyset)$.
\end{definition}

\begin{example}
Here is an example of a concrete pseudo 5-diagram.
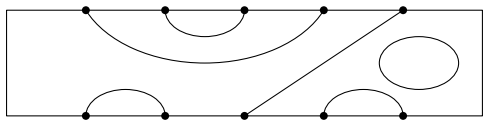
\begin{center}
%-- New mfpic environment, number 74 of 189. (size of end split: 2, should be 2)  ------------------->
\includegraphics{ThesisFigs1.074}
\end{center}
Here an example of a drawing that is \textit{not} a concrete pseudo 5-diagram.
\begin{center}
%-- New mfpic environment, number 75 of 189. (size of end split: 2, should be 2)  ------------------->
\includegraphics{ThesisFigs1.075}
\end{center}
\end{example}

\begin{remark}\label{diagram_repn}
When representing a pseudo $k$-diagram with a drawing, we pick an arbitrary concrete representative among a continuum of equivalent choices.  When no confusion can arise, we will not make a distinction between a concrete pseudo $k$-diagram and the equivalence class that it represents.  We say that two concrete pseudo $k$-diagrams are \emph{vertically equivalent} if they are equivalent in the above sense by an isotopy that preserves setwise each vertical cross-section of the $k$-box.  
\end{remark}

We will refer to any closed curves occurring in the pseudo $k$-diagram as a \emph{loop edge}, or simply a \emph{loop}.  The diagram in the example above has a single loop.  Note that we used the word ``pseudo'' in our definition to emphasize that we allow loops to appear in our diagrams.  Most examples of diagram algebras in the literature ``scale away'' loops that appear (i.e., remove and multiply by a scalar).  There are loops in the diagram algebra that we are interested in preserving, so as to obtain infinitely many diagrams.  When no confusion will arise, we will refer to a pseudo $k$-diagram as simply a diagram.  The presence of $\emptyset$ in the definition above is to emphasize that the edges of the diagrams are undecorated.  In the next section, we will allow for the presence of decorations.

\bigskip

Let $d$ be a diagram.  If $d$ has an edge $e$ that joins node $i$ in the north face to node $j'$ in the south face, then $e$ is called a \emph{propagating edge from $i$ to $j'$}.  (Propagating edges are often referred to as ``through strings'' in the literature.)  If a propagating edge joins $i$ to $i'$, then we will call it a \emph{vertical propagating edge}.  If an edge is not propagating, loop edge or otherwise, it will be called \emph{non-propagating}.  

\bigskip

If a diagram $d$ has at least one propagating edge, then we say that $d$ is \emph{dammed}.  If, on the other hand, $d$ has no propagating edges (which can only happen if $k$ is even), then we say that $d$ is \emph{undammed}.  Note that the number of non-propagating edges in the north face of a diagram must be equal to the number of non-propagating edges in the south face.  We define the function $\a: T_{k}(\emptyset) \to \Z^{+}\cup \{0\}$ via
	$$\a(d)=\text{ number of non-propagating edges in the north face of } d.$$
	
\begin{remark}
There is only one diagram with $\a$-value $0$ having no loops; namely
	$$d_{e}:=\begin{tabular}[c]{@{} c@{}}
%-- New mfpic environment, number 76 of 189. (size of end split: 2, should be 2)  ------------------->
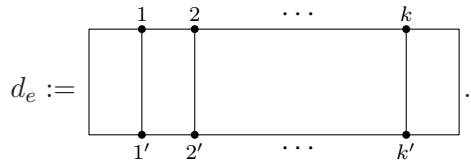
\includegraphics{ThesisFigs1.076}
\end{tabular}.$$
The maximum value that $\a(d)$ can take is $\lfloor k/2 \rfloor$.  In particular, if $k$ is even, then the maximum value that $\a(d)$ can take is $k/2$, i.e., $d$ is undammed.  On the other hand, if $\a(d)=\lfloor k/2 \rfloor$ while $k$ is odd, then $d$ has a unique propagating edge.
\end{remark}

We wish to define an associative algebra that has the pseudo $k$-diagrams as a basis.

\begin{definition}\label{defn_P_{k}}
Let $R$ be a commutative ring with $1$.  The associative algebra $\P_{k}(\emptyset)$ over $R$ is the free $R$-module having $T_{k}(\emptyset)$ as a basis, with multiplication defined as follows.  If $d, d' \in T_{k}(\emptyset)$, the product $d'd$ is the element of $T_{k}(\emptyset)$ obtained by placing $d'$ on top of $d$, so that node $i'$ of $d'$ coincides with node $i$ of $d$, rescaling vertically by a factor of $1/2$ and then applying the appropriate translation to recover a standard $k$-box.
\end{definition}

\begin{remark}
For a proof that this procedure does in fact define an associative algebra see \cite[\textsection 2]{Green.R:M} and \cite{Jones.V:A}.
\end{remark}

We will refer to the multiplication of diagrams as \emph{diagram concatenation}.  The (ordinary) Temperley--Lieb diagram algebra (see \cite{Green.R:N, Green.R:M, Jones.V:A, Penrose.R:A}) can be easily defined in terms of this formalism.

\begin{definition}\label{def.DTL(A_n)}
Let $\DTL(A_{n})$ be the associative $\Z[\delta]$-algebra equal to the quotient of $\P_{n+1}(\emptyset)$ by the following relation:
\begin{center}
\begin{tabular}[c]{@{} c@{}}
%-- New mfpic environment, number 77 of 189. (size of end split: 2, should be 2)  ------------------->
\includegraphics{ThesisFigs1.077}
\end{tabular}
= $\delta$
\end{center}
It is well-known that $\DTL(A_{n})$ is the free $\Z[\delta]$-module with basis given by the elements of $T_{n+1}(\emptyset)$ having no loops. The multiplication is inherited from the multiplication on $\P_{n+1}(\emptyset)$ except we multiply by a factor of $\delta$ for each resulting loop and then discard the loop.  We will refer to $\DTL(A_{n})$ as the \emph{(ordinary) Temperley--Lieb diagram algebra}.
\end{definition}

\begin{example}
Here is an example of multiplication of three basis diagrams of $\DTL(A_{4})$.
{\setlength\arraycolsep{1pt}
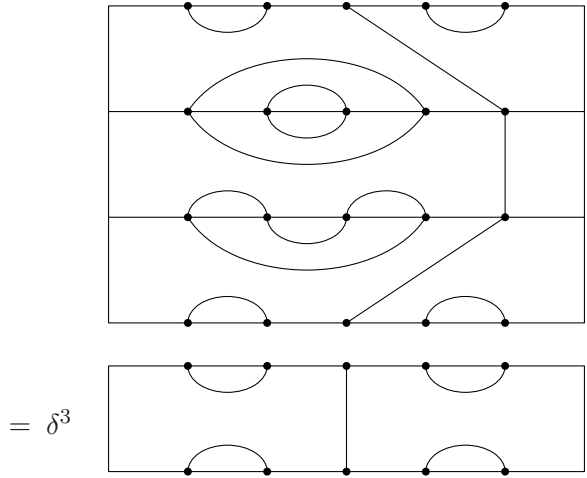
\begin{eqnarray*}
& & \ \begin{tabular}[c]{l}
%-- New mfpic environment, number 78 of 189. (size of end split: 2, should be 2)  ------------------->
\includegraphics{ThesisFigs1.078}
\end{tabular} \\
= &\  \delta^{3} &\ \begin{tabular}[c]{l}
%-- New mfpic environment, number 79 of 189. (size of end split: 2, should be 2)  ------------------->
\includegraphics{ThesisFigs1.079}
\end{tabular}
\end{eqnarray*}}

\end{example}

For a proof of the following theorem, see \cite{Kauffman.L:B} or \cite{Penrose.R:A}.

\begin{theorem}
As $\Z[\delta]$-algebras, $\TL(A_{n}) \cong \DTL(A_{n})$.  Moreover, each loop-free diagram from $T_{n+1}(\emptyset)$ corresponds to a unique monomial basis element of $\TL(A_{n})$.  \hfill $\qed$
\end{theorem}

\end{section}

\begin{section}{Decorated pseudo diagrams}

Now, we will adorn the edges of a diagram with elements from an associative algebra having a basis containing $1$.  First, we need to develop some terminology and lay out a few restrictions on how we decorate our diagrams.

\bigskip

Let $R$ be a commutative ring with $1$ and let $X=\{x_{i}: i\in I\}$ be a set such that $X \cap R$ does not contain $1 \in R$.  Let $\sim$ be a set of relations on $RX^{*}$ chosen so that $A:=RX^{*}/\sim$ has a basis $\B$ given by a subset of $X^{*}$ containing $1$ and $X$.  Then $A$ is an associative $R$-algebra.  We will refer to each $x_{i}\in X$ as a \emph{decoration}.  Every basis element in $\B$ can be written as a finite product of decorations.  Let $\mathbf{b}=x_{i_{1}}x_{i_{2}}\cdots x_{i_{r}}$ be a finite sequence of decorations in $X^{*}$; note that by definition each decoration is not equal to the identity and that $\mathbf{b}$ may or may not be a basis element of $A$.  We say that $x_{i_{j}}$ and $x_{i_{k}}$ are \emph{adjacent} if $|j-k|=1$ and we will refer to $\mathbf{b}$ as a \emph{block} of decorations of \emph{width} $r$.  Note that a block of width $1$ is a just a single decoration.  

\bigskip

Temporarily, we will ignore the relations of $A$.  That is, at this point, we should consider blocks of decorations as elements of $X^{*}$.

\begin{example}
We remind the reader that according to Remark \ref{beads}, $\V$ with basis consisting of finite alternating products of open and closed decorations is such an algebra, where the decoration set is $\{\bcirc, \wcirc, \btri, \wtri\}$.  Observe that
	$$\bcirc \wcirc \wcirc \btri \wtri \bcirc \wcirc$$
is block of width 7, but is not a basis element of $\V$ since there are two adjacent $\wcirc$ decorations in the second and third positions.
\end{example}

Next, let $d$ be a fixed concrete pseudo $k$-diagram and let $e$ be a non-loop edge of $d$.  We may adorn $e$ with a finite sequence of blocks of decorations $\mathbf{b}_{1}, \dots, \mathbf{b}_{m}$ such that adjacency of blocks and decorations of each block is preserved as we travel along $e$.  The convention we adopt is that the decorations of the block are placed so that we can read off the sequence of decorations from left to right as we traverse $e$ from $i$ to $j'$ if $e$ is propagating or from $i$ to $j$ (respectively, $i'$ to $j'$) with $i < j$ (respectively, $i' < j'$) if $e$ is non-propagating.  Furthermore, we should encounter block $\mathbf{b_{i}}$ before block $\mathbf{b_{i+1}}$.  We may also adorn a loop edge with a sequence of blocks. In this case, reading the corresponding sequence of decorations depends on an arbitrary choice of starting point and direction round the loop. We say two sequences of blocks are \emph{loop equivalent} if one can be changed to the other or its opposite by any cyclic permutation. Note that loop equivalence is an equivalence relation on the set of sequences of blocks.  So, the sequence of blocks on a loop is only defined up to loop equivalence.  That is, if we adorn a loop edge with a sequence of blocks of decorations, we only require that adjacency be preserved.

\bigskip

Again, let $d$ be a fixed concrete pseudo $k$-diagram and let $e$ be an edge of $d$.  Each decoration $x_{i}$ on $e$  has coordinates in the plane.  In particular, each decoration has an associated $y$-value, which we will call its \emph{vertical position}.  We also require the following:
\begin{enumerate}[label=\rm{(\arabic*)}]
\item If $\a(d)\neq 0$ and $e$ is non-propagating (loop edge or otherwise), then we allow adjacent blocks on $e$ to be conjoined to form larger blocks, and in particular, if we conjoin all adjacent blocks on $e$, then there is a unique maximal block.
\item If $\a(d)>1$ and $e$ is propagating, then as in (1), we allow adjacent blocks on $e$ to be conjoined to form larger blocks, and in particular, if we conjoin all adjacent blocks on $e$, then there is a unique maximal block.
\item \label{unusual} If $\a(d)=1$ and $e$ is propagating, then we allow $e$ to be decorated subject to the following constraints.
\begin{enumerate}[label=\rm{(\alph*)}]
\item All decorations occurring on propagating edges must have vertical position lower (respectively, higher) than the vertical positions of decorations occurring on the (unique) non-propagating edge in the north face (respectively, south face) of $d$.

\item If $\mathbf{b}$ is a block of decorations occurring on $e$, then no other decorations occurring on any other propagating edges may have vertical position in the range of vertical positions that $\mathbf{b}$ occupies.

\item If $\mathbf{b}_{i}$ and $\mathbf{b}_{i+1}$ are two adjacent blocks occurring on $e$, then they may be conjoined to form a larger block only if the previous requirements are not violated.
\end{enumerate}
\end{enumerate}

\begin{remark}
Note that \ref{unusual} above is an unusual requirement for decorated diagrams.  We require this feature to ensure faithfulness of our diagrammatic representation on monomial basis elements of $\TL(\C_{n})$ indexed by the type I elements of $W(\C_{n})$.
\end{remark}

If an edge has no decorations on it, then we say that it is \emph{undecorated}, and in this case, we can think of the edge as being adorned with $1 \in \B$.  A diagram is undecorated if all of its edges are undecorated.  We require that all diagrams with $\a$-value 0 be undecorated.  In particular, the unique diagram $d_{e}$ having $\a$-value 0 and no loops is undecorated.

%Last sentence of the above paragraph used to be:  ``We require that the unique diagram $d_{e}$ having $\a$-value $0$ be undecorated.''  But this wrong; need all diagrams with $\a$-value 0 to be undecorated (including those with loops).

\begin{definition}\label{decorated.pseudo.diagram.def}
A \emph{concrete $A$-decorated pseudo $k$-diagram} is any concrete $k$-diagram decorated by elements of the decoration set $X$ that satisfies the conditions given above.
\end{definition}

\begin{example}\label{first.dec.diagram.ex}
Let $A=\V$, so that $\{\bcirc, \btri, \wcirc, \wtri\}$ is our decoration set.
\begin{enumerate}[label=\rm{(\alph*)}]
\item \label{first.dec.diagram.ex.diagram1} Here is an example of a concrete $\V$-decorated pseudo $5$-diagram.

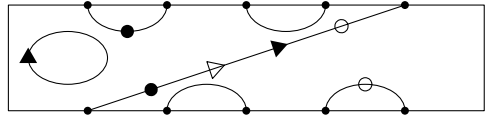
\begin{center}
%-- New mfpic environment, number 80 of 189. (size of end split: 2, should be 2)  ------------------->
\includegraphics{ThesisFigs1.080}
\end{center}

In this example, there are no restrictions on the relative vertical position of decorations since the $\a$-value is greater than 1.  

%Fixed trivial typo:  used to say ``but with $a$-value is 1.''
\item \label{first.dec.diagram.ex.diagram2} Here is another example of a concrete $\V$-decorated pseudo $5$-diagram, but with $\a$-value 1.

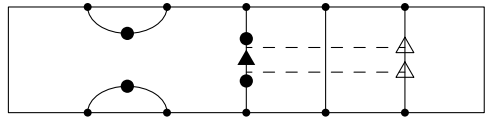
\begin{center}
%-- New mfpic environment, number 81 of 189. (size of end split: 2, should be 2)  ------------------->
\includegraphics{ThesisFigs1.081}
\end{center}

We use the horizontal dotted lines to indicate that the three closed decorations on the leftmost propagating edge are in three distinct blocks.  We cannot conjoin these three decorations to form a single block because there are decorations on the last propagating edge occupying vertical positions between them.  Similarly, the open decorations on the last propagating edge form two distinct blocks that may not be conjoined.  

\item \label{first.dec.diagram.ex.diagram3} Lastly, here is an example of a concrete $\V$-decorated pseudo $6$-diagram.

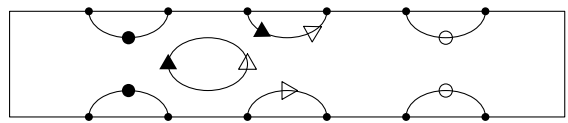
\begin{center}
%-- New mfpic environment, number 82 of 189. (size of end split: 2, should be 2)  ------------------->
\includegraphics{ThesisFigs1.082}
\end{center}
\end{enumerate}

\end{example}

Note that an isotopy of a concrete $A$-decorated pseudo $k$-diagram $d$ that preserves the faces of the standard $k$-box may not preserve the relative vertical position of the decorations even if it is mapping $d$ to an equivalent diagram.  However, if two diagrams are vertically equivalent then the relative vertical position of decorations will be preserved.  This is a bit too restrictive for us.  The only time equivalence is an issue is when $\a(d)=1$.  In this case, we wish to preserve the relative vertical position of the blocks.  We define two concrete pseudo $A$-decorated $k$-diagrams to be \emph{$A$-equivalent} if we can isotopically deform one diagram into the other such that any intermediate diagram is also a concrete pseudo $A$-decorated $k$-diagram.  Note that we do allow decorations from the same maximal block to pass each other's vertical position.  

\begin{definition}\label{def.T_k(A)}
An \emph{$A$-decorated pseudo $k$-diagram} is defined to be an equivalence class of $A$-equivalent concrete $A$-decorated pseudo $k$-diagrams.  We denote the set of $A$-decorated pseudo $k$-diagrams by $T_{k}(A)$.
\end{definition}

\begin{remark}
As in Remark \ref{diagram_repn}, when representing an $A$-decorated pseudo $k$-diagram with a drawing, we pick an arbitrary concrete representative among a continuum of equivalent choices.  When no confusion will arise, we will not make a distinction between a concrete $A$-decorated pseudo $k$-diagram and the equivalence class that it represents. 
\end{remark}

We wish to generalize Definition \ref{defn_P_{k}} to the case of $A$-decorated $k$-diagrams.  First, we need a lemma to justify that our multiplication is well-defined and associative.

\begin{lemma}\label{a-value=1}
Let $d$ be a diagram with $\a(d)=1$.  Suppose that the unique non-propagating edge in the north face of $d$ joins $i$ to $i+1$.  Let $d'$ be any other diagram.  Then $\a(d'd)=1$ if and only if $\a(d')=1$ and the unique non-propagating edge in the south face of $d'$ joins either $\mathrm{(a)}$ $(i-1)'$ to $i'$; $\mathrm{(b)}$ $i'$ to $(i+1)'$; or $\mathrm{(c)}$ $(i+1)'$ to $(i+2)'$.
\end{lemma}

\begin{proof}
First, assume that $\a(d'd)=1$.  It is a general fact that $\a(d'd)\geq \a(d')$, which implies that $\a(d')=1$.  

\bigskip

Conversely, assume that $\a(d)=1$ and that the unique non-propagating edge in the south face of $d'$ joins either (a) $(i-1)'$ to $i'$; (b) $i'$ to $(i+1)'$; or (c) $(i+1)'$ to $(i+2)'$.  

\bigskip

Assume that we are in situation (a).  Suppose that the propagating edge leaving node $(i+1)'$ in the south face of $d'$ is connected to node $j$ in the north face.  Also, suppose that the propagating edge leaving node $i-1$ in the north face of $d$ is connected to node $l'$ in the south face.  Then $d'd$ has a propagating edge joining node $j$ to node $l'$.  Furthermore, the only non-propagating edge in the north (respectively, south) face of $d'd$ is the same as the unique non-propagating edge in the north (respectively, south) face of $d'$ (respectively, $d$).  It follows that $\a(d'd)=1$.  

\bigskip

Next, assume (b) happens.  Then $d'd$ has one more loop than the sum total of loops from $d'$ and $d$.  Furthermore, the only non-propagating edge in the north (respectively, south) face of $d'd$ is the same as the unique non-propagating edge in the north (respectively, south) face of $d'$ (respectively, $d$), and so $\a(d'd)=1$.  

\bigskip

Lastly, if (c) happens, then the proof that $\a(d'd)=1$ is symmetric to case (a).
\end{proof}

\begin{definition}\label{A-decorated_diagram_algebra}
Let $(A, \B)$ be an associative algebra as in the beginning of this section.  We define $\P_{k}(A)$ to be the free $R$-module having the $A$-decorated pseudo $k$-diagrams $T_{k}(A)$ as a basis.  We define multiplication in $\P_{k}(A)$ by defining multiplication in the case where $d$ and $d'$ are basis elements of $\P_{k}(A)$, and then extend bilinearly.  To calculate the product $d'd$, concatenate $d'$ and $d$ (as in Definition \ref{defn_P_{k}}).  While maintaining $A$-equivalence, conjoin adjacent blocks.
\end{definition}

We emphasize that we are \textit{not} currently applying any of the relations of $A$.  We are simply adorning the edges with decorations subject to certain constraints and describing rules for conjoining blocks when multiplying diagrams.

\begin{remark}
We claim that the multiplication defined above turns $\P_{k}(A)$ into a well-defined associative $R$-algebra.  This claim follows from arguments in \cite[\textsection 3]{Martin.P;Green.R;Parker.A:A} and Lemma \ref{a-value=1} above.  The only case that requires serious consideration is when multiplying two diagrams that both have $\a$-value $1$.  If $\a(d)=\a(d')=1$ while $\a(d'd)>1$, then there are no concerns.  However, if $\a(d'd)=1$, then according to Lemma \ref{a-value=1}, if the unique non-propagating edge $e'$ in the south face of $d'$ joins $i'$ to $(i+1)'$, it must be the case that unique non-propagating edge $e$ in the north face of $d$ joins either (a) $i-1$ to $i$; (b) $i$ to $i+1$; or (c) $i+1$ to $i+2$.  If (a) or (c) happens, then the only blocks that get conjoined are the blocks on $e$ and $e'$, which presents no problems.  If (b) happens, then we get a loop edge and we conjoin the blocks from $e$ and $e'$.  As a consequence, it is possible that the block occurring on a propagating edge of $d'$ having the lowest vertical position may be conjoined with the block occurring on a propagating edge of $d$ having the highest vertical position.  This can only happen if these two edges are joined in $d'd$, and regardless, presents no problems.  Since there are no relations to apply, the product of two elements of $T_{k}(A)$ is equal to a single basis element.
\end{remark}

\end{section}

\begin{section}{The $\V$-decorated diagram algebra $\D_{n}$ and admissible diagrams}

For the remainder of this thesis, we will assume that $R=\Z[\delta]$, where $\delta=v+v^{-1}$, and that $A=\V$ is equipped with basis consisting of the identity and all finite alternating products of open and closed decorations.  Let $n\geq 2$.  We now focus our attention on a particular set of diagrams from $\P_{n+2}(\V)$.  

\begin{definition}\label{def.T_n+2^LR(V)}
Let $d \in T_{n+2}(\V)$.  That is, the blocks on the edges of $d$ consist of sequences of the decorations $\bcirc$, $\btri$, $\wtri$, and $\wcirc$.
\begin{enumerate}[label=\rm{(\arabic*)}]
\item An edge of $d$ is called \emph{L-exposed} (respectively, \emph{R-exposed}) if it can be deformed to touch the left (respectively, right) wall of the diagram without crossing any other edges.
\item We call $d$ \emph{L-decorated} (respectively, \emph{R-decorated}) if the only edges labelled with closed (respectively, open) decorations are L-exposed (respectively, R-exposed).
\item We call $d$ \emph{LR-decorated} if $d$ is both L-decorated and R-decorated, with the added constraint that it must be possible to deform decorated edges so as to take open decorations to the left and closed decorations to the right simultaneously.  We denote the set of LR-decorated diagrams from $T_{n+2}(\V)$ by $T_{n+2}^{LR}(\V)$.

\item We say that a decoration on a non-propagating edge $e$ joining $i$ to $j$ (respectively, $i'$ to $j'$) with $i<j$ (respectively, $i'<j'$) is \emph{first} (respectively, \emph{last}) if it is the first (respectively, last) decoration encountered as we traverse $e$ from $i$ to $j$ (respectively, $i'$ to $j'$).

\item Similarly, if $\a(d)=1$, then we say that a decoration on a propagating edge $e$ from $i$ to $j'$ is \emph{first} (respectively, \emph{last}) if it is the first (respectively, last) decoration encountered as we traverse $e$ from $i$ to $j'$.  

\end{enumerate}
\end{definition}

\begin{remark}\label{LR-decorated}
We make several observations.
\begin{enumerate}[label=\rm{(\arabic*)}]
\item The set of LR-decorated diagrams $T_{n+2}^{LR}(\V)$ is infinite since there is no limit to the number of loops that may appear.

\item \label{closed.under.concatenation} Concatenating diagrams cannot change an L-exposed edge to a non-L-exposed edge, and similarly for R-exposed edges. Thus, the set of LR-decorated diagrams is closed under diagram concatenation.

\item \label{closed.left.open.right} If $d$ is an undammed LR-decorated diagram, then all closed decorations occurring on an edge connecting nodes in the north face (respectively, south face) of $d$ must occur before all of the open decorations occurring on the same edge as we travel the edge from the left node to the right node. Otherwise, we would not be able to simultaneously deform decorated edges to the left and right.  Furthermore, if an edge joining nodes in the north face of $d$ is adorned with an open (respectively, closed) decoration, then no non-propagating edge occurring to the right (respectively, left) in the north face may be adorned with closed (respectively, open) decorations.  We have an analogous statement for non-propagating edges in the south face.

%This remark was wrong in actual thesis.  It used to say ``Loops can only be decorated if $d$ is undammed.''  Of course, I meant decorated by both types of decorations.
\item \label{undammed.dec.loops} Loops can only be decorated by both types of decorations if $d$ is undammed.  Again, we would not be able to simultaneously deform decorated edges to the left and right, otherwise.

\item \label{both.decs.one.prop} If $d$ is a dammed LR-decorated diagram, then closed decorations (respectively, open decorations) only occur to the left of (respectively, right of) and possibly on the leftmost (respectively, rightmost) propagating edge.  The only way a propagating edge can have decorations of both types is if there is a single propagating edge, which can only happen if $n+2$ is odd.
\end{enumerate}
\end{remark}

\begin{example}
The diagram in Example \ref{first.dec.diagram.ex}\ref{first.dec.diagram.ex.diagram3} is an example that illustrates conditions \ref{closed.left.open.right} and \ref{undammed.dec.loops} of Remark \ref{LR-decorated} above while the diagram in Example \ref{first.dec.diagram.ex}\ref{first.dec.diagram.ex.diagram1} illustrates condition \ref{both.decs.one.prop}.
\end{example}

\begin{definition}\label{def.P_n+2^LR(V)}
We denote the $\Z[\delta]$-submodule of $P_{n+2}(\V)$ spanned by the LR-decorated diagrams by $\P_{n+2}^{LR}(\V)$.  
\end{definition}

The following proposition follows immediately from Remark \ref{LR-decorated}\ref{closed.under.concatenation}.

\begin{proposition}
The $\Z[\delta]$-submodule $\P_{n+2}^{LR}(\V)$ is a $\Z[\delta]$-subalgebra of $P_{n+2}(\V)$. \hfill $\qed$
\end{proposition}

We remark that since the set of LR-decorated diagrams is infinite, that $\P_{n+2}^{LR}(\V)$ is an infinite dimensional algebra.  Next, we define a particular quotient of $\P_{n+2}^{LR}(\V)$, which is closely related to our diagram algebra of interest.

\begin{definition}\label{big.diagram.alg.defn}
Let $\widehat{\P}_{n+2}^{LR}(\V)$ be the associative $\Z[\delta]$-algebra equal to the quotient of $\P_{n+2}^{LR}(\V)$ by the following set of relations:
\begin{enumerate}[label=\rm{(\arabic*)}]
\item \begin{tabular}[c]{@{}c@{}}
%-- New mfpic environment, number 83 of 189. (size of end split: 2, should be 2)  ------------------->
\includegraphics{ThesisFigs1.083}
\end{tabular}
= 
\begin{tabular}[c]{@{} c@{}}
%-- New mfpic environment, number 84 of 189. (size of end split: 2, should be 2)  ------------------->
\includegraphics{ThesisFigs1.084}
\end{tabular};

\item \begin{tabular}[c]{@{}c@{}}
%-- New mfpic environment, number 85 of 189. (size of end split: 2, should be 2)  ------------------->
\includegraphics{ThesisFigs1.085}
\end{tabular}
= 
\begin{tabular}[c]{@{} c@{}}
%-- New mfpic environment, number 86 of 189. (size of end split: 2, should be 2)  ------------------->
\includegraphics{ThesisFigs1.086}
\end{tabular};

\item \begin{tabular}[c]{@{} c@{}}
%-- New mfpic environment, number 87 of 189. (size of end split: 2, should be 2)  ------------------->
\includegraphics{ThesisFigs1.087}
\end{tabular}
=  
\begin{tabular}[c]{@{} c@{}}
%-- New mfpic environment, number 88 of 189. (size of end split: 2, should be 2)  ------------------->
\includegraphics{ThesisFigs1.088}
\end{tabular}
=   2
\begin{tabular}[c]{@{} c@{}}
%-- New mfpic environment, number 89 of 189. (size of end split: 2, should be 2)  ------------------->
\includegraphics{ThesisFigs1.089}
\end{tabular};

\item \begin{tabular}[c]{@{} c@{}}
%-- New mfpic environment, number 90 of 189. (size of end split: 2, should be 2)  ------------------->
\includegraphics{ThesisFigs1.090}
\end{tabular}
=  
\begin{tabular}[c]{@{} c@{}}
%-- New mfpic environment, number 91 of 189. (size of end split: 2, should be 2)  ------------------->
\includegraphics{ThesisFigs1.091}
\end{tabular}
=   2
\begin{tabular}[c]{@{} c@{}}
%-- New mfpic environment, number 92 of 189. (size of end split: 2, should be 2)  ------------------->
\includegraphics{ThesisFigs1.092}
\end{tabular};

\item \begin{tabular}[c]{@{} c@{}}
%-- New mfpic environment, number 93 of 189. (size of end split: 2, should be 2)  ------------------->
\includegraphics{ThesisFigs1.093}
\end{tabular}
=  
\begin{tabular}[c]{@{} c@{}}
%-- New mfpic environment, number 94 of 189. (size of end split: 2, should be 2)  ------------------->
\includegraphics{ThesisFigs1.094}
\end{tabular}
=  
\begin{tabular}[c]{@{} c@{}}
%-- New mfpic environment, number 95 of 189. (size of end split: 2, should be 2)  ------------------->
\includegraphics{ThesisFigs1.095}
\end{tabular}
=\   $\delta$;
\end{enumerate}
where the decorations on the edges above represent adjacent decorations of the same block.

\end{definition}

Note that with the exception of the relations involving loops, multiplication in $\widehat{\P}_{n+2}^{LR}(\V)$ is inherited from the relations of the decoration set $\V$.   Also, observe that all of the relations are local in the sense that a single reduction only involves a single edge.  Furthermore, as a consequence of the relations above, we also have the following relations:
\begin{enumerate}[label=\rm{(\arabic*)}, resume]
\item  \begin{tabular}[c]{@{}c@{}}
%-- New mfpic environment, number 96 of 189. (size of end split: 2, should be 2)  ------------------->
\includegraphics{ThesisFigs1.096}
\end{tabular}
= 2
\begin{tabular}[c]{@{} c@{}}
%-- New mfpic environment, number 97 of 189. (size of end split: 2, should be 2)  ------------------->
\includegraphics{ThesisFigs1.097}
\end{tabular}; 

\item \begin{tabular}[c]{@{}c@{}}
%-- New mfpic environment, number 98 of 189. (size of end split: 2, should be 2)  ------------------->
\includegraphics{ThesisFigs1.098}
\end{tabular}
= 2
\begin{tabular}[c]{@{} c@{}}
%-- New mfpic environment, number 99 of 189. (size of end split: 2, should be 2)  ------------------->
\includegraphics{ThesisFigs1.099}
\end{tabular}.
\end{enumerate}

\begin{example}

Here is an example of multiplication of three diagrams in $\widehat{\P}_{n+2}^{LR}(\V)$.

{\setlength\arraycolsep{1pt}
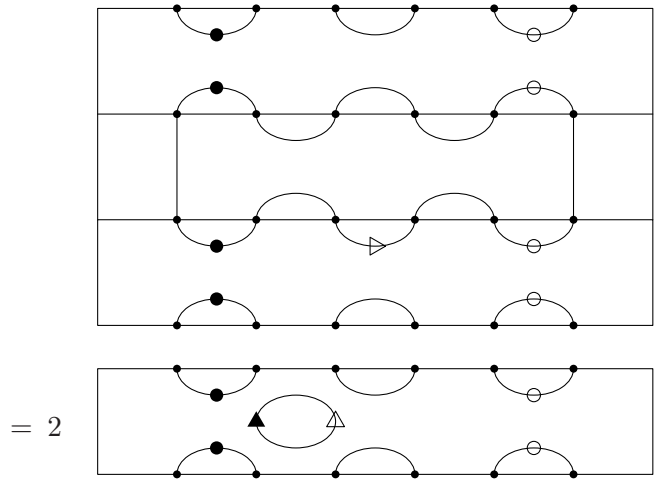
\begin{eqnarray*}
& & \ \begin{tabular}[c]{l}
%-- New mfpic environment, number 100 of 189. (size of end split: 2, should be 2)  ------------------->
\includegraphics{ThesisFigs2.001}
\end{tabular} \\ 
= &\  2 &\ \begin{tabular}[c]{l}
%-- New mfpic environment, number 101 of 189. (size of end split: 2, should be 2)  ------------------->
\includegraphics{ThesisFigs2.002}
\end{tabular}
\end{eqnarray*}}

Here is a second example, where each of the diagrams and their product have $\a$-value 1.

{\setlength\arraycolsep{1pt}
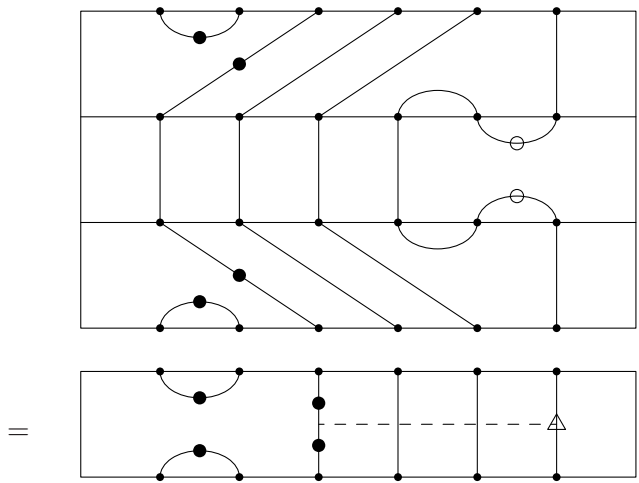
\begin{eqnarray*}
& & \ \begin{tabular}[c]{l}
%-- New mfpic environment, number 102 of 189. (size of end split: 2, should be 2)  ------------------->
\includegraphics{ThesisFigs2.003}
\end{tabular} \\
= &\  &\ \begin{tabular}[c]{l}
%-- New mfpic environment, number 103 of 189. (size of end split: 2, should be 2)  ------------------->
\includegraphics{ThesisFigs2.004}
\end{tabular}
\end{eqnarray*}}

Again, we use the dotted line to emphasize that the two closed decorations on the leftmost propagating edge belong to distinct blocks.

\end{example}

Our immediate goal is to show that a basis for $\widehat{\P}_{n+2}^{LR}(\V)$ consists of the set of LR-decorated diagrams having blocks corresponding to nonidentity basis elements in $\V$.  That is, no block may contain adjacent decorations of the same type (open or closed).  To accomplish this task, we will make use of a diagram algebra version of Bergman's diamond lemma.  For other examples of this type of application of Bergman's diamond lemma, see  \cite{Green.R;Martin.P;Parker.A:A} and \cite{Martin.P;Green.R;Parker.A:A}.

\bigskip

Define the function $r: T_{n+2}^{LR}(\V) \to T_{n+2}(\emptyset)$ via
	$$r(d)=d \text{ with all decorations and loops removed}.$$
In the literature, if $d$ has no loops, then $r(d)$ is sometimes referred to as the ``shape'' of $d$.  For example, if 
%-- New mfpic environment, number 104 of 189. (size of end split: 2, should be 2)  ------------------->
%-- New mfpic environment, number 105 of 189. (size of end split: 2, should be 2)  ------------------->
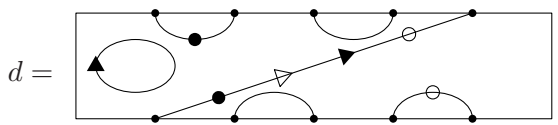
\begin{align*}
d &=\ \begin{tabular}[c]{@{} c@{}}
\includegraphics{ThesisFigs2.005}
\end{tabular},\\
\intertext{then}
r(d) & =\ \begin{tabular}[c]{@{} c@{}}	
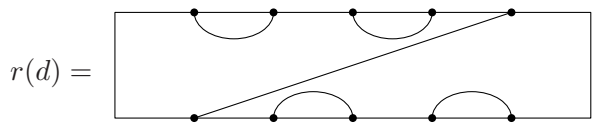
\includegraphics{ThesisFigs2.006}
\end{tabular}.
\end{align*}
Next, define a function $h: T_{n+2}^{LR}(\V) \to \Z^{+}\cup \{0\}$ via
	$$h(d)=\text{sum of the number of decorations and the number of loops}.$$
For example, if $d$ is the diagram from above, then $h(d)=8$ (1 loop plus 7 decorations).  Define $\leq_{\widehat{\P}}$ on $T_{n+2}^{LR}(\V)$ via $d < _{\widehat{\P}} d'$ if and only if $r(d)=r(d')$ and $h(d) < h(d')$.

\bigskip

Let $\S$ be the collection of reductions determined by the relations of $\widehat{\P}_{n+2}^{LR}(\V)$ given Definition \ref{big.diagram.alg.defn}.  If we apply any single reduction (loop removal or any other local reduction) to a diagram from $\widehat{\P}_{n+2}^{LR}(\V)$, then we obtain a scalar multiple of a strictly smaller diagram with respect to $\leq_{\widehat{\P}}$.  Thus, our reduction system $\S$ (i.e., diagram relations) is compatible with $\leq_{\widehat{\P}}$.  Now, suppose that $d <_{\widehat{\P}} d'$ and let $d''$ be any other element from $\widehat{\P}_{n+2}^{LR}(\V)$.  Then
	$$r(d''d)=r(d''d')$$
and
	$$r(dd'')=r(d'd'').$$
Since $r(d)=r(d')$, multiplying $d$ or $d'$ on the same side by $d''$ will increase the number of decorations and number of loops by the same amount.  So, we have
	$$h(dd'') < h(d'd'')$$
and
	$$h(d''d) < h(d''d').$$
Therefore, $dd'' <_{\widehat{\P}} d'd''$ and $d''d <_{\widehat{\P}} d''d'$.  This shows that $\leq_{\widehat{\P}}$ is a semigroup partial order on $T_{n+2}^{LR}(\V)$.  Clearly, $\leq_{\widehat{\P}}$ satisfies the descending chain condition.

\begin{proposition}\label{diagram.independence}
The set of LR-decorated diagrams having maximal blocks corresponding to nonidentity basis elements in $\V$ forms a basis for $\widehat{\P}_{n+2}^{LR}(\V)$.
\end{proposition}

\begin{proof}
Let $\leq_{\widehat{\P}}$ be as above.  Following the setup of Bergman's diamond lemma, it remains to show that all of the ambiguities are resolvable.  By inspecting the relations of Definition \ref{big.diagram.alg.defn}, we see that there are no inclusion ambiguities, so we only need to check that the overlap ambiguities are resolvable.  Let $d$ be a diagram from $\widehat{\P}_{n+2}^{LR}(\V)$ and suppose that there are two competing reductions that we could apply.  If both reductions involve the same non-loop edge, then the ambiguity is easily seen to be resolvable since the algebra $\V$ is associative.  In particular, in the $\a$-value $1$ case, the reductions could involve two distinct blocks on the same edge, in which case, the order that we apply the reductions is immaterial.  If the reductions involve distinct edges, loop edges or otherwise, the ambiguity is quickly seen to be resolvable since the reductions commute.  Lastly, suppose that the two competing reductions involve the same loop edge.  There are three possibilities for this loop edge: (a) the loop is undecorated, (b) the loop carries only one type of decoration (open or closed), and (c) the loop carries both types of symbols.  Note that (a) cannot happen since then there could not have been two competing reductions involving this edge to apply.  If (b) happens, then any ambiguity involving this loop edge (including removing the loop) is resolvable since each of $V_{3}$ and $V'_{3}$ are commutative and associative.  Finally, assume (c) happens.  Note that the nature of our relations prevents the complete elimination of closed (respectively, open) decorations from this loop edge.  Since all loop relations involve either undecorated loops or loops decorated with a single type of decoration, this loop edge can never be removed.  Since $\V$ is associative and none of the relations involve both decoration types at the same time, the ambiguity is easily seen to be resolvable since the reductions commute.  According to Bergman's diamond lemma (Theorem \ref{bergman}), we can conclude that the LR-decorated diagrams having no relations to apply are a basis, as desired.
\end{proof}

The algebra $\widehat{\P}_{n+2}^{LR}(\V)$ is still too large for our purposes.  Next, we define a few special diagrams that will form a generating set for a much smaller algebra that will be our object of interest.  Define the diagrams $d_{1}, d_{2}, \dots, d_{n+1}$ via
%-- New mfpic environment, number 106 of 189. (size of end split: 2, should be 2)  ------------------->
%-- New mfpic environment, number 107 of 189. (size of end split: 2, should be 2)  ------------------->
%-- New mfpic environment, number 108 of 189. (size of end split: 2, should be 2)  ------------------->
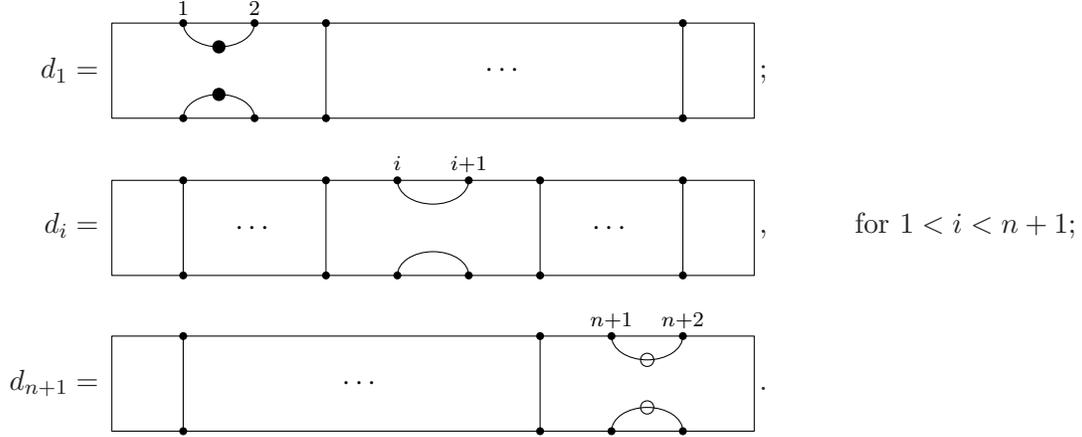
\begin{align*}
d_{1} &=
\begin{tabular}[c]{@{} c@{}}
\includegraphics{ThesisFigs2.007}
\end{tabular};\\ 
d_{i} &=
\begin{tabular}[c]{@{} c@{}}
\includegraphics{ThesisFigs2.008}
\end{tabular}, & \text{for }1<i<n+1; \\
d_{n+1} &= 
\begin{tabular}[c]{@{} c@{}}
\includegraphics{ThesisFigs2.009}
\end{tabular}. &
\end{align*}

We will refer to each of $d_{1}, d_{2}, \dots, d_{n+1}$ as a \emph{simple diagram}.  Note that the simple diagrams are basis elements of $\widehat{\P}_{n+2}^{LR}(\V)$.

\begin{definition}\label{def.D_n}
Let $\D_{n}$ be the $\Z[\delta]$-subalgebra of $\widehat{\P}_{n+2}^{LR}(\V)$ generated (as a unital algebra) by $d_{1}, d_{2}, \dots, d_{n+1}$ with multiplication inherited from $\widehat{\P}_{n+2}^{LR}(\V)$.
\end{definition}

The next definition describes the set of diagrams that will turn out to form a basis for $\D_{n}$.  This definition is motivated by the definition of $B$-admissible (after an appropriate change of basis) given by R.M. Green in \cite[Definition 2.2.4]{Green.R:H} for diagrams in the context of type $B$.  Since $\C$ is, in some sense, type $B$ on one end and type $B'$ on the other, the general idea is to build the axioms of $B$-admissible into our definition of $\C$-admissible at both ends. 

\begin{definition}\label{admissible.def}
Let $d$ be an LR-decorated diagram.  Then we say that $d$ is \emph{$\C$-admissible}, or simply \emph{admissible}, if the following axioms are satisfied.
\begin{enumerate}[label=\rm{(C\arabic*)}]

\item \label{C1}The only loops that may appear are equivalent to the following.

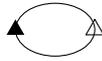
\begin{center}
%-- New mfpic environment, number 109 of 189. (size of end split: 2, should be 2)  ------------------->
\includegraphics{ThesisFigs2.010}
\end{center}

\item \label{C2} If $d$ is undammed (which can only happen if $n$ is even), then the (non-propagating) edges joining nodes $1$ and $1'$ (respectively, nodes $n+2$ and $(n+2)'$) must be decorated with a $\bcirc$ (respectively, $\wcirc$).  Furthermore, these are the only $\bcirc$ (respectively, $\wcirc$) decorations that may occur on $d$ and must be the first (respectively, last) decorations on their respective edges.

\item \label{C3} Assume $d$ has exactly one propagating edge $e$ (which can only happen if $n$ is odd). Then $e$ may be decorated by an alternating sequence of $\btri$ and $\wtri$ decorations.  If $e$ is decorated by both open and closed decorations and is connected to node 1 (respectively, $1'$), then the first (respectively, last) decoration occurring on $e$ must be a $\bcirc$.  Similarly, if $e$ is connected to node $n+2$ (respectively, $(n+2)'$), then the first (respectively, last) decoration occurring on $e$ must be a $\wcirc$.  If $e$ joins $1$ to $1'$ (respectively, $n+2$ to $(n+2)'$) and is decorated by a single decoration, then $e$ is decorated by a single $\btri$ (respectively, $\wtri$).  Furthermore, if there is a non-propagating edge connected to $1$ or $1'$ (respectively, $n+2$ or $(n+2)'$) it must be decorated only by a single $\bcirc$ (respectively, $\wcirc$).  Finally, no other $\bcirc$ or $\wcirc$ decorations appear on $d$.

\item \label{C4} Assume that $d$ is dammed with $\a(d)>1$ and has more than one propagating edge.  If there is a propagating edge joining $1$ to $1'$ (respectively, $n+2$ to $(n+2)'$), then it is decorated by a single $\btri$ (respectively, $\wtri$).  Otherwise, both edges leaving either of $1$ or $1'$ (respectively, $n+2$ or $(n+2)'$) are each decorated by a single $\bcirc$ (respectively, $\wcirc$) and there are no other $\bcirc$ or $\wcirc$ decorations appearing on $d$.

\item \label{C5} If $\a(d)=1$, then the western end of $d$ is equal to one of the following:

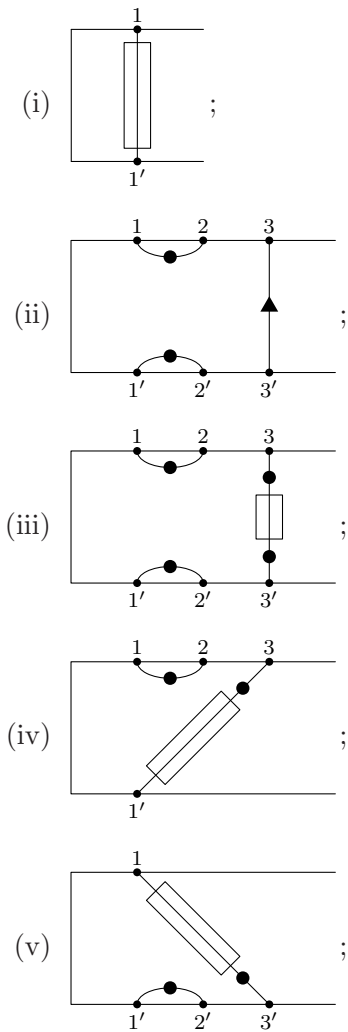
\begin{enumerate}[label=\rm{(\roman*)}]
\item \begin{tabular}[c]{@{}c@{}}%1
%-- New mfpic environment, number 110 of 189. (size of end split: 2, should be 2)  ------------------->
\includegraphics{ThesisFigs2.011}
\end{tabular};

\item \begin{tabular}[c]{@{}c@{}}%2
%-- New mfpic environment, number 111 of 189. (size of end split: 2, should be 2)  ------------------->
\includegraphics{ThesisFigs2.012}
\end{tabular};

\item \begin{tabular}[c]{@{}c@{}}%2
%-- New mfpic environment, number 112 of 189. (size of end split: 2, should be 2)  ------------------->
\includegraphics{ThesisFigs2.013}
\end{tabular};

\item \begin{tabular}[c]{@{}c@{}}%3
%-- New mfpic environment, number 113 of 189. (size of end split: 2, should be 2)  ------------------->
\includegraphics{ThesisFigs2.014}
\end{tabular};

\item \begin{tabular}[c]{@{}c@{}}%4
%-- New mfpic environment, number 114 of 189. (size of end split: 2, should be 2)  ------------------->
\includegraphics{ThesisFigs2.015}
\end{tabular};
\end{enumerate}
where the rectangle represents a sequence of blocks (possibly empty) such that each block is a single $\btri$ and the diagram in (ii) can only occur if $d$ is not decorated by any open decorations.  Also, the occurrences of the $\bcirc$ decorations occurring on the propagating edge have the highest (respectively, lowest) relative vertical position of all decorations occurring on any propagating edge.  In particular, if the rectangle in (iv) (respectively, (v)) is empty, then the $\bcirc$ decoration has the highest (respectively, lowest) relative vertical position among all decorations occurring on propagating edges.  We have an analogous requirement for the eastern end of $e$, where the closed decorations are replaced with open decorations.  Furthermore, if there is a non-propagating edge connected to $1$ or $1'$ (respectively, $n+2$ or $(n+2)'$) it must be decorated only by a single $\bcirc$ (respectively, $\wcirc$).  Finally, no other $\bcirc$ or $\wcirc$ decorations appear on $d$.

\end{enumerate}

Let $\Diag^{b}_{n}(\V)$ denote the set of all $\C$-admissible $(n+2)$-diagrams.
\end{definition}

\begin{remark}\label{comment.admissible}
We collect a few comments concerning admissible diagrams.
\begin{enumerate}[label=\rm{(\arabic*)}]
\item Note that the only time an admissible diagram $d$ can have an edge adorned with both open and closed decorations is if $d$ is undammed (which only happens when $n$ is even) or if $d$ has a single propagating edge (which only happens when $n$ is odd).  See Example \ref{first.dec.diagram.ex} (a) and (c) for examples that demonstrate this restriction.

\item \label{alternating.decorations.a=1} If $d$ is an admissible diagram with $\a(d)=1$, then the restrictions on the relative vertical position of decorations on propagating edges along with axiom \ref{C5} imply that the relative vertical positions of closed decorations on the leftmost propagating edge and open decorations on the rightmost propagating edge must alternate.  In particular, the number of closed decorations occurring on the leftmost propagating edge differs from the number of open decorations occurring on the rightmost propagating edge by at most 1.  For example, if 
$$d=\begin{tabular}[c]{@{}c@{}}
%-- New mfpic environment, number 115 of 189. (size of end split: 2, should be 2)  ------------------->
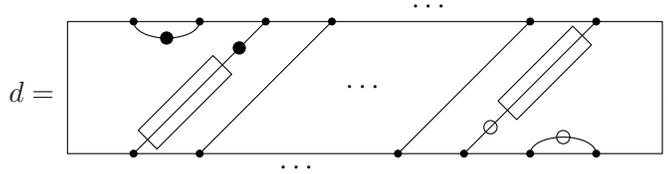
\includegraphics{ThesisFigs2.016}
\end{tabular},$$
where the leftmost propagating edge carries $k$ $\btri$ decorations, then the rightmost propagating edge must carry $k$ $\wtri$ decorations, as well. 

\item Note that the rectangle in diagram (iii) of axiom \ref{C5} cannot be empty; otherwise, the leftmost propagating edge would not be decorated by a basis element of $\V$.

\item It is clear that $\Diag^{b}_{n}(\V)$ is an infinite set.  If an admissible diagram $d$ is undammed, then there is no limit to the number of loops given in axiom \ref{C1} that may occur.  Also, if $d$ is an admissible diagram with exactly one propagating edge, then there is no limit to the width of the block of decorations that may occur on the lone propagating edge.  Furthermore, if $d$ is admissible with $\a(d)=1$, then there is no limit to the number of $\btri$-blocks (respectively, $\wtri$-blocks) that may occur on the leftmost (respectively, righmost) propagating edge.  

\item \label{admissible.basis} Each of the admissible diagrams is a basis element of $\widehat{\P}_{n+2}^{LR}(\V)$.

\item The symbol $b$ in the notation $\Diag^{b}_{n}(\V)$ is to emphasize that we are constructing a set of diagrams that are intended to correspond to the monomial basis of $\TL(\C_{n})$.  A topic of future research is to construct diagrams that correspond to the ``canonical basis'' of $\TL(\C_{n})$, which is defined for arbitrary Coxeter groups in \cite{Green.R;Losonczy.J:A}.

\end{enumerate}
\end{remark}

\begin{definition}\label{def.mathcalM[Diag_n^b(V)}
Let $\mathcal{M}[\Diag_{n}^{b}(\V)]$ be the $\Z[\delta]$-submodule of $\widehat{\P}_{n+2}^{LR}(\V)$ spanned by the admissible diagrams.
\end{definition}

\begin{proposition}\label{prop.admissible.basis}
The set of admissible diagrams $\Diag_{n}^{b}(\V)$ is a basis for the module $\mathcal{M}[\Diag_{n}^{b}(\V)]$.
\end{proposition}

\begin{proof}
This follows immediately from \ref{admissible.basis} in Remark \ref{comment.admissible}.
\end{proof}

In the next chapter we will show that $\mathcal{M}[\Diag_{n}^{b}(\V)]$ is, in fact, a subalgebra of $\widehat{\P}_{n+2}^{LR}(\V)$ equal to $\D_{n}$.   And, in the final chapter, we will prove our main result, which states that $\D_{n}$ is a faithful representation of $\TL(\C_{n})$, where the admissible diagrams correspond to the monomial basis.  

\end{section}

\begin{section}{Temperley--Lieb diagram algebras of type $B$}

We conclude this chapter by discussing how $\TL(B_{n})$ and $\TL(B'_{n})$ are related to $\D_{n}$.

\begin{definition}\label{def.DTL(B_n)}
Let $\DTL(B_{n})$ and $\DTL(B'_{n})$ denote the subalgebras of $\D_{n}$ generated by $d_{1}, d_{2}, \dots, d_{n}$ and $d_{2}, d_{3}, \dots, d_{n+1}$, respectively.  We refer to $\DTL(B_{n})$ (respectively, $\DTL(B'_{n})$ as the \emph{Temperley--Lieb diagram algebra of type $B$} (respectively, \emph{type $B'$}). 
\end{definition}

It is clear that $\DTL(B_{n})$ (respectively, $\DTL(B'_{n})$) consist entirely of L-decorated (respectively, R-decorated) diagrams.  Also, note that all of the technical requirements about how to decorate a diagram $d$ when $\a(d)=1$ are irrelevant since only the leftmost (respectively, rightmost) propagating edge can carry decorations in  $\DTL(B_{n})$ (respectively, $\DTL(B'_{n})$).  The following fact is implicit in \cite[\textsection 2]{Green.R:H} after the appropriate change of basis involving a change of basis on the decoration set.

\begin{proposition}
As $\Z[\delta]$-algebras, $\TL(B_{n}) \cong \DTL(B_{n})$ and $\TL(B'_{n}) \cong \DTL(B'_{n})$, where each isomorphism is determined by
	$$b_{i} \mapsto d_{i}$$
for the appropriate restrictions on $i$.  \hfill $\qed$
\end{proposition}

Recall from Lemma \ref{cheby.basis.verlinde} that $\{u_{0}, u_{1}, u_{2}\}$ is an alternate basis for $V_{3}$, where $u_{0}$, $u_{1}$, and $u_{2}$ are the images of $U_{0}(x)=1$, $U_{1}(x)=x$, and $U_{2}(x)=x^{2}-1$, respectively.   Using the change of basis $v_{i} \mapsto u_{i}$ (respectively, $v'_{i} \mapsto u'_{i}$) on the decoration set $V_{3}$ (respectively, $V'_{3}$), the basis diagrams in $\DTL(B_{n})$ (respectively, $\DTL(B'_{n})$) become $B$-admissible in the sense of \cite{Green.R:H, Green.R:M}.  Moreover, it is easily verified that the axioms for $B$-admissible given in \cite[Definition 2.2.4]{Green.R:H} imply (again, under the appropriate change of basis involving the decoration set) that all of the basis diagrams in $\DTL(B_{n})$ and $\DTL(B'_{n})$ are $\C$-admissible.  

\end{section}

\end{chapter}

%%%%%%%%%% Chapter 9 %%%%%%%%%%%

\begin{chapter}{A basis for $\D_{n}$}

The main result of this chapter will be that the $\C$-admissible diagrams form a basis for $\D_{n}$.  To achieve this end, we require several intermediate results.

\begin{section}{Preparatory lemmas}

If $d$ is an admissible diagram, then we say that a non-propagating edge joining $i$ to $i+1$ (respectively, $i'$ to $(i+1)'$) is \emph{simple} if it is identical to the edge joining $i$ to $i+1$ (respectively, $i'$ to $(i+1)'$) in the simple diagram $d_{i}$.  That is, an edge is simple if it joins adjacent vertices in the north face (respectively, south face) and is undecorated, except when one of the vertices is 1 or $1'$ (respectively, $n+2$ or $(n+2)'$), in which case it is decorated by only a single $\bcirc$ (respectively, $\wcirc$).  

\bigskip

The next six lemmas mimic Lemmas 5.1.4--5.1.7 in \cite{Green.R:H}.  The proof of each lemma is immediate and throughout we assume that $d$ is admissible.

\begin{lemma}\label{nonprops.connect.adjacent.verts}
Assume that in the north face of $d$ there is an edge, say $e$, connecting node $j$ to node $i$, and assume that there is another undecorated edge, say $e'$, connecting node $i+1$ to node $k$ with $j<i$ and $i+1<k<n+2$.  Further, suppose that $j$ and $k$ are chosen so that $|j-k|$ is minimal.  Then $d_{i}d$ is the admissible diagram that results from $d$ by removing $e'$, disconnecting $e$ from node $i$ and reattaching it to node $k$, and adding a simple edge to $i$ and $i+1$ (note that edge $e$ maintains its original decorations).  That is, 

$$\begin{tabular}[c]{@{}c@{}}
%-- New mfpic environment, number 116 of 189. (size of end split: 2, should be 2)  ------------------->
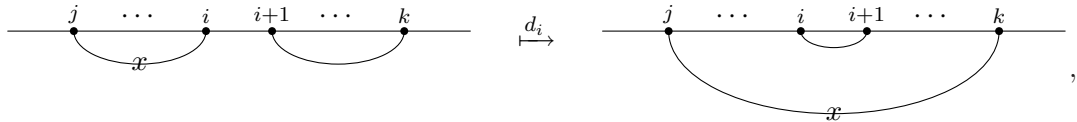
\includegraphics{ThesisFigs2.017}
\end{tabular},$$
where $x$ represents an arbitrary (possibly empty) block of decorations.
\hfill $\qed$
\end{lemma}

\begin{lemma}\label{nonprops.connect.adjacent.verts2}
Assume that in the north face of $d$ there is an edge, say $e$, connecting node $1$ to node $n$ labeled by a single $\bcirc$ (this can only happen if $n$ is even), and assume that there is a simple edge, say $e'$, connecting node $n+1$ to node $n+2$ (which must be labeled by a single $\wcirc$).  Then $d_{n}d$ is the admissible diagram that results from $d$ by joining the right end of $e$ to the left end of $e'$, and adding a simple edge that joins $n$ to $n+1$.  Note that the new edge formed by joining $e$ and $e'$ connects node $1$ to node $n+2$ and is labeled by the block $\bcirc \wcirc$.  That is, 
$$\begin{tabular}[c]{@{}c@{}}
%-- New mfpic environment, number 117 of 189. (size of end split: 2, should be 2)  ------------------->
\includegraphics{ThesisFigs2.018}
\end{tabular}.$$
\hfill $\qed$
\end{lemma}

\begin{lemma}\label{swap.prop.nonprop}
Assume that $d$ has a propagating edge, say $e$, joining node $i$ to node $j'$ with $1<i<n$.  Further, assume that there is a simple edge, say $e'$, joining nodes $i+1$ and $i+2$.  Then $d_{i}d$ is the admissible diagram that results from $d$ by removing $e'$, disconnecting $e$ from node $i$ and reattaching it to node $i+2$, and adding a simple edge to $i+1$ and $i+2$ (note that $e$ retains its original decorations).  This procedure has an inverse, since $i<n$ and $d_{i+1}d_{i}d=d$.  That is,
$$\begin{tabular}[c]{@{}c@{}}
%-- New mfpic environment, number 118 of 189. (size of end split: 2, should be 2)  ------------------->
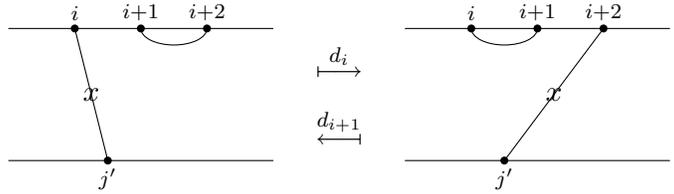
\includegraphics{ThesisFigs2.019}
\end{tabular},$$
where $x$ represents an arbitrary (possibly empty) block of decorations.
\hfill $\qed$
\end{lemma}

\begin{lemma}\label{make.decoration}
Assume that $d$ has simple edges joining node $1$ to node $2$ and node $3$ to node $4$.  Then $d_{1}d_{2}d$ is the admissible diagram that results from $d$ by adding a $\btri$ to the edge joining $3$ to $4$.  That is, 
$$\begin{tabular}[c]{@{}c@{}}
%-- New mfpic environment, number 119 of 189. (size of end split: 2, should be 2)  ------------------->
\includegraphics{ThesisFigs2.020}
\end{tabular}.$$
\hfill $\qed$
\end{lemma}

\begin{lemma}\label{move.decoration}
Assume that $d$ has two edges, say $e$ and $e'$, joining node $i$ to node $i+1$ and node $i+2$ to node $i+3$, respectively, where $1<i<n-1$ and $e'$ is simple.  Then $d_{i}d_{i+1}d$ is the admissible diagram that results from $d$ by removing the decorations from $e$ and adding them to $e'$.  This procedure has an inverse, since $d_{i+2}d_{i+1}(d_{i}d_{i+1}d)=d$.  That is, 
$$\begin{tabular}[c]{@{}c@{}}
%-- New mfpic environment, number 120 of 189. (size of end split: 2, should be 2)  ------------------->
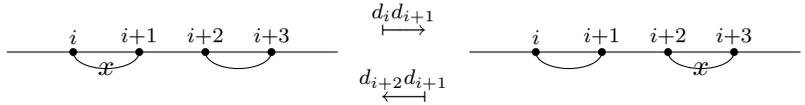
\includegraphics{ThesisFigs2.021}
\end{tabular},$$
where $x$ represents an arbitrary (possibly empty) block of decorations.  \hfill $\qed$
\end{lemma}

\begin{lemma}\label{move.decoration2}
Assume that $d$ has two edges, say $e$ and $e'$, joining node $i$ to node $i+1$ and node $i+2$ to node $i+3$, respectively, with $1<i<n-1$. Further, assume that $e$ is decorated by a single $\btri$ decoration only and that $e'$ is decorated by a single $\wtri$ decoration only.  Then $d_{i+2}d_{i+1}d$ is the admissible diagram that results from $d$ by removing the $\wtri$ decoration from $e'$ and adding it to $e$ to the right of the $\btri$ decoration.  That is,
$$\begin{tabular}[c]{@{}c@{}}
%-- New mfpic environment, number 121 of 189. (size of end split: 2, should be 2)  ------------------->
\includegraphics{ThesisFigs2.022}
\end{tabular}.$$  \hfill $\qed$
\end{lemma}

\begin{remark}
Each of Lemmas \ref{nonprops.connect.adjacent.verts}--\ref{move.decoration2} have left-right symmetric analogues (perhaps involving closed decorations), as well as versions that involve edges in the south face.
\end{remark}

\end{section}

\begin{section}{The admissible diagrams are generated by the simple diagrams}

Next, we state and prove several lemmas that we will use to prove that each admissible diagram can be written as a product of simple diagrams in $\D_{n}$.  (Note that in each of the lemmas of this section, all non-propagating edges are simple.)

\begin{lemma}\label{a-value1.diagrams.gen.by.simples}
If $d$ is an admissible diagram with $\a(d)=1$, then $d$ can be written as a product of simple diagrams.
\end{lemma}

\begin{proof}
Assume that $d$ is an admissible diagram with $\a(d)=1$.  The proof is an exhaustive case by case check, where we consider all the possible diagrams that are consistent with axiom \ref{C5}.  We consider five cases; any remaining cases follow by analogous arguments.  

\bigskip

Case (1): First, assume that
$$d=\begin{tabular}[c]{@{}c@{}}
%-- New mfpic environment, number 122 of 189. (size of end split: 2, should be 2)  ------------------->
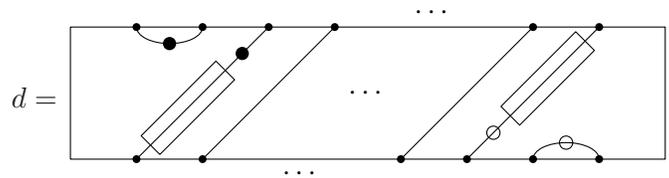
\includegraphics{ThesisFigs2.023}
\end{tabular},$$	
where the leftmost propagating edge carries $k$ $\btri$ decorations, and hence, the rightmost propagating edge carries $k$ $\wtri$ decorations by \ref{alternating.decorations.a=1} of Remark \ref{comment.admissible}.  In this case, it can quickly be verified that we can obtain $d$ via
	$$d=(d_{z_{1}}d_{z_{2}})^{k}d_{z_{1}}d_{n+1},$$
where $d_{z_{1}}=d_{1}d_{2}\cdots d_{n}$ and $d_{z_{2}}=d_{n+1}d_{n}\cdots d_{2}$.  Therefore, $d$ can be written as a product of simple diagrams, as desired.

\bigskip

Case (2): For the second case, assume that
$$d=\begin{tabular}[c]{@{}c@{}}
%-- New mfpic environment, number 123 of 189. (size of end split: 2, should be 2)  ------------------->
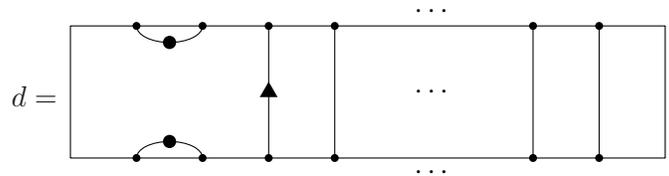
\includegraphics{ThesisFigs2.024}
\end{tabular}.$$
Note that $d$ does not carry any open decorations.  In this case, $d=d_{1}d_{2}d_{1}$, and so $d$ can be written as a product of simple diagrams, as expected.

\bigskip

Case (3): For the third case, assume that
$$d=\begin{tabular}[c]{@{}c@{}}
%-- New mfpic environment, number 124 of 189. (size of end split: 2, should be 2)  ------------------->
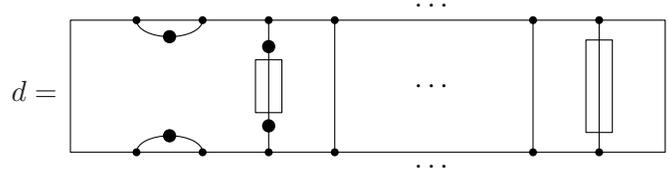
\includegraphics{ThesisFigs2.025}
\end{tabular},$$
where the leftmost propagating edge carries $k-1$ $\btri$ decorations, so that the rightmost propagating edge carries $k$ $\wtri$ decorations.  Then
	$$d=(d_{z_{1}}d_{z_{2}})^{k}d_{1},$$
where $d_{z_{1}}$ and $d_{z_{2}}$ are as in case (1), and hence $d$ can be written as a product of simple diagrams.

\bigskip

Case (4): Next, assume that
$$d=\begin{tabular}[c]{@{}c@{}}
%-- New mfpic environment, number 125 of 189. (size of end split: 2, should be 2)  ------------------->
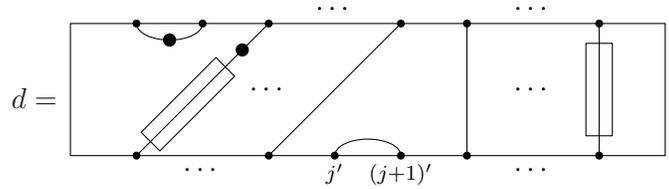
\includegraphics{ThesisFigs2.026}
\end{tabular},$$
where $1<j<n+1$ and the leftmost propagating edge carries $k$ $\btri$ decorations.  Then by Remark \ref{comment.admissible}\ref{alternating.decorations.a=1}, the rightmost propagating edge carries $l$ $\wtri$ decorations, where $l=k$ or $k+1$.  If $l=k+1$, then define
$$d'=\begin{tabular}[c]{@{}c@{}}
%-- New mfpic environment, number 126 of 189. (size of end split: 2, should be 2)  ------------------->
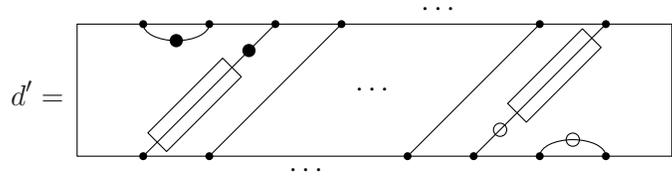
\includegraphics{ThesisFigs2.027}
\end{tabular},$$
where the leftmost (respectively, rightmost) propagating edge carries $k$ $\btri$ (respectively, $\wtri$) decorations.  By case (1), $d'$ can be written as a product of simple diagrams.  We see that
	$$d=d' d_{n}d_{n-1}\cdots d_{j+1}d_{j},$$
which implies that $d$ can be written as a product of simple diagrams, as desired.  If, on the other hand, $l=k$, then define $d'$ to be identical $d$ except that the last $\btri$ decoration occurring on the leftmost propagating edge has been removed.  Then by the subcase we just completed (where the rightmost propagating edge carried one more $\wtri$ decoration than the leftmost propagating edge carried $\btri$ decorations), $d'$ can be written as a product of simple diagrams.  We see that
	$$d=d' d_{j-1}d_{j-2}\cdots d_{2}d_{1}d_{2} \cdots d_{j-1}d_{j},$$
which implies that $d$ can be written as a product of simple diagrams.

\bigskip

Case (5):  For the final case, assume that
$$d=\begin{tabular}[c]{@{}c@{}}
%-- New mfpic environment, number 127 of 189. (size of end split: 2, should be 2)  ------------------->
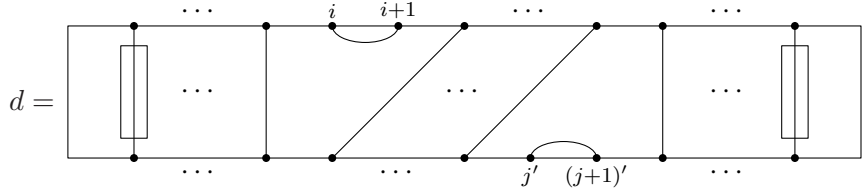
\includegraphics{ThesisFigs2.028}
\end{tabular},$$
where $1<i<n+1$, $1<j<n+1$, and the leftmost propagating edge carries $k$ $\btri$ decorations.  Then again by Remark \ref{comment.admissible}\ref{alternating.decorations.a=1}, the rightmost propagating edge carries $l$ $\wtri$ decorations, where $|k-l| \leq 1$.  Without loss of generality, assume that $k \leq l$, so that $l=k$ or $k+1$.  If $k=l$, then without loss of generality, assume that the first decoration occurring on the leftmost propagating edge has the highest relative vertical position of all decorations occurring on propagating edges.  Define the diagram
$$d'=\begin{tabular}[c]{@{}c@{}}
%-- New mfpic environment, number 128 of 189. (size of end split: 2, should be 2)  ------------------->
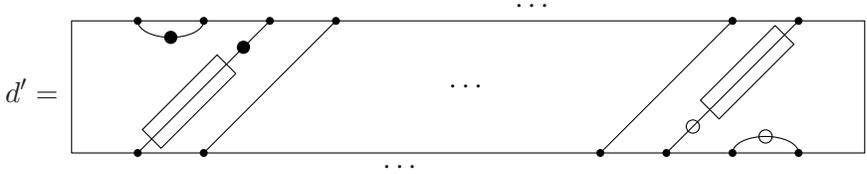
\includegraphics{ThesisFigs2.029}
\end{tabular},$$
where the leftmost (respectively, rightmost) propagating edge carries $k-1$ $\btri$ (respectively, $\wtri$) decorations.  By case (4), $d'$ can be written as a product of simple diagrams.  Also, we see that
	$$d=d_{i}d_{i-1}\cdots d_{3}d_{2} d' d_{n}d_{n-1} \cdots d_{j+1}d_{j},$$
which implies that $d$ can be written as a product of simple diagrams, as desired.  If, on the other hand, $l=k+1$, define
$$d'=\begin{tabular}[c]{@{}c@{}}
%-- New mfpic environment, number 129 of 189. (size of end split: 2, should be 2)  ------------------->
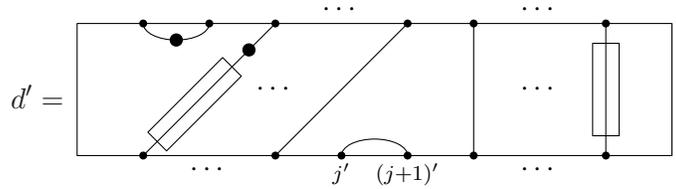
\includegraphics{ThesisFigs2.030}
\end{tabular},$$
where the leftmost propagating edge carries $k-1$ $\btri$ decorations while the rightmost propagating edge carries $k$ $\wtri$ decorations.  Again, by case (4), $d'$ can be written as a product of simple diagrams.  We see that
	$$d=d_{i}d_{i+1}\cdots d_{n}d_{n+1}d_{n}\cdots d_{3}d_{2}d',$$
which implies that $d$ can be written as a product of simple diagrams.
\end{proof}

\begin{lemma}\label{a-value.max.dammed.diagrams.gen.by.simples}
If $d$ is an admissible diagram with $1< \a(d)< \lfloor \frac{n+2}{2} \rfloor$ such that all non-propagating edges are simple, then $d$ can be written as a product of simple diagrams.
\end{lemma}

\begin{proof}
Let $d$ be an admissible diagram with $1< \a(d)< \lfloor \frac{n+2}{2} \rfloor$ such that all non-propagating edges are simple.  (Note that the restrictions on $\a(d)$ imply that $d$ has more than one propagating edge and has at least one non-propagating edge.)   We consider two cases, where the second case has two subcases.

\bigskip

Case (1): First, assume that $d$ has a vertical propagating edge, say $e_{i}$, joining $i$ to $i'$.  Now, define the admissible diagrams $d'$ and $d''$ via
	$$d'=\begin{tabular}[c]{@{}c@{}}
%-- New mfpic environment, number 130 of 189. (size of end split: 2, should be 2)  ------------------->
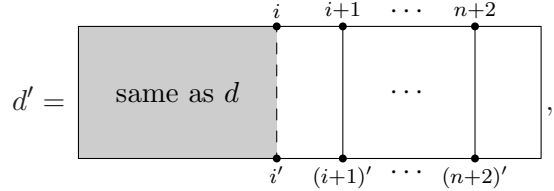
\includegraphics{ThesisFigs2.031}
\end{tabular},$$
and
$$d''=\begin{tabular}[c]{@{}c@{}}
%-- New mfpic environment, number 131 of 189. (size of end split: 2, should be 2)  ------------------->
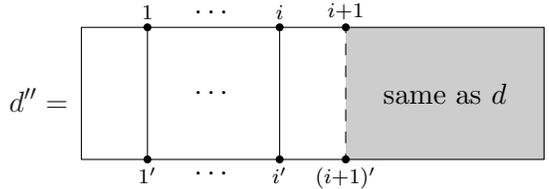
\includegraphics{ThesisFigs2.032}
\end{tabular},$$
where each of the shaded regions is identical to the corresponding regions of $d$.  Then $d=d'd''$.  Furthermore, $d'$ (respectively, $d''$) is L-exposed (respectively, R-exposed), and hence is only decorated with closed (respectively, open) decorations.  Since $d$ is admissible, $d' \in \DTL(B_{n})$ while $d'' \in \DTL(B'_{n})$.  This implies that both $d'$ and $d''$ can be written as a product of simple diagrams.  Therefore, $d$ can be written as a product of simple diagrams, as desired.

\bigskip

Case (2):  Next, assume that $d$ has no vertical propagating edges.  Suppose that the leftmost propagating edge joins node $i$ in the north face to node $j'$ in the south face, and without loss of generality, assume that $j<i$.  (Note that since $d$ has more than one propagating edge, $i<n+2$.)  We wish to make use of case (1), but we must consider two subcases.

\bigskip

(a): For the first subcase, assume that $j \neq 1$.  Since $d$ is admissible, we must have
$$d=\begin{tabular}[c]{@{}c@{}}
%-- New mfpic environment, number 132 of 189. (size of end split: 2, should be 2)  ------------------->
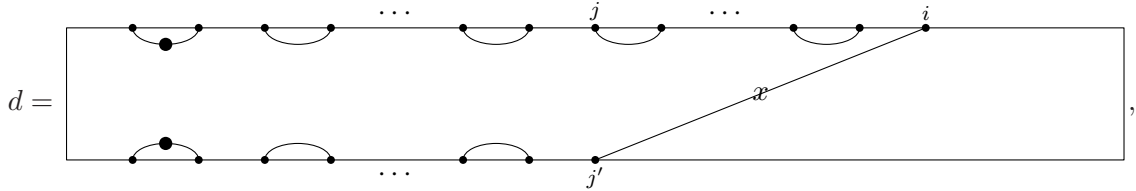
\includegraphics{ThesisFigs2.033}
\end{tabular},$$
where $x$ on the propagating edge from $i$ to $j'$ is either trivial (i.e., the edge is undecorated) or equal to a single $\btri$ decoration.  Define the admissible diagram
$$d'=\begin{tabular}[c]{@{}c@{}}
%-- New mfpic environment, number 133 of 189. (size of end split: 2, should be 2)  ------------------->
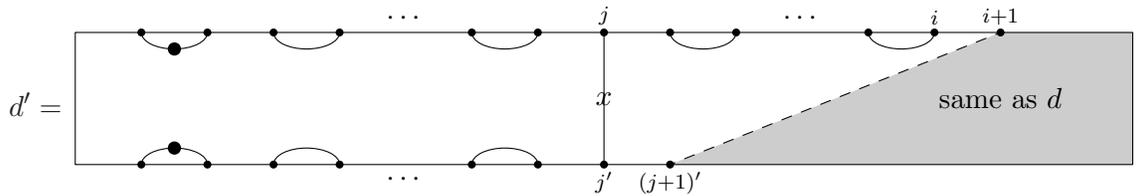
\includegraphics{ThesisFigs2.034}
\end{tabular},$$
where the leftmost propagating edge carries the same decoration as the leftmost propagating edge in $d$ and the shaded region is identical to the corresponding region of $d$.  By case (1), $d' $ can be written as a product of simple diagrams.  By making repeated applications of Lemma \ref{swap.prop.nonprop}, we can transform $d'$ into $d$, which shows that $d$ can be written as a product of simple diagrams, as desired.

\bigskip

(b): For the second subcase, assume that $j=1$, so that
$$d=\begin{tabular}[c]{@{}c@{}}
%-- New mfpic environment, number 134 of 189. (size of end split: 2, should be 2)  ------------------->
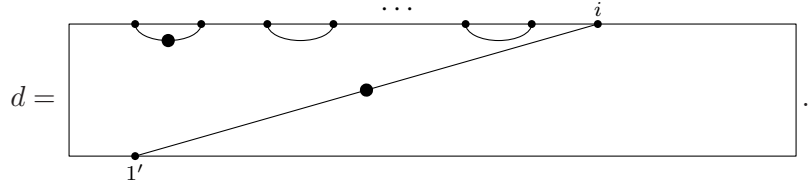
\includegraphics{ThesisFigs2.035}
\end{tabular}.$$
Since $1< \a(d)< \lfloor \frac{n+2}{2} \rfloor$, there is at least one other propagating edge occurring to the right of the leftmost propagating edge.  Furthermore, since the number of non-propagating edges in the north face is equal to the number of non-propagating edges in the south face, there is at least one undecorated non-propagating edge in the south face of $d$.  By making repeated applications, if necessary, of the southern version of Lemma \ref{swap.prop.nonprop}, we may assume that
$$d=\begin{tabular}[c]{@{}c@{}}
%-- New mfpic environment, number 135 of 189. (size of end split: 2, should be 2)  ------------------->
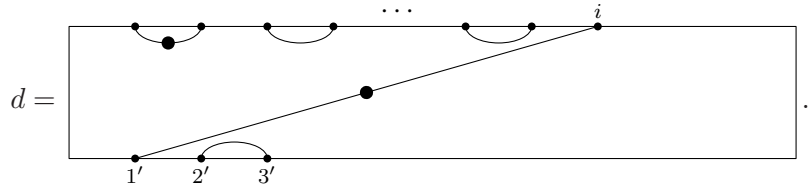
\includegraphics{ThesisFigs2.036}
\end{tabular}.$$
Now, define the admissible diagrams $d'$ and $d''$ via
$$d'=\begin{tabular}[c]{@{}c@{}}
%-- New mfpic environment, number 136 of 189. (size of end split: 2, should be 2)  ------------------->
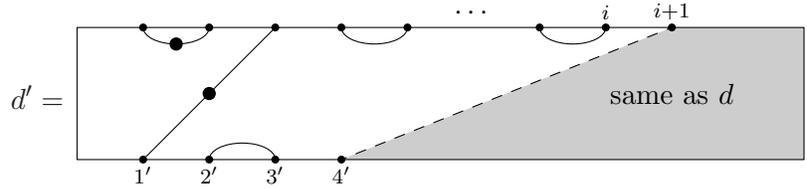
\includegraphics{ThesisFigs2.037}
\end{tabular}$$
and
$$d''=\begin{tabular}[c]{@{}c@{}}
%-- New mfpic environment, number 137 of 189. (size of end split: 2, should be 2)  ------------------->
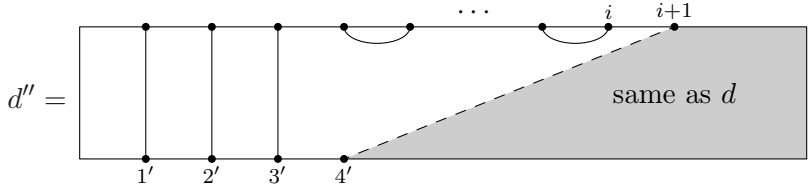
\includegraphics{ThesisFigs2.038}
\end{tabular}$$
where the shaded regions are identical to the corresponding regions of $d$.  By case (1), $d''$ can be written as a product of simple diagrams.  Also, we see that $d'=d_{1}d_{2}d''$, which implies that $d'$ can be written as a product of simple diagrams, as well.  By making repeated applications of Lemma \ref{swap.prop.nonprop}, we must have $d$ can be written as a product of simple diagrams.  
\end{proof}

\begin{lemma}\label{a-value.max.undammed.diagrams.gen.by.simples}
If $n$ is odd and $d$ is an admissible diagram with $\a(d)=\lfloor \frac{n+2}{2} \rfloor$ such that all non-propagating edges are simple, then $d$ can be written as a product of simple diagrams.
\end{lemma}

\begin{proof}
Assume that $n$ is odd and that $d$ is an admissible diagram with $\a(d)=\lfloor \frac{n+2}{2} \rfloor$.  In this case, $d$ has a unique propagating edge.  Also, assume that all of the non-propagating edges of $d$ are simple.  The proof is an exhaustive case by case check, where we consider the possible edges that are consistent with axiom \ref{C3} of Definition \ref{admissible.def}.  We consider five cases; any remaining cases follow by analogous arguments.  

\bigskip

Case (1):  For the first case, assume that
$$d=\begin{tabular}[c]{@{}c@{}}
%-- New mfpic environment, number 138 of 189. (size of end split: 2, should be 2)  ------------------->
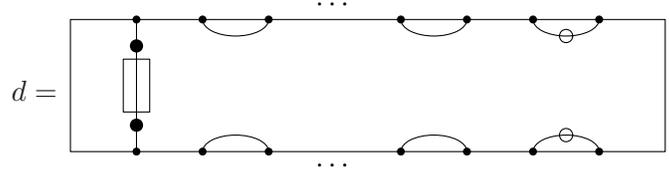
\includegraphics{ThesisFigs2.039}
\end{tabular},$$
where the rectangle on the propagating edge is equal to a block consisting of an alternating sequence of $k-1$ $\btri$ decorations and $k$ $\wtri$ decorations.  It is quickly verified that
	$$d=(d_{\E}d_{\O})^{k}d_{\E},$$
where
	$$d_{\E}=d_{2}d_{4}\cdots d_{n+1}$$
and
	$$d_{\O}=d_{1}d_{3}\cdots d_{n}.$$
This shows that $d$ can be written as a product of simple diagrams, as desired.

\bigskip

Case (2):  For the second case, assume that
$$d=\begin{tabular}[c]{@{}c@{}}
%-- New mfpic environment, number 139 of 189. (size of end split: 2, should be 2)  ------------------->
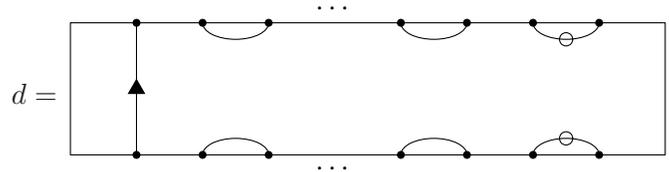
\includegraphics{ThesisFigs2.040}
\end{tabular}.$$
In this case, we see that
	$$d=d_{2}d_{1}d_{2}d_{4}\cdots d_{n-1}d_{n+1},$$
which shows that $d$ can be written as a product of simple diagrams.

\bigskip

Case (3):  Next, assume that
$$d=\begin{tabular}[c]{@{}c@{}}
%-- New mfpic environment, number 140 of 189. (size of end split: 2, should be 2)  ------------------->
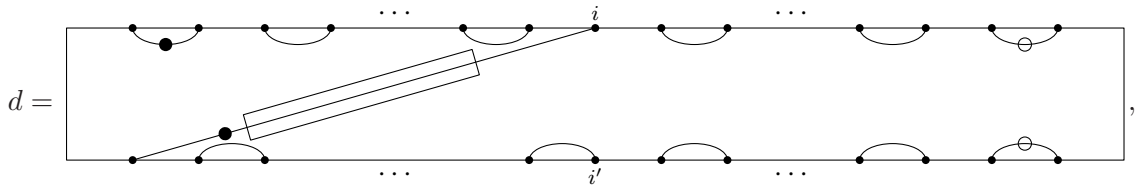
\includegraphics{ThesisFigs2.041}
\end{tabular},$$
where the rectangle on the propagating edge is either empty or equal to a block consisting of an alternating sequence of $k$ $\btri$ decorations and $l$ $\wtri$ decorations, where $l=k$ or $k+1$.  (Note that $i$ must be odd.)  If the rectangle is empty, then
	$$d=d_{1}d_{3}\cdots d_{i-2}d_{\E},$$
where $d_{\E}$ is as in case (1).  In this case, $d$ can be written as a product of simple diagrams.  On the other hand, if the rectangle is nonempty, so that the rectangle is equal to a block consisting of an alternating sequence of $k$ $\btri$ decorations and $l$ $\wtri$ decorations, where $l=k$ or $k+1$, define the admissible diagram 
$$d'=\begin{tabular}[c]{@{}c@{}}
%-- New mfpic environment, number 141 of 189. (size of end split: 2, should be 2)  ------------------->
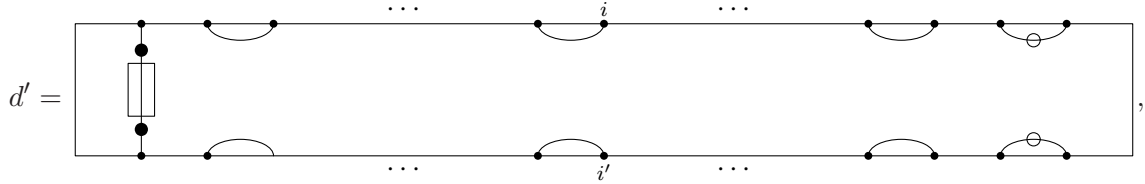
\includegraphics{ThesisFigs2.042}
\end{tabular},$$
where the rectangle on the propagating edge is equal to a block consisting of an alternating sequence of $k-1$ $\btri$ decorations and $k$ $\wtri$ decorations.  By case (1), $d'$ can be written as a product of simple diagrams.  If $k=l$, then we see that 
	$$d=d_{i-2}d_{i-4}\cdots d_{3}d_{1}d',$$
which implies that $d$ can be written as a product of simple diagrams, as desired.  If, on the other hand, $l=k+1$, then we see that
	$$d=d_{i+1}d_{i+3}\cdots d_{n-1}d_{n+1}d_{\O}d',$$
where $d_{\O}$ is as in case (1).  This shows that $d$ can be written as a product of simple diagrams.

\bigskip

Case (4):  Now, assume that
$$d=\begin{tabular}[c]{@{}c@{}}
%-- New mfpic environment, number 142 of 189. (size of end split: 2, should be 2)  ------------------->
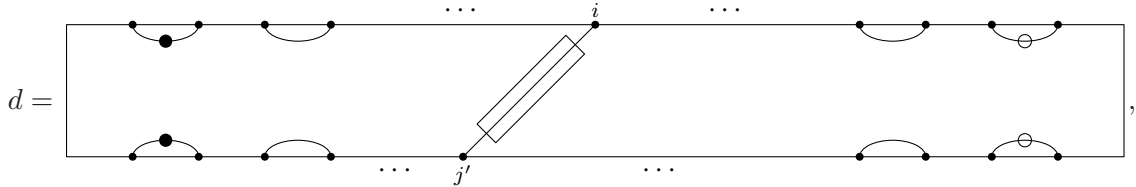
\includegraphics{ThesisFigs2.043}
\end{tabular},$$
where $i,j \notin \{1,n+2\}$ and the rectangle on the propagating edge is equal to a block consisting of an alternating sequence of $k$ $\btri$ decorations and $l$ $\wtri$ decorations with $|k-l| \leq 1$.  (Note that $i$ and $j$ must be odd.)  Without loss of generality, assume that $k \leq l$, so that $l=k$ or $k+1$.   Now, assume that the last decoration on the propagating edge is a $\btri$; the case with the last decoration being a $\wtri$ is handled with an analogous argument.  If $l=k$ (respectively, $l=k+1$), then the first decoration on the propagating edge is a $\wtri$ (respectively, $\btri$).  In either case, define the admissible diagram
$$d'=\begin{tabular}[c]{@{}c@{}}
%-- New mfpic environment, number 143 of 189. (size of end split: 2, should be 2)  ------------------->
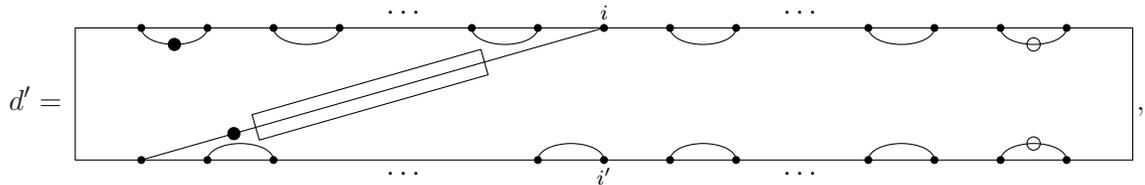
\includegraphics{ThesisFigs2.044}
\end{tabular},$$
where the rectangle on the propagating edge is equal to a block consisting of an alternating sequence of $k-1$ $\btri$ decorations and $l$ $\wtri$ decorations.  By case (3), $d'$ can be written as a product of simple diagrams.  Then it is quickly verified that
	$$d=d'd_{1}d_{3}\cdots d_{j-2}d_{j},$$
and so $d$ can be written as a product of simple diagrams.

\bigskip

Case (5):  For the final case, assume that
$$d=\begin{tabular}[c]{@{}c@{}}
%-- New mfpic environment, number 144 of 189. (size of end split: 2, should be 2)  ------------------->
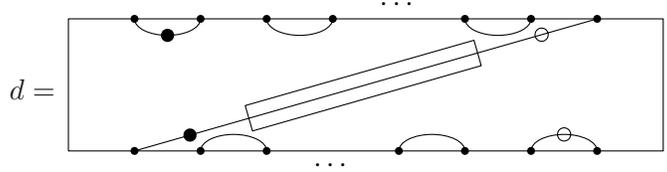
\includegraphics{ThesisFigs2.045}
\end{tabular},$$
where the rectangle on the propagating edge is equal to a block consisting of an alternating sequence of $k$ $\btri$ decorations and $k$ $\wtri$ decorations.  It is quickly seen that
	$$d=(d_{\O}d_{\E})^{k+1},$$
where $d_{\O}$ and $d_{\E}$ are as in case (1).  So, $d$ can be written as a product of simple diagrams, as expected.
\end{proof}

Now, we make use of the previous lemmas to prove the next proposition, which states that the admissible diagrams are generated by the simple diagrams.

\begin{proposition}\label{admissibles.in.algebra.gen.by.simples}
Each admissible diagram can be written as a product of simple diagrams.  In particular, the admissible diagrams are contained in $\D_{n}$.
\end{proposition}

\begin{proof}
Let $d$ be an admissible diagram.  We will show that $d$ can be written as a product of simple diagrams.  Lemma \ref{nonprops.connect.adjacent.verts}, and if necessary Lemma \ref{nonprops.connect.adjacent.verts2}, along with their analogues, allow us to assume that all of the non-propagating edges of $d$ join adjacent vertices.  Furthermore, Lemmas \ref{make.decoration}, \ref{move.decoration}, and \ref{move.decoration2}, along with their analogues, allow us to assume that all of the non-propagating edges of $d$ are simple.  We now consider four distinct cases: (1) $\a(d)=1$, (2) $1< \a(d)< \lfloor \frac{n+2}{2} \rfloor$, (3) $\a(d)=\lfloor \frac{n+2}{2} \rfloor$ with $n$ odd (i.e., $d$ has a unique propagating edge), and (4) $\a(d)=\frac{n+2}{2}$ with $n$ even (i.e., $d$ is undammed).

\bigskip

Cases (1), (2), and (3) follow immediately from Lemmas \ref{a-value1.diagrams.gen.by.simples}, \ref{a-value.max.dammed.diagrams.gen.by.simples}, and \ref{a-value.max.undammed.diagrams.gen.by.simples}, respectively.

\bigskip

Case (4): For the final case, assume that $\a(d)=\frac{n+2}{2}$ with $n$ even.  Then $d$ is undammed and based on our simplifying assumptions, we must have
$$d=\begin{tabular}[c]{@{}c@{}}
%-- New mfpic environment, number 145 of 189. (size of end split: 2, should be 2)  ------------------->
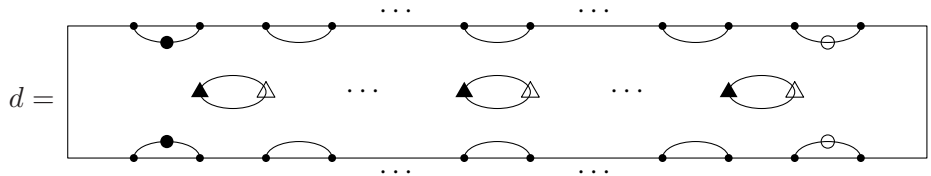
\includegraphics{ThesisFigs2.046}
\end{tabular},$$
where there are $k$ loop edges (we allow $k=0$).  Define the admissible diagram 
	$$d_{\O}=d_{1}d_{3} \cdots d_{n+1}.$$
Then
$$d_{\O}=\begin{tabular}[c]{@{}c@{}}
%-- New mfpic environment, number 146 of 189. (size of end split: 2, should be 2)  ------------------->
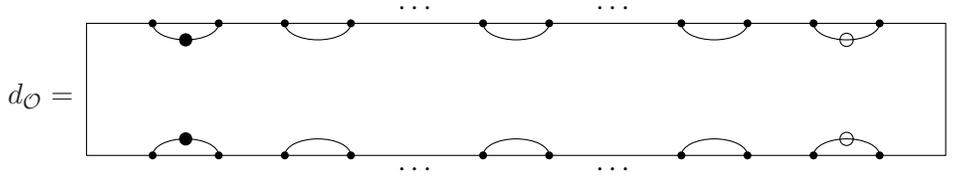
\includegraphics{ThesisFigs2.047}
\end{tabular}.$$
In particular, $d_{\O}$ is identical to $d$, except that is has no loop edges.  If $d$ has no loop edges (i.e., $k=0$), then we are done.  Suppose $k>0$.  By making the appropriate repeated applications of the left and right-handed versions of Lemmas \ref{make.decoration} and \ref{move.decoration} and a single application of Lemma \ref{move.decoration2}, there exists a sequence of simple diagrams $d_{i_{1}}, d_{i_{2}}, \dots, d_{i_{m}}$ such that
$$(d_{i_{1}}d_{i_{2}} \cdots d_{i_{m}})d_{\O}=\begin{tabular}[c]{@{}c@{}}
%-- New mfpic environment, number 147 of 189. (size of end split: 2, should be 2)  ------------------->
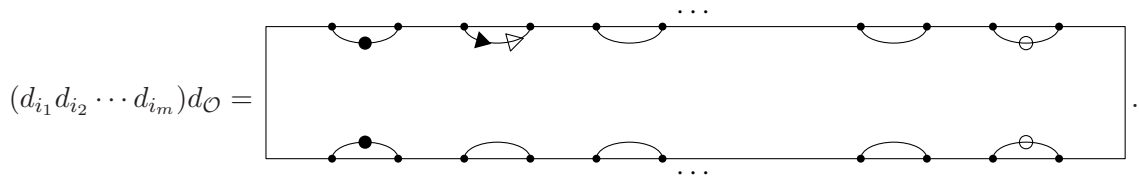
\includegraphics{ThesisFigs2.048}
\end{tabular}.$$
But then
$$d_{3}(d_{i_{1}}d_{i_{2}} \cdots d_{i_{m}})d_{\O}=\begin{tabular}[c]{@{}c@{}}
%-- New mfpic environment, number 148 of 189. (size of end split: 2, should be 2)  ------------------->
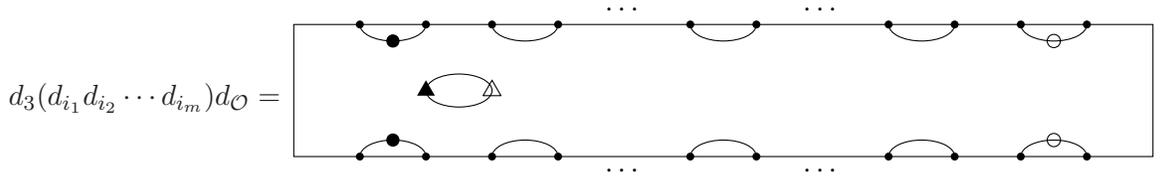
\includegraphics{ThesisFigs2.049}
\end{tabular}.$$
To produce $k$ loops, we repeat this process $k-1$ more times.  That is,
	$$d=\left(d_{3}(d_{i_{1}}d_{i_{2}}\cdots d_{i_{m}})\right)^{k}d_{\O}.$$
This shows that $d$ can be written as a product of simple diagrams, as desired.
\end{proof}

\end{section}

\begin{section}{More preparatory lemmas}

Our immediate goal is to show that the $\Z[\delta]$-module $\mathcal{M}[\Diag^{b}_{n}(\V)]$ is closed under multiplication, so that it is, in fact, a $\Z[\delta]$-algebra.  We shall prove a few lemmas that will aid in the process.

\begin{lemma}\label{first.closed.under.mult.lemma}
Let $d$ be an admissible diagram with the following edge configuration at nodes $i$ and $i+1$:
$$\begin{tabular}[c]{@{}c@{}}
%-- New mfpic environment, number 149 of 189. (size of end split: 2, should be 2)  ------------------->
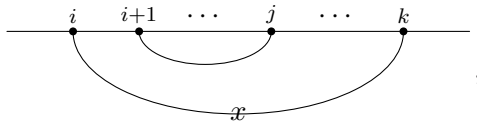
\includegraphics{ThesisFigs2.050}
\end{tabular},$$
where $x$ represents a (possibly trivial) block of decorations.  Then $d_{i}d=2^{c}d'$, where $c \in \{0,1\}$ and $d'$ is an admissible diagram.  Moreover, $c=1$ if and only if $i=1$.
\end{lemma}

\begin{proof}
The only case that requires serious consideration is if $i=1$; the result follows immediately if $i>1$.  Assume that $i=1$.  Since $d$ is admissible, $x \in \{\bcirc, \bcirc \wtri, \bcirc \wcirc\}$.  In any case, $d_{1}d=2 d'$ for some diagram $d'$, where the non-propagating edge joining node $j$ to node $k$ in $d'$ is one of the following blocks: $\btri, \btri \wtri$, or $\btri \wcirc$.  It follows that $d'$ is admissible.
\end{proof}

\begin{lemma}\label{second.closed.under.mult.lemma}
Let $d$ be an admissible diagram with the following edge configuration at nodes $i$ and $i+1$:
$$\begin{tabular}[c]{@{}c@{}}
%-- New mfpic environment, number 150 of 189. (size of end split: 2, should be 2)  ------------------->
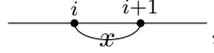
\includegraphics{ThesisFigs2.051}
\end{tabular},$$
where $x$ represents a (possibly trivial) block of decorations.  Then $d_{i}d=\delta^{c}d'$, where $c \in \{0,1\}$ and $d'$ is an admissible diagram.  Moreover, $c=0$ if and only if $x \in \{\bcirc \wtri, \btri \wtri, \btri \wcirc\}$.
\end{lemma}

\begin{proof}
We consider two cases.  For the first case, assume that $1<i<n+1$.  Since $d$ is admissible, $x \in \{1, \btri, \wtri, \btri \wtri\}$.  (Note that $x =\btri \wtri$ only if $d$ is undammed; otherwise $d$ would not be LR-decorated.)  In either case, $d_{i}d$ produces a loop decorated with the block $x$ along with a diagram that is identical to $d$, except that the block $x$ has been removed from the edge joining $i$ to $i+1$.  The loop decorated with the block $x$ is equal to $\delta$, unless $x=\btri \wtri$, in which case the loop is irreducible.  Regardless, the resulting diagram is admissible, as desired.  For the second case, assume that $i=1$ or $n$.  Without loss of generality, assume that $i=1$, the other case being symmetric.  Since $d$ is admissible, $x \in \{\bcirc, \bcirc \wtri\}$.  If $x=\bcirc$, then $d_{1}d=\delta d$, as expected.  If, on the other hand, $x=\bcirc \wtri$ (which can only happen if $d$ is undammed), then $d_{1}d$ results in an admissible diagram that is identical to $d$ except that we add a loop decorated by $\btri \wtri$ and remove the $\wtri$ decoration from the edge connecting node 1 to node 2.  
\end{proof}

\begin{lemma}\label{third.closed.under.mult.lemma}
Let $d$ be an admissible diagram with the following edge configuration at nodes $i$ and $i+1$:
$$\begin{tabular}[c]{@{}c@{}}
%-- New mfpic environment, number 151 of 189. (size of end split: 2, should be 2)  ------------------->
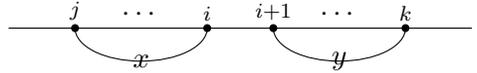
\includegraphics{ThesisFigs2.052}
\end{tabular},$$
where $x$ and $y$ represent (possibly trivial) blocks of decorations.  Then $d_{i}d=2^{c}d'$, where $c \in \{0,1\}$ and $d'$ is an admissible diagram.  
\end{lemma}

\begin{proof}
First, observe that $d_{i}d$ has the following edge configuration at nodes $i$ and $i+1$:
$$2^{c}\ \begin{tabular}[c]{@{}c@{}}
%-- New mfpic environment, number 152 of 189. (size of end split: 2, should be 2)  ------------------->
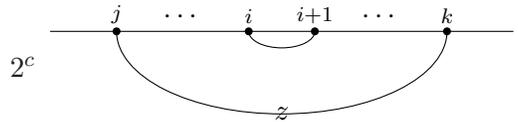
\includegraphics{ThesisFigs2.053}
\end{tabular},$$
where $xy=2^{c}z$ and $z$ is a basis element.  Note that since $d$ is admissible, there will be at most one relation to apply in the product $xy$, which will happen exactly when the last decoration in $x$ and the first decoration in $y$ are of the same type (open or closed).  This implies that $c \in \{0,1\}$.  If $j=1$ (respectively, $k=n+2$), then the first (respectively, last) decoration in $x$ (respectively, $y$) must be a $\bcirc$ (respectively, $\wcirc$) decoration.  Furthermore, if $j=1$ (respectively, $k=n+2$), then this is the only occurrence of a $\bcirc$ (respectively, $\wcirc$) decoration on a non-propagating edge in the north face of $d$.  By inspecting the possible relations we can apply, this implies that if $j=1$ (respectively, $k=n+2$), the first (respectively, last) decoration of $z$ must be a $\bcirc$ (respectively, $\wcirc$) decoration and this is the only occurrence of a $\bcirc$ (respectively, $\wcirc$) decoration on a non-propagating edge of the diagram that results from the product $d_{i}d$.  If, on the other hand, $j\neq 1$ and $k\neq n+2$, then neither of $x$ or $y$ may contain a $\bcirc$ or $\wcirc$ decoration.  In this case, $z$ will not contain any $\bcirc$ or $\wcirc$ decorations either.  This argument shows that the diagram that results from the product $d_{i}d$ must be admissible.
\end{proof}

\begin{lemma}\label{fourth.closed.under.mult.lemma}
Let $d$ be an admissible diagram such that $\a(d)>1$ with the following edge configuration at nodes $i$ and $i+1$:
$$\begin{tabular}[c]{@{}c@{}}
%-- New mfpic environment, number 153 of 189. (size of end split: 2, should be 2)  ------------------->
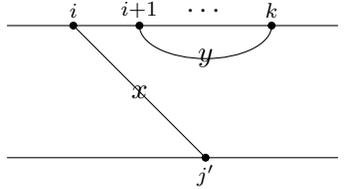
\includegraphics{ThesisFigs2.054}
\end{tabular},$$
where $x$ and $y$ represent (possibly trivial) blocks of decorations.  Then $d_{i}d=2^{c}d'$, where $c \in \{0,1\}$ and $d'$ is an admissible diagram.  
\end{lemma}

\begin{proof}
Note that $1\leq i <n+1$.  Since $d$ is dammed, $y$ is either equal to the identity in $\V$ or is equal to an open decoration.  On the other hand, $x$ could be equal to the identity in $\V$, a single closed decoration, a single open decoration, or if $d$ has a unique propagating edge, then $x$ could be an alternating sequence of open and closed decorations.  We consider two cases: (1) $1<i<n+1$ and (2) $i=1$.

\bigskip

Case (1):  For the first case, assume that $1<i<n+1$.  In this case, there will not be any relations to apply in the product of $d_{i}$ and $d$ unless the first decoration on the edge joining $i$ to $j'$ in $d$ is open and $y$ is also an open decoration, in which case $d_{i}d$ will be equal to 2 times an admissible diagram, as desired.  

\bigskip

Case (2):  Now, assume that $i=1$.  Since $d$ is admissible, either $x$ is trivial or the first decoration on the edge joining $1$ to $j'$ in $d$ must be closed.  If $x$ is trivial, then $j=1$, in which case, $d_{i}d$ is equal to a single admissible diagram.  If the first decoration is closed, then $d_{i}d$ equals 2 times an admissible diagram, as expected.
\end{proof}

\begin{lemma}\label{fifth.closed.under.mult.lemma}
Let $d$ be an admissible diagram such that $\a(d)=1$ with the following edge configuration at nodes $i$ and $i+1$:
$$\begin{tabular}[c]{@{}c@{}}
%-- New mfpic environment, number 154 of 189. (size of end split: 2, should be 2)  ------------------->
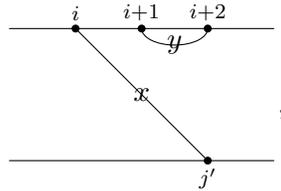
\includegraphics{ThesisFigs2.055}
\end{tabular},$$
where $x$ and $y$ represent (possibly trivial) blocks of decorations.  Then $d_{i}d=2^{c}d'$, where $c \in \{0,1\}$ and $d'$ is an admissible diagram with $\a(d')=1$.  
\end{lemma}

\begin{proof}
Since $\a(d)=1$, the non-propagating edge joining $i+1$ to $i+2$ is the unique non-propagating edge in the north face of $d$.  Also, note that $1\leq i \leq n$, where $i$ must be odd.  Furthermore, since $\a(d)=1$, the edge configuration at nodes $i$ and $i+1$ forces $j \in \{i, i+2\}$.  According to Lemma \ref{a-value=1}, the diagram that is produced by multiplying $d_{i}$ times $d$ has $\a$-value 1.  We consider three cases: (1) $i=1$, (2) $1< i<n$, and (3) $i=n$.

\bigskip

Case (1):  Assume that $i=1$.  This implies that $j \in \{1, 3\}$.  Then the possible edge configurations at nodes 1 and 2 of $d$ that are consistent with axiom \ref{C5} of Definition \ref{admissible.def} are as follows:

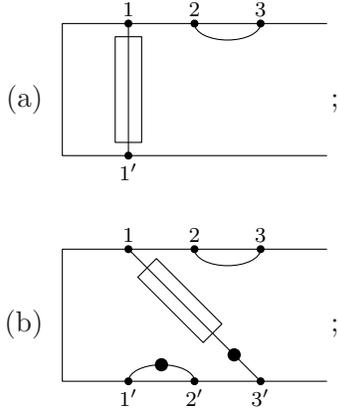
\begin{enumerate}[label=\rm{(\alph*)}]
\item \begin{tabular}[c]{@{}c@{}}%1
%-- New mfpic environment, number 155 of 189. (size of end split: 2, should be 2)  ------------------->
\includegraphics{ThesisFigs2.056}
\end{tabular};

\item \begin{tabular}[c]{@{}c@{}}%3
%-- New mfpic environment, number 156 of 189. (size of end split: 2, should be 2)  ------------------->
\includegraphics{ThesisFigs2.057}
\end{tabular};

\end{enumerate}

where the rectangle represents a (possibly trivial) sequence of blocks such that each block is a single $\btri$. In any case, we see that $d_{i}d=d_{1}d=2^{c}d'$, where $c \in \{0,1\}$ and $d'$ is an admissible diagram.

\bigskip

Case (2):  Next, assume that $1<i<n$.  In this case, since $\a(d)=1$, the restrictions on $i$ and $j'$ imply that both $x$ and $y$ are trivial.  That is, the propagating edge from $i$ to $j'$ and the non-propagating edge from $i+2$ to $i+3$ are undecorated.  Therefore, it is quickly seen that $d_{i}d=d'$ for some admissible diagram $d'$.

\bigskip

Case (3):  For the final case, assume that $i=n$.  This implies that $j \in \{n, n+2\}$.  Then the possible edge configurations at nodes $n$ and $n+1$ of $d$ that are consistent with axiom \ref{C5} of Definition \ref{admissible.def} are as follows:

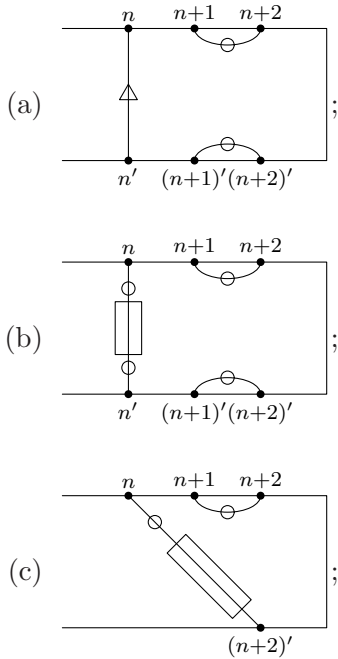
\begin{enumerate}[label=\rm{(\alph*)}]
\item \begin{tabular}[c]{@{}c@{}}
%-- New mfpic environment, number 157 of 189. (size of end split: 2, should be 2)  ------------------->
\includegraphics{ThesisFigs2.058}
\end{tabular};

\item \begin{tabular}[c]{@{}c@{}}
%-- New mfpic environment, number 158 of 189. (size of end split: 2, should be 2)  ------------------->
\includegraphics{ThesisFigs2.059}
\end{tabular};

\item \begin{tabular}[c]{@{}c@{}}
%-- New mfpic environment, number 159 of 189. (size of end split: 2, should be 2)  ------------------->
\includegraphics{ThesisFigs2.060}
\end{tabular};

\end{enumerate}
where the rectangle represents a nontrivial sequence of blocks such that each block is a single $\wtri$.  In any case, we see that $d_{i}d=d_{n}d=2^{c}d'$, where $c \in \{0,1\}$ and $d'$ is an admissible diagram.
\end{proof}

\begin{lemma}\label{sixth.closed.under.mult.lemma}
Let $d$ be an admissible diagram such that $\a(d)>1$ with the following edge configuration at nodes $i$ and $i+1$:
$$\begin{tabular}[c]{@{}c@{}}
%-- New mfpic environment, number 160 of 189. (size of end split: 2, should be 2)  ------------------->
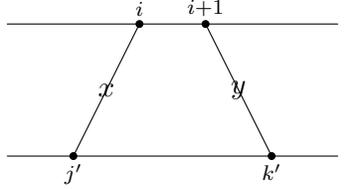
\includegraphics{ThesisFigs2.061}
\end{tabular},$$
where $x$ and $y$ represent (possibly trivial) blocks of decorations.  Then $d_{i}d=2^{c}d'$, where $c \in \{0,1\}$ and $d'$ is an admissible diagram.  
\end{lemma}

\begin{proof}
Since $d$ is LR-decorated, $x$ and $y$ cannot be of the same type (open or closed).  The only time there is potential to apply any relations when multiplying $d_{i}$ times $d$ is if $i=1$ (respectively, $i=n+1$) and $x$ (respectively, $y$) is nontrivial.  Regardless, it is easily seen that the statement of the lemma is true.
\end{proof}

\begin{lemma}\label{last.closed.under.mult.lemma}
Let $d$ be an admissible diagram such that $\a(d)=1$ with the following edge configuration at nodes $i$ and $i+1$:
$$\begin{tabular}[c]{@{}c@{}}
%-- New mfpic environment, number 161 of 189. (size of end split: 2, should be 2)  ------------------->
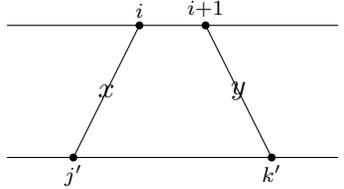
\includegraphics{ThesisFigs2.062}
\end{tabular},$$
where $x$ and $y$ represent (possibly trivial) blocks of decorations.  Then $d_{i}d=2^{k}d'$, where $k\geq 0$ and $d'$ is an admissible diagram with $\a(d)>1$. 
\end{lemma}

\begin{proof}
According to Lemma \ref{a-value=1}, the diagram that is produced by multiplying $d_{i}$ times $d$ has $\a$-value strictly greater than 1.  In this case, the sequence of blocks of decorations occurring on the leftmost (respectively, rightmost) propagating edge of $d$ will conjoin in the product of $d_{i}$ and $d$.  This implies that $d_{i}d=2^{k}d'$ for $k \geq 0$ and some diagram $d'$.  To see that $d'$ is admissible, we consider the following five possibilities for $d$; any remaining possibilities are analogous.

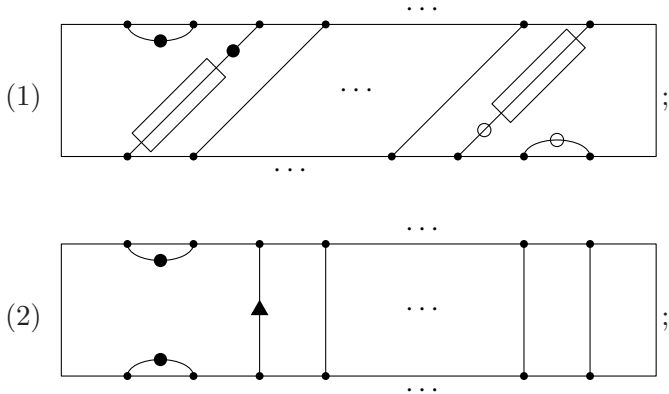
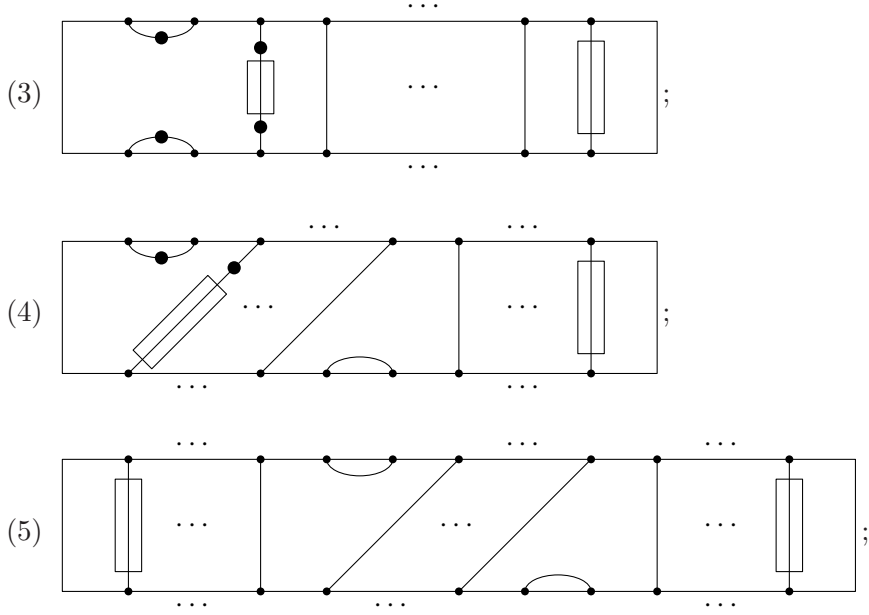
\begin{enumerate}[label=\rm{(\arabic*)}]
\item \begin{tabular}[c]{@{}c@{}}
%-- New mfpic environment, number 162 of 189. (size of end split: 2, should be 2)  ------------------->
\includegraphics{ThesisFigs2.063}
\end{tabular};

\item \begin{tabular}[c]{@{}c@{}}
%-- New mfpic environment, number 163 of 189. (size of end split: 2, should be 2)  ------------------->
\includegraphics{ThesisFigs2.064}
\end{tabular};

\item \begin{tabular}[c]{@{}c@{}}
%-- New mfpic environment, number 164 of 189. (size of end split: 2, should be 2)  ------------------->
\includegraphics{ThesisFigs2.065}
\end{tabular};

\item \begin{tabular}[c]{@{}c@{}}
%-- New mfpic environment, number 165 of 189. (size of end split: 2, should be 2)  ------------------->
\includegraphics{ThesisFigs2.066}
\end{tabular};

\item \begin{tabular}[c]{@{}c@{}}
%-- New mfpic environment, number 166 of 189. (size of end split: 2, should be 2)  ------------------->
\includegraphics{ThesisFigs2.067}
\end{tabular};

\end{enumerate}
where the rectangle on the leftmost (respectively, rightmost) propagating edge represents a (possibly trivial) sequence of blocks such that each block is a single $\btri$ (respectively, $\wtri$).  In each of these cases, if $d$ has propagating edges joined to nodes $i$ and $i+1$ in the north face,  it is quickly seen that the diagram $d'$ that results from multiplying $d_{i}$ times $d$ will be consistent with the axioms of Definition \ref{admissible.def} since $\bcirc \btri \cdots \btri \bcirc$ and $\btri \cdots \btri$ (respectively, $\wcirc \wtri \cdots \wtri \wcirc$ and $\wtri \cdots \wtri$) are equal to a power of 2 times $\btri$ (respectively, $\wtri$).
\end{proof}

\end{section}

\begin{section}{The set of admissible diagrams form a basis for $\D_{n}$}

The next lemma states that the product of a simple diagram and an admissible diagram results in a multiple of an admissible diagram.  The proof relies on stringing together Lemmas \ref{first.closed.under.mult.lemma}--\ref{last.closed.under.mult.lemma}.

\begin{lemma}\label{powers_of_2_and_delta}
Let $d$ be an admissible diagram.  Then 
	$$d_{i}d=2^{k}\delta^{m}d'$$ 
for some $k,m \in \Z^{+}\cup\{0\}$ and admissible diagram $d'$.
\end{lemma}

\begin{proof}
Let $d$ be an admissible diagram and consider the product $d_{i}d$.  Observe that the possible edge configurations for $d$ at nodes $i$ and $i+1$ are as follows:
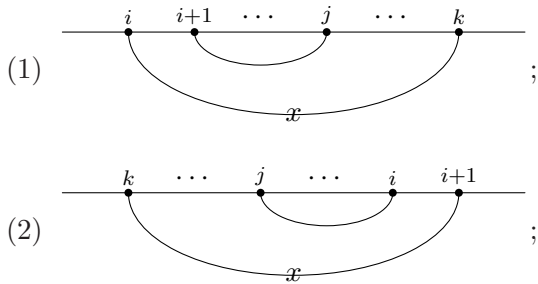
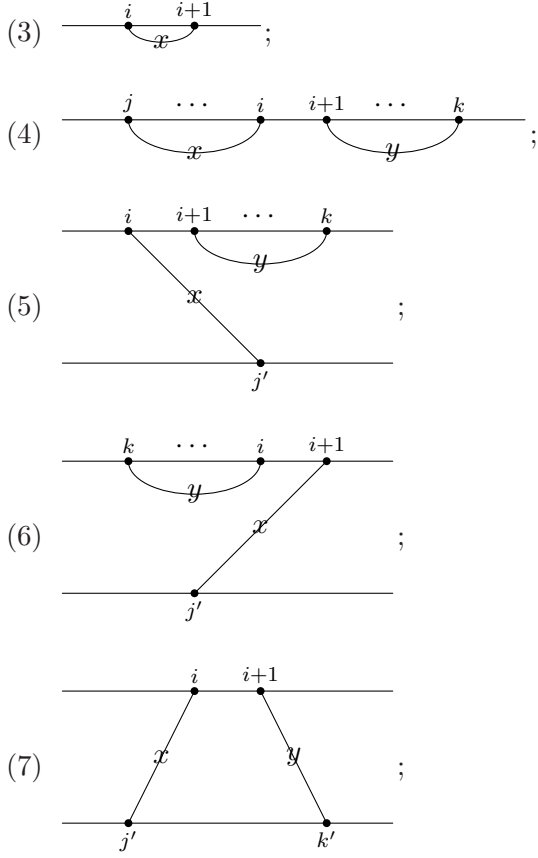
\begin{enumerate}[label=\rm{(\arabic*)}]
\item \begin{tabular}[c]{@{}c@{}}
%-- New mfpic environment, number 167 of 189. (size of end split: 2, should be 2)  ------------------->
\includegraphics{ThesisFigs2.068}
\end{tabular};

\item \begin{tabular}[c]{@{}c@{}}
%-- New mfpic environment, number 168 of 189. (size of end split: 2, should be 2)  ------------------->
\includegraphics{ThesisFigs2.069}
\end{tabular};

\item \begin{tabular}[c]{@{}c@{}}
%-- New mfpic environment, number 169 of 189. (size of end split: 2, should be 2)  ------------------->
\includegraphics{ThesisFigs2.070}
\end{tabular};

\item \begin{tabular}[c]{@{}c@{}}
%-- New mfpic environment, number 170 of 189. (size of end split: 2, should be 2)  ------------------->
\includegraphics{ThesisFigs2.071}
\end{tabular};

\item \begin{tabular}[c]{@{}c@{}}
%-- New mfpic environment, number 171 of 189. (size of end split: 2, should be 2)  ------------------->
\includegraphics{ThesisFigs2.072}
\end{tabular};

\item \begin{tabular}[c]{@{}c@{}}
%-- New mfpic environment, number 172 of 189. (size of end split: 2, should be 2)  ------------------->
\includegraphics{ThesisFigs2.073}
\end{tabular};

\item \begin{tabular}[c]{@{}c@{}}
%-- New mfpic environment, number 173 of 189. (size of end split: 2, should be 2)  ------------------->
\includegraphics{ThesisFigs2.074}
\end{tabular};

\end{enumerate}
Case (1) follows from Lemma \ref{first.closed.under.mult.lemma} and case (2) follows from a symmetric argument.  Case (3) follows from Lemma \ref{second.closed.under.mult.lemma}.  Lemma \ref{third.closed.under.mult.lemma} yields case (4).  Case (5) follows from Lemmas \ref{fourth.closed.under.mult.lemma} and \ref{fifth.closed.under.mult.lemma} and case (6) follows by a symmetric argument.  Finally, Lemmas \ref{sixth.closed.under.mult.lemma} and \ref{last.closed.under.mult.lemma} prove case (7).
\end{proof}

\begin{proposition}\label{module.is.subalgebra}
The $\Z[\delta]$-module $\mathcal{M}[\Diag^{b}_{n}(\V)]$ is a $\Z[\delta]$-subalgebra of $\widehat{\P}_{n+2}^{LR}(\V)$.
\end{proposition}

\begin{proof}
It remains to show that the product of any two admissible diagrams results in a linear combination of admissible diagrams.  According to Proposition \ref{admissibles.in.algebra.gen.by.simples}, each admissible diagram can be written as product of simple diagrams.  It suffices to show that the product of a simple diagram and an admissible diagram is equal to a linear combination of admissible diagrams.  But this is exactly Lemma \ref{powers_of_2_and_delta}, and so we have our desired result.
\end{proof}

We now state the main result of this chapter.

\begin{theorem}\label{module.of.admissibles.is.the.algebra.gen.by.simples}
The $\Z[\delta]$-algebras $\mathcal{M}[\Diag^{b}_{n}(\V)]$ and $\D_{n}$ are equal.  Moreover, the set of admissible diagrams is a basis for $\D_{n}$.
\end{theorem}

\begin{proof}
Propositions \ref{admissibles.in.algebra.gen.by.simples} and \ref{module.is.subalgebra} imply that $\mathcal{M}[\Diag^{b}_{n}(\V)]$ is a subalgebra of $\D_{n}$.  However, $\D_{n}$ is the smallest algebra containing the simple diagrams, which $\mathcal{M}[\Diag^{b}_{n}(\V)]$ also contains since the simple diagrams are admissible.  Therefore, we must have equality of the two algebras.  By Proposition \ref{prop.admissible.basis}, the set of admissible diagrams is a basis for $\mathcal{M}[\Diag^{b}_{n}(\V)]$.  Therefore, the set of admissible diagrams is a basis for $\D_{n}$.
\end{proof}

\end{section}

\end{chapter}

%%%%%%%%%%%% Chapter 10 %%%%%%%%%%%

\begin{chapter}{Main results}

We conclude this chapter with a proof that $\TL(\C_{n})$ and $\D_{n}$ are isomorphic as $\Z[\delta]$-algebras under the correspondence induced by
	$$b_{i} \mapsto d_{i}.$$
Moreover, we show that the admissible diagrams correspond to the monomial basis of $\TL(\C_{n})$.  

\begin{section}{The homomorphism $\theta$ from $\TL(\C_{n})$ to $\D_{n}$}

\begin{proposition}\label{surjective.homomorphism}
Let $\theta: \TL(\C_{n}) \to \D_{n}$ be the function determined by
	$$\theta(b_{i})=d_{i}.$$
Then $\theta$ is a surjective algebra homomorphism.
\end{proposition}

\begin{proof}
Checking that each of the following relations holds for the simple diagrams is easily verified.
\begin{enumerate}[label=\rm{(\arabic*)}]
\item $d_{i}^{2}=\delta d_{i}$ for all $i$,
\item $d_{i}d_{j}=d_{i}d_{j}$ if $|i-j|>1$,
\item $d_{i}d_{j}d_{i}=d_{i}$ if $|i-j|=1$ and $1< i,j < n+1$,
\item $d_{i}d_{j}d_{i}d_{j}=2d_{i}d_{j}$ if $\{i,j\}=\{1,2\}$ or $\{n,n+1\}$.
\end{enumerate}
That is, $\D_{n}$ satisfies the relations of $\TL(\C_{n})$.  This implies that $\theta$ is an algebra homomorphism.  Since the simple diagrams $\{d_{i}: 1\leq i \leq n+1\}$ generate $\D_{n}$, $\theta$ is surjective.
\end{proof}

\begin{lemma}\label{powers.of.2.and.delta.for.images.of.monomials}
Let $w \in W_{c}(\C_{n})$ have reduced expression $s_{i_{1}}\cdots s_{i_{r}}$.  Then
	$$\theta(b_{w})=2^{k}\delta^{m}d,$$
for some $k,m \in \Z^{+}\cup\{0\}$ and admissible diagram $d$.
\end{lemma}

\begin{proof}
By repeated applications of  Lemma \ref{powers_of_2_and_delta}, we have
\begin{align*}
\theta(b_{w}) & = \theta(b_{i_{1}} \cdots b_{i_{r}})  \\
& = \theta(b_{i_{1}}) \cdots \theta(b_{i_{r}})\\
& = d_{i_{1}} \cdots d_{i_{r}} \\
& = 2^{k}\delta^{m}d
\end{align*}
for some $k,m \in \Z^{+}\cup\{0\}$ and admissible diagram $d$.  
\end{proof}

\begin{remark}
Since $\theta$ is well-defined, $k$, $m$, and $d$ do not depend on the choice of reduced expression for $w$ that we start with.  We will denote the diagram $d$ from Lemma \ref{powers.of.2.and.delta.for.images.of.monomials} by $d_{w}$.  That is, if $w \in W_{c}$, then $d_{w}$ is the admissible diagram satisfying
	$$\theta(b_{w})=2^{k}\delta^{m}d_{w}.$$
\end{remark}

\begin{remark}
Let $d$ be an admissible diagram.  Since $\theta$ is surjective, there exists $w \in W_{c}(\C_{n})$ such that $\theta(b_{w})=d_{w}=d$.  Suppose that $w$ has reduced expression $\w=s_{i_{1}}\cdots s_{i_{r}}$.  Then $d=d_{i_{1}}\cdots d_{i_{r}}$.   For each $d_{i_{j}}$ fix a concrete representative that has straight propagating edges and no unnecessary ``wiggling'' of the simple non-propagating edges.  Now, consider the concrete diagram that results from stacking the concrete simple diagrams $d_{i_{1}},\dots, d_{i_{r}}$, rescaling to recover the standard $(n+2)$-box, but not deforming any of the edges or applying any relations among the decorations.  We will refer to this concrete diagram as the \emph{concrete simple representation of $d_{\w}$} (which depends on $\w$).  Since $w$ is fully commutative and vertical equivalence respects commutation, given two different reduced expressions $\w$ and $\w'$ for $w$, the concrete simple representations $d_{\w}$ and $d_{\w'}$ will be vertically equivalent.  We define the vertical equivalence class of concrete simple representations to be  the \emph{simple representation of $d_{w}$}.
\end{remark}

\begin{example}
Consider $w=\z_{1,1}^{R,1}$ in $W(\C_{3})$.  Then
$$\begin{tabular}[c]{@{}c@{}}
%-- New mfpic environment, number 174 of 189. (size of end split: 2, should be 2)  ------------------->
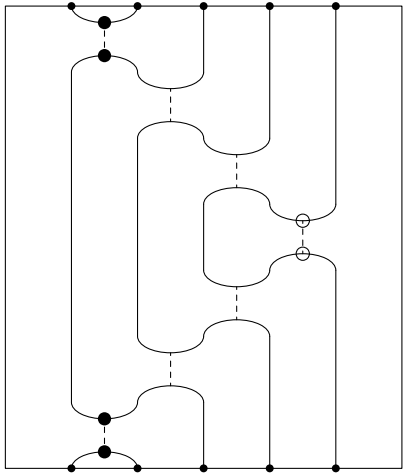
\includegraphics{ThesisFigs2.075}
\end{tabular}$$	
is vertically equivalent to the simple representation of $d_{w}$, where the vertical dashed lines in the figure indicate that the two curves are part of the same generator.
\end{example}

\begin{lemma}\label{image_zigzag}
Let $w \in W_{c}(\C_{n})$ be of type I.  Then $\theta(b_{w})=d_{w}$, where $\a(d_{w})=n(w)=1$.  If $w$ and $w'$ are both of type I with $w \neq w'$, then $d_{w} \neq d_{w'}$.  Moreover, if $s_{i} \in \L(w)$ (respectively, $\R(w)$), then there is a simple edge joining $i$ to $i+1$ (respectively, $i'$ to $(i+1)'$).  
\end{lemma}

\begin{proof}
This lemma follows easily from the definition of $\theta$.
\end{proof}

\begin{lemma}\label{image_wsrm}
Let $w\in W_{c}(\C_{n})$ be a non-type I irreducible element.  Then $\theta(b_{w})=d_{w}$, where $\a(d_{w})=n(w)>1$.  If $w$ and $w'$ are both non-type I irreducible elements with $w \neq w'$, then $d_{w} \neq d_{w'}$.  Moreover, if $s_{i} \in \L(w)$ (respectively, $\R(w)$), then there is a simple edge joining $i$ to $i+1$ (respectively, $i'$ to $(i+1)'$). 
\end{lemma}

\begin{proof}
This lemma follows from definition of $\theta$ and the classification of the irreducible elements in Theorem \ref{affineCwsrm}.
\end{proof}

\begin{remark}\label{image.of.irred.is.single.diagram}
The upshot of the previous two lemmas is that the image of a monomial indexed by any type I element or any irreducible element is a single admissible diagram (i.e., there are no powers of 2 or $\delta$). 
\end{remark}

Our immediate goal is to show that $\theta(b_{w})=d_{w}$ for any $w \in W_{c}(\C_{n})$.

\end{section}

\begin{section}{Preparatory lemmas}

The next lemma is \cite[Lemma 2.9]{Shi.J:C}.

\begin{lemma}\label{star.ops.preserve.n-value}
Let $X$ be a Coxeter graph and let $w \in W_{c}(X)$.  If $\bigstar^{L}_{s,t}(w)$ (respectively, $\bigstar^{R}_{s,t}(w)$) is defined, then $n(w)=n\left(\bigstar^{L}_{s,t}(w)\right)$ (respectively, $n(w)=n\left(\bigstar^{R}_{s,t}(w)\right)$). \hfill $\Box$
\end{lemma}

\begin{corollary}\label{weak.star.ops.preserve.n-value}
Let $w \in W_{c}(\C_{n})$.  If $\star^{L}_{s,t}(w)$ (respectively, $\star^{R}_{s,t}(w)$) is defined, then $n(w)=n\left(\star^{L}_{s,t}(w)\right)$ (respectively, $n(w)=n\left(\star^{R}_{s,t}(w)\right)$).
\end{corollary}

\begin{proof}
This follows from Lemma \ref{star.ops.preserve.n-value} since weak star reductions (when defined) are a special case of ordinary star reductions.
\end{proof}

We will state an ``if and only if'' version of the next lemma later (see Lemma \ref{diagram.descent.set}); the ``only if'' direction requires results that will depend on results that we have not yet proved.  We do, however, need the ``if'' direction to prove Lemma \ref{weak.star.preserve.a.value} which follows.

\begin{lemma}\label{simple.edge.in.N.face}
Let $w \in W_{c}(\C_{n})$.  If $s_{i} \in \L(w)$ (respectively, $\R(w)$), then there is a simple edge joining node $i$ to node $i+1$ (respectively, node $i'$ to node $(i+1)'$) in the north (respectively, south) face of $d_{w}$.
\end{lemma} 

\begin{proof}
Assume that $s_{i} \in \L(w)$.  Then we can write $w=s_{i}v$ (reduced).  By Lemma \ref{powers.of.2.and.delta.for.images.of.monomials} applied to $\theta(b_{v})$, we must have
\begin{align*}
\theta(b_{w})&=\theta(b_{i}b_{v})\\
&=\theta(b_{i})\theta(b_{v})\\
&=d_{i} 2^{k}\delta^{m}d_{v}\\
&=2^{k}\delta^{m}d_{i}d_{v}.
\end{align*}
This implies that we obtain $d_{w}$ by concatenating $d_{i}$ on top of $d_{v}$.  In this case, there are no relations to interact with the simple edge from $i$ to $i+1$ in the north face of $d_{i}$.  So, we must have a simple edge joining $i$ to $i+1$ in the north face of $d_{w}$.  

\bigskip

The proof that $s_{i} \in \R(w)$ implies that there is a simple edge joining $i'$ to $(i+1)'$ is symmetric.
\end{proof}

\begin{lemma}\label{weak.star.preserve.a.value}
Suppose $w \in W_{c}(\C_{n})$ is left weak star reducible by $s$ with respect to $t$ to $v$.  Then $\a(d_{w})=\a(d_{v})$.
\end{lemma}

\begin{proof}
Since $w$ is left weak star reducible by $s$ with respect to $t$ to $v$, by Remark \ref{monomial.weak.star.reductions}, we have $b_{t}b_{w}=2^{c}b_{v}$, where $c \in \{0,1\}$ and $c=1$ if and only if $m(s,t)=4$.  This implies that
	$$\theta(b_{t}b_{w})=\theta(2^{c}b_{v})=2^{c}2^{k}\delta^{m}b_{v},$$
where $k,m \in \Z^{+}\cup\{0\}$.  But, on the other hand, we have
\begin{align*}
\theta(b_{t}b_{w})&=\theta(b_{t})\theta(b_{w})\\
&=d_{t}2^{k'}\delta^{m'}d_{w}\\
&=2^{k'}\delta^{m'}d_{t}d_{w},
\end{align*}
where $k',m' \in \Z^{+}\cup\{0\}$.  Therefore, we have
	$$2^{c}2^{k}\delta^{m}d_{v}=2^{k'}\delta^{m'}d_{t}d_{w}.$$
This implies that when we multiply $d_{t}$ times $d_{w}$, we obtain a scalar multiple of $d_{v}$.  Also, since $w$ is left weak star reducible by $s$ with respect to $t$, we must have $s \in \L(w)$.  Suppose that $s=s_{i}$.  By Lemma \ref{simple.edge.in.N.face}, there must be a simple edge joining $i$ to $i+1$ in north face of $d_{w}$. Without loss of generality, assume that $t=s_{i+1}$; the case with $t=s_{i-1}$ follows symmetrically.  Then the edge configuration at nodes $i$, $i+1$, and $i+2$ of $d_{w}$ must be as follows:
$$\begin{tabular}[c]{@{}c@{}}
%-- New mfpic environment, number 175 of 189. (size of end split: 2, should be 2)  ------------------->
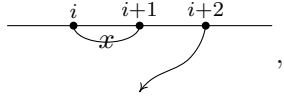
\includegraphics{ThesisFigs2.076}
\end{tabular},$$
where $x=\bcirc$ if and only if $i=1$ and is trivial otherwise.  Also, the edge leaving node $i+2$ may be decorated and may be propagating or non-propagating.  Then multiplying $d_{w}$ on the left by $d_{t}=d_{i+1}$, we see that the resulting diagram has the same $\a$-value as $d_{w}$.  Therefore, $\a(d_{v})=\a(d_{w})$, as desired.
\end{proof}

\begin{remark}
Lemma \ref{weak.star.preserve.a.value} has an analogous statement involving right weak star reductions.
\end{remark}

\begin{lemma}\label{a=1.implies.typeI}
Let $w \in W_{c}(\C_{n})$.  Then $\a(d_{w})=1$ if and only if $w$ is of type I.
\end{lemma}

\begin{proof}
First, assume that $\a(d_{w})=1$.  If $w$ is irreducible, then by Lemmas \ref{image_zigzag} and \ref{image_wsrm}, $w$ must be of type I.  Assume that $w$ is not irreducible.  Then there exists a sequence of weak star reductions that reduce $w$ to an irreducible element.  By Lemma \ref{weak.star.preserve.a.value}, each diagram corresponding to the elements of this sequence have the same $\a$-value as $d_{w}$, namely $\a$-value 1.  Since the sequence of weak star operations terminates at an irreducible element and the diagram corresponding to this element has $\a$-value 1, the irreducible element must have $n$-value 1, as well (again, by Lemmas \ref{image_zigzag} and \ref{image_wsrm}).  Since weak star operations preserve the $n$-value (Corollary \ref{weak.star.ops.preserve.n-value}), it must be the case that $n(w)$ was equal to 1 to begin with.  Therefore, $w$ is of type I.  Conversely, if $w$ is of type I, then according to Lemma \ref{image_zigzag}, $\a(d_{w})=1$.
\end{proof}

\begin{lemma}\label{diagram.version.main.zigzag.lemma}
Let $w \in W_{c}(\C_{n})$.  Suppose that there exists $i$ with $1<i < n+1$ such that $s_{i+1}$ does not occur between two consecutive occurrences of $s_{i}$ in $w$.    Then one or both of the following must be true about $d_{w}$:
\begin{enumerate}[label=\rm{(\roman*)}]
\item the western end of the simple representation of $d_{w}$ is vertically equivalent to
$$\begin{tabular}[c]{@{}c@{}}
%-- New mfpic environment, number 176 of 189. (size of end split: 2, should be 2)  ------------------->
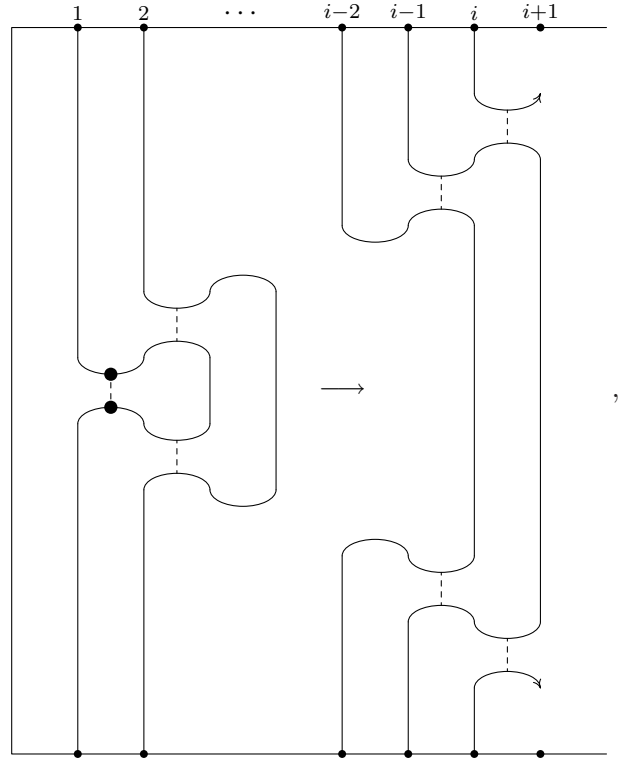
\includegraphics{ThesisFigs2.077}
\end{tabular},$$
where the vertical dashed lines in the figure indicate that the two curves are part of the same generator $d_{j}$ and the free horizontal arrow indicates a continuation of the pattern of the same shape.  Furthermore, there are no other occurrences of the generators $d_{1}, \dots, d_{i}$ in the simple representation of $d_{w}$.

\item $\a(d_{w})=1$; 

\end{enumerate}
\end{lemma}

\begin{proof}
This follows immediately from Lemma \ref{main_zigzag_lemma} by applying $\theta$ to the monomial indexed by the type I element $\z_{i,i}^{L,1}$.
\end{proof}

\begin{lemma}\label{direction.change}
Let $w \in W_{c}(\C_{n})$.  Then the only way that an edge of the simple representation of $d_{w}$ may change direction from right to left is if a convex subset of the simple representation of $d_{w}$ is vertically equivalent to one of
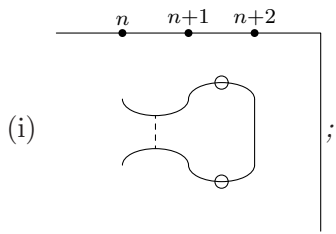
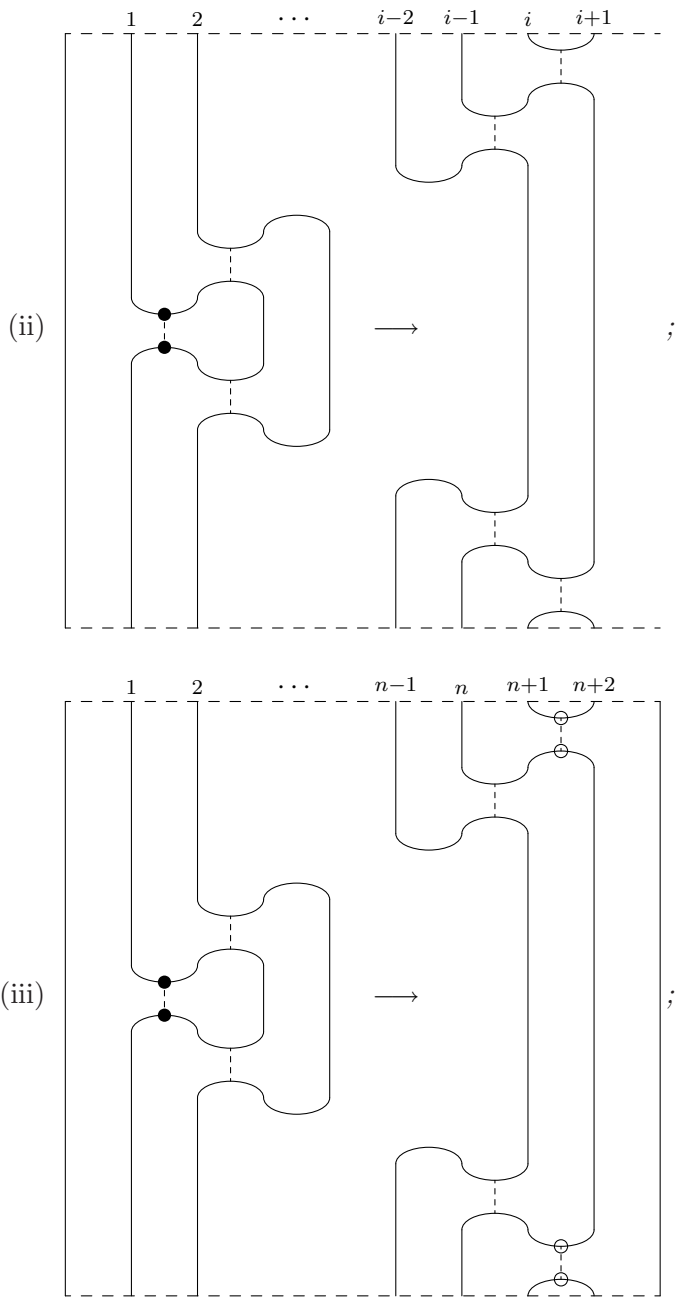
\begin{enumerate}[label=\rm{(\roman*)}]
\item \begin{tabular}[c]{@{}c@{}}
%-- New mfpic environment, number 177 of 189. (size of end split: 2, should be 2)  ------------------->
\includegraphics{ThesisFigs2.078}
\end{tabular};

\item \begin{tabular}[c]{@{}c@{}}
%-- New mfpic environment, number 178 of 189. (size of end split: 2, should be 2)  ------------------->
\includegraphics{ThesisFigs2.079}
\end{tabular};

\item \begin{tabular}[c]{@{}c@{}}
%-- New mfpic environment, number 179 of 189. (size of end split: 2, should be 2)  ------------------->
\includegraphics{ThesisFigs2.080}
\end{tabular};

\end{enumerate}
where the vertical dashed lines in the figure indicate that the two curves are part of the same generator $d_{j}$ and the arrows indicate a continuation of the pattern of the same shape.
\end{lemma}

\begin{proof}
An edge changing direction from right to left directly below node $i+1$ indicates that there are two consecutive occurrences of the simple diagram $d_{i}$ not having an occurrence of $d_{i+1}$ between them. (Note that this forces $i>1$.)  If $1<i<n+1$, then we must be in the situation of Lemma \ref{diagram.version.main.zigzag.lemma}, in which case we have the diagram in (ii).  If, on the other hand, $i=n+1$, then there are two possibilities.  One possibility is that there are two occurrences of $d_{n}$ occurring between the two occurrences of $d_{n+1}$.  (Note that there can only be at most two occurrences of $d_{n}$ occurring between the two consecutive occurrences of $d_{n+1}$; otherwise we contradict Lemma \ref{main_zigzag_lemma}.)  Then applying Lemma \ref{diagram.version.main.zigzag.lemma} to the two consecutive occurrences of $d_{n}$ forces us to have the diagram in (iii).  The second possibility is that there is a single occurrence of $d_{n}$ between the two consecutive occurrences of $d_{n+1}$.  This corresponds to the sequence $b_{n+1}b_{n}b_{n+1}$, which yields the diagram in (i).
\end{proof}

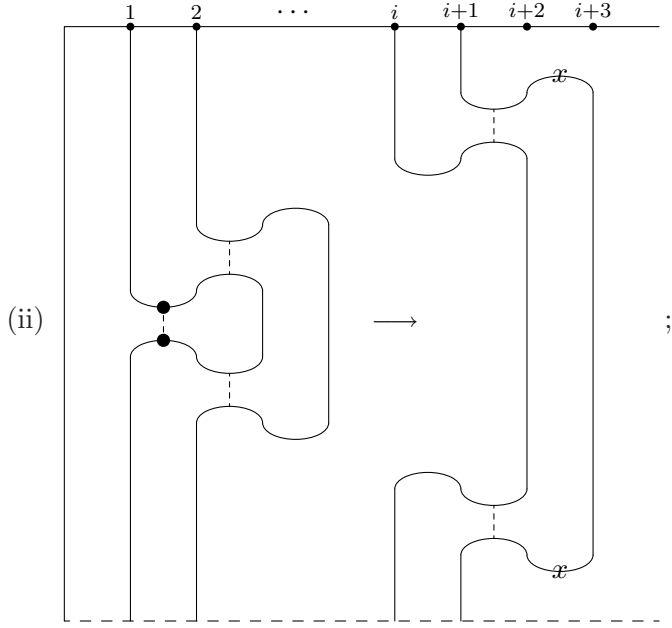
\begin{remark}\label{remark.direction.change}
If case (iii) from Lemma \ref{direction.change} occurs, then $w$ must be of type I by Lemma \ref{is_zigzag}.  
\end{remark}

\begin{lemma}\label{diagram.descent.set}
Let $w \in W_{c}(\C_{n})$.  Then $s_{i} \in \L(w)$ (respectively, $\R(w)$) if and only if there is a simple edge joining $i$ to $i+1$ (respectively, $i'$ to $(i+1)'$) in $d_{w}$.
\end{lemma}

\begin{proof}
First, assume that $s_{i} \in \L(w)$.  Then by Lemma \ref{simple.edge.in.N.face} there is a simple edge joining $i$ to $i+1$ in the north face, as desired.

\bigskip

For the converse, assume that there is a simple edge joining node $i$ to node $i+1$ in the north face of $d_{w}$.  We need to show that $s_{i} \in \L(w)$.  By Lemma \ref{a=1.implies.typeI}, if $\a(d)=1$, then $w$ is of type I.  In this case, $s_{i} \in \L(w)$ by Lemma \ref{image_zigzag}.  Now, assume that $\a(d)>1$.  Consider the simple representation for $d_{w}$ and let $e$ be the edge joining $i$ to $i+1$.  For sake of a contradiction, assume that $s_{i} \notin \L(w)$.  Then either 
\begin{enumerate}[label=\rm{(\alph*)}]
\item it is not the case that the end of $e$ leaving node $i+1$ encounters the northernmost occurrence of $d_{i}$ before any other generator;
\end{enumerate}
or
\begin{enumerate}[label=\rm{(\alph*)}, resume]
\item it is not the case that the end of $e$ leaving node $i$ encounters the northernmost occurrence of $d_{i}$ before any other generator.
\end{enumerate}
(We allow both (a) and (b) to occur.) Note that since $e$ crosses the line $x=i+1/2$, it must encounter $d_{i}$ at some stage.  We consider three distinct cases.

\bigskip

Case (1):  Assume that $i \notin \{1, 2, n, n+1\}$.  Then $e$ is undecorated.  We deal with case (a) from above; case (b) has a symmetric argument.  Since the curve must eventually encounter $d_{i}$, the edge $e$ must change direction from right to left.  Then we must be in one of the three situations of Lemma \ref{direction.change}.  But since we are assuming that $\a(d_{w})>1$, by Remark \ref{remark.direction.change}, there are only two possibilities for the edge leaving node $i+1$ in the simple representation for $d_{w}$:
\begin{enumerate}[label=\rm{(\roman*)}]
\item \begin{tabular}[c]{@{}c@{}}
%-- New mfpic environment, number 180 of 189. (size of end split: 2, should be 2)  ------------------->
\includegraphics{ThesisFigs2.081}
\end{tabular};

\item \begin{tabular}[c]{@{}c@{}}
%-- New mfpic environment, number 181 of 189. (size of end split: 2, should be 2)  ------------------->
\includegraphics{ThesisFigs2.082}
\end{tabular};
\end{enumerate}
where $x=\wcirc$ if $i=n-1$ and is trivial otherwise.  Certainly, we cannot have (i) since $e$ is undecorated and there is no sequence of relations that can completely remove decorations from a non-loop edge.  If (ii) occurs, then it must be the case that $s_{i+2}$ does not occur between two consecutive occurrences of $s_{i+1}$ in $w$.  Then by Lemma \ref{diagram.version.main.zigzag.lemma}, $d_{i}$ cannot occur again in $d_{w}$.  But this prevents the end of $e$ leaving node $i$ to join up with the other end leaving node $i+1$, which contradicts $d_{w}$ having a simple edge joining $i$ to $i+1$.

\bigskip

Case (2):  Assume that $i \in \{2, n\}$.  Then, as in case (1), $e$ is undecorated.  We assume that $i=2$; the case with $i=n$ follows symmetrically.  If (a) from above happens, then we can apply the arguments in case (1) and arrive at the same contradictions.  If (b) occurs, then the end of $e$ leaving node $i=2$ must immediately encounter $d_{1}$.  This adds a $\bcirc$ decoration to $e$, which is again a contradiction since there is no sequence of relations that can completely remove decorations from a non-loop edge.

\bigskip

Case (3):  Assume that $i \in \{1, n+1\}$.  We assume that $i=1$; the case with $i=n+1$ follows symmetrically.  Then $e$ is decorated precisely by a single $\bcirc$.  We must be in the situation described in (a) above.  Since the curve must eventually encounter $d_{1}$, the edge must change direction from right to left.   As in case (1), there are only two possibilities for the edge leaving node $i+1=2$ in the simple representation for $d_{w}$:
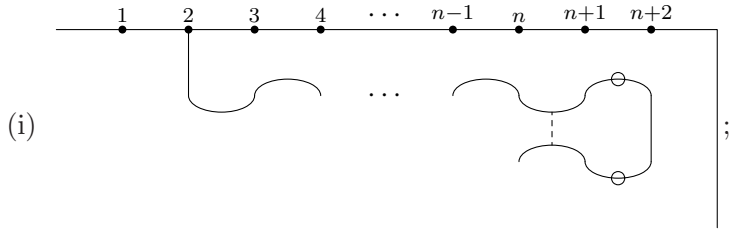
\begin{enumerate}[label=\rm{(\roman*)}]
\item \begin{tabular}[c]{@{}c@{}}
%-- New mfpic environment, number 182 of 189. (size of end split: 2, should be 2)  ------------------->
\includegraphics{ThesisFigs2.083}
\end{tabular};

\item \begin{tabular}[c]{@{}c@{}}
%-- New mfpic environment, number 183 of 189. (size of end split: 2, should be 2)  ------------------->
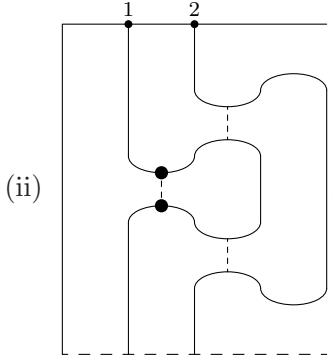
\includegraphics{ThesisFigs2.084}
\end{tabular}.
\end{enumerate}
Either way, we arrive at contradictions similar to those in case (1).

\bigskip

The proof that $s_{i} \in \R(w)$ if and only if there is a simple edge joining $i'$ to $(i+1)'$ is symmetric to the above.
\end{proof}

\end{section}

\begin{section}{Each monomial maps to a single admissible diagram}

\begin{proposition}\label{monomials.map.to.single.diagrams}
If $w \in W_{c}(\C_{n})$, then $\theta(b_{w})=d_{w}$.  That is, each monomial basis element maps to a single admissible diagram.
\end{proposition}

\begin{proof}
Let $w \in W_{c}(\C_{n})$.  By Lemma \ref{powers.of.2.and.delta.for.images.of.monomials}, we can write
	$$\theta(b_{w})=2^{k}\delta^{m}d_{w},$$
where $k,m \in \Z^{+}\cup\{0\}$.  We need to show that $m=0$ and $k=0$.  Since $w \in W_{c}(\C_{n})$, there exists a sequence (possibly trivial) of left and right weak star reductions that reduce $w$ to an irreducible element.  We induct on the number of steps in the sequence of weak star operations.  For the base case, assume that $w$ is irreducible.  Then by Lemmas \ref{image_zigzag} and \ref{image_wsrm}, $\theta(b_{w})=d_{w}$, which gives us our desired result.  For the inductive step, assume that $w$ is not irreducible.  If $w$ is of type I, then by Lemma \ref{image_zigzag}, $\theta(b_{w})=d_{w}$.  So, assume that $w$ is not of type I (i.e., $n(w)\neq 1$).  Without loss of generality, suppose that $w$ is left weak star reducible by $s$ with respect to $t$.  Choose $s$ and $t$ such that $\star^{L}_{s,t}(w)$ requires fewer steps to reduce to an irreducible element.  We can write: (1) $w=stv$ (reduced) if $m(s,t)=3$ or (2) $w=stsv$ (reduced) if $m(s,t)=4$.  We consider these two cases separately.

\bigskip

Case (1):  Assume that $m(s,t)=3$.  Without loss of generality, assume that $s=s_{i}$ and $t=s_{i+1}$ with $1<i<n$.  This implies that
%-- New mfpic environment, number 184 of 189. (size of end split: 2, should be 2)  ------------------->
\begin{align*}
\theta(b_{w})&=\theta(b_{i}b_{s_{i+1}v})\\
&=\theta(b_{i})\theta(b_{s_{i+1}v})\\
&=d_{i} d_{s_{i+1}v}\\
&=\begin{tabular}[c]{@{}c@{}}
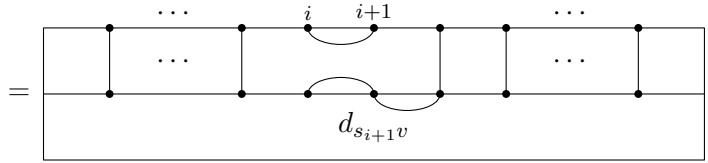
\includegraphics{ThesisFigs2.085}
\end{tabular},
\end{align*}
where we are applying the induction hypothesis to $\theta(b_{s_{i+1}v})$ ($w$ is left star reducible to $s_{i+1}v$ and requires fewer steps to reduce to an irreducible element) and we are using Lemma \ref{diagram.descent.set} to draw the bottom diagram in the last line.  By inspecting the product $d_{i}d_{s_{i+1}v}$, we see that there are no loops and no new relations to apply since $d_{s_{i+1}v}$ is admissible.  Therefore, $\theta(b_{w})=d_{w}$.

\bigskip

Case (2):  Assume that $m(s,t)=4$.  Without loss of generality, assume that $\{s,t\}=\{s_{1},s_{2}\}$; the case with $\{s,t\}=\{s_{n},s_{n+1}\}$ is symmetric.  Since $w=stsv$ (reduced) and $w$ is fully commutative, neither $s$ nor $t$ are in $\L(v)$. Also, since $w$ is left weak star reducible by $s$ with respect to $t$ to $tsv$, by induction, we have $\theta(b_{tsv})=d_{tsv}$.  By Lemma \ref{powers.of.2.and.delta.for.images.of.monomials}, there exists $k', m' \in \Z^{+}\cup\{0\}$ such that
	$$\theta(b_{v})=2^{k'}\delta^{m'}d_{v}.$$
But then
\begin{align*}
d_{tsv}&=\theta(b_{tsv})\\
&=\theta(b_{t}b_{s}b_{v})\\
&=\theta(b_{t})\theta(b_{s})\theta(b_{v})\\
&=d_{t}d_{s}2^{k'}\delta^{m'}d_{v}\\
&=2^{k'}\delta^{m'}d_{t}d_{s}d_{v}.
\end{align*}
Since the product in the last line is equal to the admissible diagram $d_{tsv}$, we must have $k'=0$ and $m'=0$.  That is, $\theta(b_{v})=d_{v}$, and a similar argument shows that $\theta(b_{sv})=d_{sv}$, as well.  Now, we consider two possible subcases: (a) $s=s_{1}$ and $t=s_{2}$; and (b) $s=s_{2}$ and $t=s_{1}$.  

\bigskip

(a):  Assume that $s=s_{1}$ and $t=s_{2}$.  We see that
%-- New mfpic environment, number 185 of 189. (size of end split: 2, should be 2)  ------------------->
\begin{align*}
2^{k}\delta^{m}d_{w}&=\theta(b_{w})\\
&=\theta(b_{1}b_{2}b_{1}b_{v})\\
&=\theta(b_{1})\theta(b_{2})\theta(b_{1})\theta(b_{v})\\
&=d_{1}d_{2}d_{1}d_{v}\\
&=\begin{tabular}[c]{@{}c@{}}
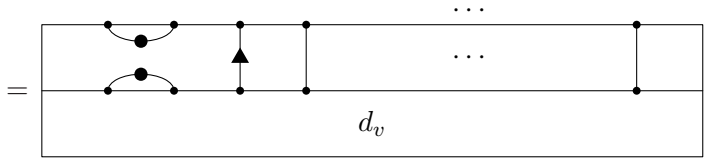
\includegraphics{ThesisFigs2.086}
\end{tabular}.
\end{align*}
Since $s_{1} \notin \L(v)$, by Lemma \ref{diagram.descent.set}, there cannot be a simple edge joining node 1 to node 2 in the north face of $d_{v}$.  This implies that there can be no loops in the product in the last line above, and so, $m=0$.  It also implies that the edge leaving node 3 of $d_{v}$ is not exposed to the west, and so it cannot be decorated with a closed symbol.  Since $s_{2} \notin \L(v)$, there cannot be a simple edge joining node 2 to node 3 in $d_{v}$.  This implies that in order for $d_{1}d_{2}d_{1}d_{v}$ to be equal to a power of 2 times $d_{w}$, the edge leaving node 1 in $d_{v}$ must be decorated with a closed decoration.  If the first  decoration on the edge leaving node 1 in $d_{v}$ is a $\bcirc$, then in order to produce a power of 2 in the product $d_{1}d_{2}d_{1}d_{v}$, we must have a simple edge between nodes 2 and 3, but we have already said that this cannot happen.  Suppose that the first decoration occurring on the edge leaving node 1 in $d_{v}$ is a $\btri$.  In this case, $d_{1}d_{v}=2 d'$, for some admissible diagram $d'$.  This contradicts 
	$$d_{1}d_{v}=\theta(b_{1})\theta(b_{v})=\theta(b_{sv})=d_{sv}.$$
Therefore, there can be no power of 2 in the product $d_{1}d_{2}d_{1}d_{v}$.  So, $k=0$, as desired.

\bigskip

(b): Now, assume that $s=s_{2}$ and $t=s_{1}$.  In this case, we see that
%-- New mfpic environment, number 186 of 189. (size of end split: 2, should be 2)  ------------------->
\begin{align*}
2^{k}\delta^{m}d_{w}&=\theta(b_{w})\\
&=\theta(b_{2}b_{1}b_{2}b_{v})\\
&=\theta(b_{2})\theta(b_{1})\theta(b_{2})\theta(b_{v})\\
&=d_{2}d_{1}d_{2}d_{v}\\
&=\begin{tabular}[c]{@{}c@{}}
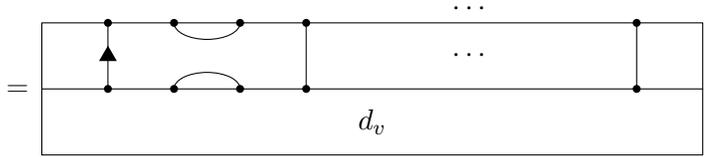
\includegraphics{ThesisFigs2.087}
\end{tabular}.
\end{align*}
Since $s_{2} \notin \L(v)$, by Lemma \ref{diagram.descent.set}, there cannot be a simple edge joining node 2 to node 3 in the north face of $d_{v}$.  This implies that there can be no loops in the product in the last line above, and so, $m=0$.  In order for $d_{2}d_{1}d_{2}d_{v}$ to be equal to a power of 2 times $d_{w}$, the edge leaving node 1 in $d_{v}$ must be decorated with a closed decoration.  For sake of a contradiction, assume that this is the case.  This implies that if $d_{v}$ is written as a product of simple diagrams, there must be at least one occurrence of $d_{1}$ (this is the only way we can acquire closed decorations).  Then we must have $s_{1} \in \supp(v)$, which implies that $tsv=s_{1}s_{2}v$ contains at least two occurrences of $s_{1}$.  Consider the top two occurrences of $s_{1}$ in the canonical representation of $H(tsv)=H(s_{1}s_{2}v)$.  Since $w=s_{2}s_{1}s_{2}v$ (reduced) and $w$ is fully commutative, there must be an entry in $H(v)$ labeled by $s_{2}$ that covers the highest occurrence of $s_{1}$.  By a right-handed version of Lemma \ref{main_zigzag_lemma}, we must have $\z_{1,1}^{R,1}$ as the subword of some reduced expression for $w$.  But by Lemma \ref{is_zigzag}, $w$ must be of type I.  This contradicts our earlier assumption that $w$ is not of type I.  Therefore, the edge leaving node 1 in $d_{v}$ does not carry any closed decorations.  So, there can be no power of 2 in the product $d_{2}d_{1}d_{2}d_{v}$, and hence $k=0$, as desired.
\end{proof}

\end{section}

\begin{section}{Proof of main result}

The next lemma will be useful for simplifying the argument in the proof of our main result (Theorem \ref{main.result}).

\begin{lemma}\label{nugget}
Let $w, w' \in W_{c}(\C_{n})$ such that $d_{w}=d_{w'}$.  Suppose that $w'$ is left weak star reducible by $s$ with respect to $t$.  Then
	$$b_{t}b_{w}=\begin{cases}
b_{w''}, & \text{if } m(s,t)=3, \\
2b_{w''}, & \text{if } m(s,t)=4,
\end{cases}$$
for some $w'' \in W_{c}(\C_{n})$.
\end{lemma}

\begin{proof}
Since $w'$ is left weak star reducible by $s$ with respect to $t$, we can write
$$w'=\begin{cases}
stv', & \text{if } m(s,t)=3, \\
stsv', & \text{if } m(s,t)=4,
\end{cases}$$
where each product is reduced.  By Remark \ref{monomial.weak.star.reductions}, this implies that
$$b_{t}b_{w'}=\begin{cases}
b_{tv'}, & \text{if } m(s,t)=3, \\
2b_{tsv'}, & \text{if } m(s,t)=4.
\end{cases}$$
By Proposition \ref{monomials.map.to.single.diagrams}, we have
\begin{align*}
\theta(b_{t}b_{w})&=d_{t}d_{w}\\
&=d_{t}d_{w'}\\
&=\theta(b_{t}b_{w'})\\
&=\begin{cases}
\theta(b_{tv'}), & \text{if } m(s,t)=3, \\
2\theta(b_{tsv,}), & \text{if } m(s,t)=4,
\end{cases}\\
&=\begin{cases}
d_{tv'}, & \text{if } m(s,t)=3, \\
2d_{tsv'}, & \text{if } m(s,t)=4.
\end{cases}
\end{align*}
Then again by Proposition \ref{monomials.map.to.single.diagrams}, there must exist $w'' \in W_{c}(\C_{n})$ such that 
$$b_{t}b_{w}=\begin{cases}
b_{w''}, & \text{if } m(s,t)=3, \\
2b_{w''}, & \text{if } m(s,t)=4,
\end{cases}$$
as desired.
\end{proof}

\begin{remark}\label{remark.nugget}
Lemma \ref{nugget} has an analogous statement involving right weak star reductions.
\end{remark}

Finally, we state our main result.

\begin{theorem}\label{main.result}
The map $\theta$ given in Proposition \ref{surjective.homomorphism} is an isomorphism of $\TL(\C_{n})$ and $\D_{n}$.  Moreover, each admissible diagram corresponds to a unique monomial basis element.
\end{theorem}

\begin{proof}
According to Proposition \ref{surjective.homomorphism}, $\theta$ is a surjective homomorphism.  Also, by Proposition \ref{monomials.map.to.single.diagrams}, the image of each monomial basis element is a single admissible diagram.  It remains to show that $\theta$ is injective.  For sake of a contradiction, assume that $\theta$ is not injective.  Then there exist $w, w' \in W_{c}(\C_{n})$ with $w \neq w'$ such that
	$$\theta(b_{w})=\theta(b_{w'}),$$
so that
	$$d_{w}=d_{w'}.$$
By Lemma \ref{diagram.descent.set}, $\L(w)=\L(w')$ and $\R(w)=\R(w')$.  

\bigskip

If either of $w$ or $w'$ are of type I, then according to Lemma \ref{image_zigzag}, $\a(d_{w})=\a(d_{w'})=1$.  In this case, both $w$ and $w'$ are of type I by Lemma \ref{a=1.implies.typeI}.  But then by Lemma \ref{image_zigzag}, we must have $w=w'$.  So, neither of $w$ or $w'$ are of type I.  

\bigskip

Now, we will argue that we may simplify the argument and assume that at least one of $w$ or $w'$ is irreducible.  Suppose neither of $w$ or $w'$ is irreducible.  Then there exist sequences of left and right weak star reductions $\star_{s'_{1},t'_{1}}^{L}, \dots, \star_{s'_{l},t'_{l}}^{L}$ and $\star_{s''_{1},t''_{1}}^{R}, \dots, \star_{s''_{l'},t''_{l'}}^{R}$, respectively, that reduce $w'$ to an irreducible element, say $w''$.  Then
\begin{equation}\label{eqn1}
	b_{t'_{l}}\cdots b_{t'_{1}}b_{w'}b_{t''_{1}}\cdots b_{t''_{l'}}=2^{k}b_{w''},
\end{equation}
where $k\geq 0$.  By Lemma \ref{nugget} and Remark \ref{remark.nugget}, it follows that
\begin{equation}\label{eqn2}
b_{t'_{l}}\cdots b_{t'_{1}}b_{w}b_{t''_{1}}\cdots b_{t''_{l'}}=2^{k}b_{w'''},
\end{equation}
for some $w''' \in W_{c}(\C_{n})$.  Since $\theta(b_{w'})=\theta(b_{w})$, by applying $\theta$ to equations \ref{eqn1} and \ref{eqn2}, we can  conclude that $\theta(b_{w''})=\theta(b_{w'''})$, where $w''$ is irreducible.  By making repeated applications of Lemma \ref{weak.star.reverse}, we see that there exists $k' \geq 0$ such that
	$$b_{s'_{1}}\cdots b_{s'_{l}}b_{t'_{l}}\cdots b_{t'_{1}}b_{w'}b_{t''_{1}}\cdots b_{t''_{l'}}b_{s''_{l'}}\cdots b_{s''_{1}}=2^{k'}b_{w'},$$
which implies that
	$$b_{s'_{1}}\cdots b_{s'_{l}}b_{t'_{l}}\cdots b_{t'_{1}}b_{w}b_{t''_{1}}\cdots b_{t''_{l'}}b_{s''_{l}}\cdots b_{s''_{1}}=2^{k'}b_{w}.$$
That is, we can reverse the the sequences that reduced $b_{w'}$ (respectively, $b_{w}$) to a multiple of $b_{w''}$ (respectively, $b_{w'''}$).  This shows that we may simplify the argument and assume that at least one of $w$ or $w'$ is irreducible.  

\bigskip

Without loss of generality, assume that $w$ is irreducible.  If $w'$ is also irreducible, then we must have $w=w'$ since monomials indexed by distinct irreducible elements map to distinct diagrams (see Lemmas \ref{image_zigzag} and \ref{image_wsrm}).  So, $w'$ is not irreducible.   Without loss of generality, suppose that $w'$ is left weak star reducible by $s$ with respect to $t$.  Then we may write 
$$w'=\begin{cases}
stv', & \text{if } m(s,t)=3, \\
stsv', & \text{if } m(s,t)=4,
\end{cases}$$
where each product is reduced.  By Remark \ref{monomial.weak.star.reductions}, this implies that
$$b_{t}b_{w'}=\begin{cases}
b_{tv'}, & \text{if } m(s,t)=3, \\
2b_{tsv'}, & \text{if } m(s,t)=4.
\end{cases}$$
Note that since $\L(w)=\L(w')$ and $s \in \L(w')$, we have $s \in \L(w)$.  Then since $w$ is irreducible, $tw$ is reduced and fully commutative.  This implies that
	$$b_{t}b_{w}=b_{tw}.$$
This shows that $m(s,t) \neq 4$; otherwise, we contradict Lemma \ref{nugget}.  So, we must have $m(s,t)=3$.  

\bigskip

Without loss of generality, assume that $s=s_{i+1}$ and $t=s_{i}$ with $2<i<n+1$, so that $w'=s_{i+1}s_{i}v'$ (reduced).  By Lemma \ref{diagram.descent.set}, $d_{w'}$, and hence $d_{w}$, has a simple edge joining node $i+1$ to node $i+2$.  For sake of contradiction, assume that $d_{w'}$, and hence $d_{w}$, has a simple edge joining node $i-1$ to node $i$.  Then by Lemma \ref{diagram.descent.set}, $s_{i-1} \in \L(w')=\L(w)$, which contradicts $w'=s_{i+1}s_{i}v'$.  So, there cannot be a simple edge joining node $i-1$ to node $i$, which implies that $s_{i-1} \notin \L(w')=\L(w)$.  Since $s_{i+1} \in \L(w')=\L(w)$ while $s_{i-1} \notin \L(w')=\L(w)$, $w$ cannot be of type II.  Since $w$ is irreducible, but not of type I or II, it follows from Theorem \ref{affineCwsrm} that $w$ can be written as a product of a type $B$ irreducible element times a type $B'$ irreducible element.  This implies that $w$ contains a single occurrence of $s_{i+1}$ and no occurrences of $s_{i}$.  Then $d_{w}$, and hence $d_{w'}$, can be drawn so that no edges intersect the line $x=i+1/2$.  Furthermore, there are no closed (respectively, open) decorations occurring to the right (respectively, left) of the line $x=i+1/2$.  However, we see that
\begin{align*}
%-- New mfpic environment, number 187 of 189. (size of end split: 2, should be 2)  ------------------->
d_{w'}& = d_{i+1}d_{i}d_{v'} \\
& = \begin{tabular}[c]{@{}c@{}}
\includegraphics{ThesisFigs2.088}
\end{tabular}.
\end{align*}
This implies that the edge leaving node $i$ in the simple representation of $d_{w'}$ must change direction from right to left.  By Lemma \ref{direction.change} and Remark \ref{remark.direction.change}, the  simple representation of $d_{w'}$ must be vertically equivalent to one of the following diagrams:
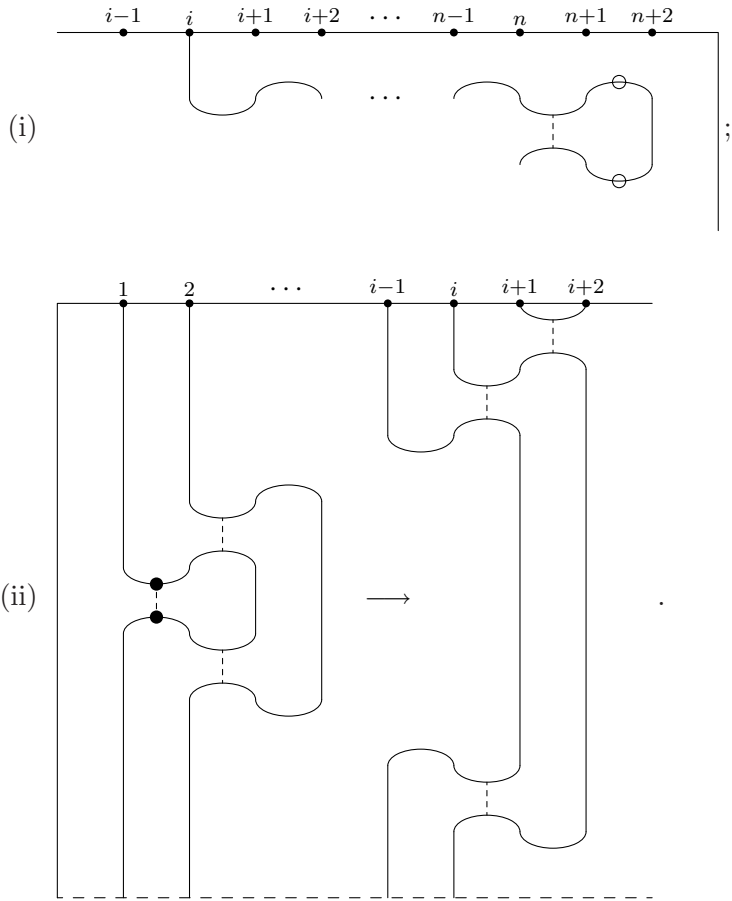
\begin{enumerate}[label=\rm{(\roman*)}]
%-- New mfpic environment, number 188 of 189. (size of end split: 2, should be 2)  ------------------->
\item \begin{tabular}[c]{@{}c@{}}
\includegraphics{ThesisFigs2.089}
\end{tabular};
%-- New mfpic environment, number 189 of 189. (size of end split: 2, should be 2)  ------------------->
\item \begin{tabular}[c]{@{}c@{}}
\includegraphics{ThesisFigs2.090}
\end{tabular}.
\end{enumerate}
We cannot have the diagram in (i) since then we would have open decorations occurring to the left of $x=i+1/2$.  So, we must have the diagram in (ii).  But then we are in the situation of Lemma \ref{diagram.version.main.zigzag.lemma}.  Since $w$ is not of type I, there are no other occurrences of the generators $d_{1}, \dots, d_{i}$ in the simple representation of $d_{w'}$.  This implies that $d_{w}$ has a propagating edge connecting node $1$ to node $1'$ that is labeled by a single $\btri$.  By inspecting the images of monomials indexed by non-type I and non-type II irreducible elements, we see that none of them have this configuration.  Therefore, we have a contradiction, and hence $\theta$ is injective, as desired.
\end{proof}

\end{section}

\end{chapter}

\bibliographystyle{plain}

\bibliography{ErnstPhDThesis}

\end{document}